\documentclass[a4paper,12pt]{article}

\usepackage{palatino}
\usepackage{url}

\usepackage{amssymb}
\usepackage{latexsym}
\usepackage{amsfonts}
\usepackage{amsmath}
\usepackage{amsthm}
\usepackage{hyperref}
\usepackage{tensor}
\usepackage{mathtools}
\usepackage{mathrsfs}
\usepackage{mdwlist}
\usepackage{pgf,tikz,pgfplots}
\pgfplotsset{width=7.5cm,compat=1.9}

\usetikzlibrary{arrows,patterns,positioning}
\usetikzlibrary{calc}

\usepackage[margin=1in]{geometry}

\usepackage[compact]{titlesec}

\usepackage[scale=0.9]{tgheros}

\usepackage[T1]{fontenc}

\newcommand\FanoPlane[1][1cm]{%
\begin{tikzpicture}[
mydot/.style={
  draw,
  circle,
  fill=black,
  inner sep=1.5pt}
]
\draw
  (0,0) coordinate[label=below left:$\Delta_1$] (A) --
  (#1,0) coordinate[label=below right:$\Delta_3$] (B) --
  ($ (A)!.5!(B) ! {sin(60)*2} ! 90:(B) $) coordinate[label=above:$\ell$] (C) -- cycle;
\coordinate[label={[label distance=#1*0.05]right:$\Phi_2$}] (O) at
  (barycentric cs:A=1,B=1,C=1);
\draw (O) circle [radius=#1*1.717/6];
\draw (C) -- ($ (A)!.5!(B) $) coordinate[label=below:$\Delta_2$] (LC);
\draw (A) -- ($ (B)!.5!(C) $) coordinate[label=right:$\Phi_3$] (LA);
\draw (B) -- ($ (C)!.5!(A) $) coordinate[label=left:$\Phi_1$] (LB);
\foreach \Nodo in {A,B,C,O,LC,LA,LB}
  \node[mydot] at (\Nodo) {};
\end{tikzpicture}%
}

\DeclareMathOperator{\Aut}{Aut}

\DeclareMathOperator{\PSL}{PSL}

\DeclareMathOperator{\PSiL}{P\Sigma L}
\DeclareMathOperator{\PGL}{PGL}
\DeclareMathOperator{\PSU}{PSU}
\DeclareMathOperator{\PGU}{PGU}
\DeclareMathOperator{\pg}{PG}
\DeclareMathOperator{\PGaU}{P\Gamma U}

\DeclareMathOperator{\Sp}{Sp}

\DeclareMathOperator{\Co}{Co}

\DeclareMathOperator{\Ree}{Ree}
\DeclareMathOperator{\HS}{HS}
\DeclareMathOperator{\PGaL}{P\Gamma L}
\DeclareMathOperator{\AGaL}{A\Gamma L}

\DeclareMathOperator{\GL}{GL}

\DeclareMathOperator{\Sym}{Sym}
\DeclareMathOperator{\soc}{soc}

\DeclareMathOperator{\supp}{supp}

\DeclareMathOperator{\alt}{A}
\DeclareMathOperator{\s}{S}
\DeclareMathOperator{\mg}{M}

\DeclareMathOperator{\wt}{wt}
\DeclareMathOperator{\diff}{diff}

\DeclareMathOperator{\PermAut}{PermAut}

\renewcommand{\b}{\mathbf}
\renewcommand{\leq}{\leqslant}
\renewcommand{\geq}{\geqslant}
\newcommand{\ra}{\rightarrow}

\newcommand{\F}{\mathbb F}

\newcommand{\G}{\mathcal{G}}
\newcommand{\D}{\mathcal D}
\newcommand{\B}{\mathcal B}

\renewcommand{\L}{\mathcal L}
\renewcommand{\H}{\mathcal H}
\renewcommand{\P}{\mathcal P}
\renewcommand{\S}{\mathcal S}

\newcommand{\RM}{\mathcal{RM}}
\newcommand{\NR}{\mathcal{NR}}
\newcommand{\PH}{\mathcal{PH}}
\newcommand{\EG}{\mathcal{E}}
\newcommand{\SG}{\mathcal{S}}

\newcommand{\E}{\mathcal E}
\newcommand{\Z}{\mathbb Z}

\theoremstyle{plain}
\newtheorem{lemma}{Lemma}
\newtheorem{theorem}[lemma]{Theorem}
\newtheorem{proposition}[lemma]{Proposition}

\theoremstyle{definition}
\newtheorem{definition}[lemma]{Definition}
\newtheorem{example}[lemma]{Example}

\newtheorem{remark}[lemma]{Remark}
\newtheorem{hypothesis}{Hypothesis}

\numberwithin{equation}{section}
\numberwithin{lemma}{section}

\begin{document}


\title{On the Classification of Binary Completely Transitive Codes with Almost-Simple Top-Group
}

\author{
 Robert F. Bailey$^1$\footnote{Supported by an NSERC Discovery Grant.}
 \and
 Daniel R. Hawtin$^2$\footnote{Supported by the Croatian Science Foundation under the project 6732.}
}

\date{
 \small{
  \emph{
   $^1$School of Science and the Environment (Mathematics),\\
   Memorial University of Newfoundland, Grenfell Campus,\\ 
   Corner Brook, NL, A2H 6P9, Canada.\\
  }
  \href{mailto:rbailey@grenfell.mun.ca}{rbailey@grenfell.mun.ca}\\
  \vspace{0.25cm}
 }
 \small{
  \emph{
   $^2$Department of Mathematics, University of Rijeka\\
   Rijeka, Croatia, 51000\\
  }
  \href{mailto:dan.hawtin@gmail.com}{dan.hawtin@gmail.com}\\
  \vspace{0.25cm}
 }
 \today
}

\maketitle

\begin{abstract}
 A code $C$ in the Hamming metric, that is, is a subset of the vertex set $V\varGamma$ of the Hamming graph $\varGamma=H(m,q)$, gives rise to a natural \emph{distance partition} $\{C,C_1,\ldots,C_\rho\}$, where $\rho$ is the covering radius of $C$. Such a code $C$ is called \emph{completely transitive} if the automorphism group $\Aut(C)$ acts transitively on each of the sets $C$, $C_1$, \ldots, $C_\rho$. A code $C$ is called $2$-neighbour-transitive if $\rho\geq 2$ and $\Aut(C)$ acts transitively on each of $C$, $C_1$ and $C_2$.
 
 Let $C$ be a completely transitive code in a binary ($q=2$) Hamming graph having full automorphism group $\Aut(C)$ and minimum distance $\delta\geq 5$. Then it is known that $\Aut(C)$ induces a $2$-homogeneous action on the coordinates of the vertices of the Hamming graph. The main result of this paper classifies those $C$ for which this induced $2$-homogeneous action is \emph{not} an affine, linear or symplectic group. We find that there are $13$ such codes, $4$ of which are non-linear codes. Though most of the codes are well-known, we obtain several new results. First, a non-linear completely transitive code that does not explicitly appear in the existing literature is constructed, as well as a related non-linear code that is $2$-neighbour-transitive but not completely transitive. Moreover, new proofs of the complete transitivity of several codes are given. Additionally, we answer the question of the existence of distance-regular graphs related to the completely transitive codes appearing in our main result.
\end{abstract}

\section{Introduction}

A subset $C$ of the vertex set $V\varGamma$ of the Hamming graph $\varGamma=H(m,q)$ is called a \emph{code}, the elements of $C$ are called \emph{codewords}, and the subset $C_i$ of $V\varGamma$ consisting of all vertices of $H(m,q)$ having nearest codeword at Hamming distance $i$ is called the \emph{set of $i$-neighbours} of $C$. Some important generalisations of perfect codes were introduced in the 1970s: \emph{uniformly packed codes}, introduced by Semakov, Zinoviev and Zaitsev \cite{SemZinZai71}, and the classes of \emph{completely regular} and \emph{$s$-regular} codes, defined by Delsarte \cite{delsarte1973algebraic}. Uniformly packed codes were shown in \cite{SemZinZai71} to be completely regular in the sense of \cite{delsarte1973algebraic}, in particular, before the formal definition of completely regular was introduced.  

The parameters of perfect codes over prime-power alphabets have been classified, and codes satisfying these parameters found; see \cite{tietavainen1973nonexistence} or \cite{Zinoviev73thenonexistence}. In contrast, the class of completely regular codes is vast, with similar classification results remaining an active area of research. Several recent results have been obtained by Borges et al.~\cite{borgesrho1,borges2012new,Borges201468}. For a survey of results on completely regular codes see \cite{borges2019completely}.

Completely regular and $s$-regular codes are defined in terms of regularity conditions on the \emph{distance partition} $\{C,C_1,\ldots, C_\rho\}$ of a code $C$, where $\rho$ is the \emph{covering radius}. The focus of current paper is the algebraic analogues, defined directly below, of the classes of completely regular and $s$-regular codes. Note that the group $\Aut(C)$ is the setwise stabiliser of $C$ in the full automorphism group of $H(m,q)$. In particular, if $C$ is linear then we consider the group of translations by codewords of $C$ to be contained in $\Aut(C)$.

\begin{definition}\label{sneighbourtransdef}
 Let $C$ be a code in $H(m,q)$ with covering radius $\rho$, let $s\in\{1,\ldots,\rho\}$, and $X\leq\Aut(C)$. Then $C$ is said to be:
 \begin{enumerate}
  \item \emph{$(X,s)$-neighbour-transitive} if $X$ acts transitively on each of the sets $C,C_1,\ldots, C_s$.
  \item \emph{$X$-neighbour-transitive} if $C$ is $(X,1)$-neighbour-transitive.
  \item \emph{$X$-completely transitive} if $C$ is $(X,\rho)$-neighbour-transitive.
  \item \emph{$s$-neighbour-transitive} if $C$ is $(\Aut(C),s)$-neighbour-transitive, \emph{neighbour-transitive} if $C$ is $\Aut(C)$-neighbour-transitive, or \emph{completely transitive} if $C$ is $\Aut(C)$-\emph{completely transitive}.
 \end{enumerate}
\end{definition}

A variant of the above concept of complete transitivity was introduced for linear codes by Sol{\'e} \cite{sole1990completely}, with the above definition first appearing in \cite{Giudici1999647}. Completely transitive codes form a subfamily of completely regular codes, and $s$-neighbour transitive codes are a sub-family of $s$-regular codes, for each $s$. Characterisation of certain families of neighbour-transitive codes in Hamming graphs have been achieved (see \cite{ntrcodes,gillespieCharNT,gillespiediadntc}), as well as several classification results of subfamilies of neighbour-transitive codes in Johnson graphs (see \cite{durante2014sets,marksphd,liebler2014neighbour,neunhoffer2014sporadic}). The results in this paper depend upon the series of papers \cite{ef2nt,aas2nt,ondimblock,minimal2nt} on $2$-neighbour-transitive codes, which form part of the second author's PhD thesis.

Each vertex of $H(m,q)$ is of the form $\alpha=(\alpha_1,\ldots,\alpha_m)$, where the entries $\alpha_i$ come from an \emph{alphabet} $Q$ of size $q$. A typical automorphism of $H(m,q)$ is a composition of two automorphisms, each of a special type. An automorphism of the first type corresponds to an $m$-tuple $(h_1,\dots,h_m)$ of permutations of $Q$ and maps a vertex $\alpha$ to $(\alpha_1^{h_1},\ldots,\alpha_m^{h_m})$. An automorphism of the second type corresponds to a permutation of the set $M=\{1,\dots,m\}$ of subscripts and simply permutes the entries of vertices, for example the map corresponding to the permutation $(123)$ of $M$ maps $\alpha= (\alpha_1, \alpha_2, \alpha_3,\alpha_4, \dots,\alpha_m)$ to $(\alpha_3, \alpha_1, \alpha_2,\alpha_4, \dots,\alpha_m)$. (More details are given in Section 2.1.)

Our main result is Theorem~\ref{binaryCTclass}, further below, the proof of which relies on \cite[Theorem~1.2]{minimal2nt}, stated below as Theorem~\ref{binaryx2ntchar}. For now we give a brief discussion of the relation between the two theorems, only what is required to properly state Theorem~\ref{binaryCTclass}. If a completely transitive code $C$ has minimum distance at least $5$ then, since this implies $C_2$ is non-empty, we have that $C$ is by definition $2$-neighbour-transitive. Hence, Theorem~\ref{binaryx2ntchar} applies. Recall that the \emph{socle} $\soc(G)$ of a group $G$ is the product of all its minimal normal subgroups. In particular, Theorem~\ref{binaryx2ntchar} gives all possibilities for the socle of the action on $M$ of the stabiliser $\Aut(C)_{\b 0}$ of the codeword ${\b 0}$ (assumed to be in $C$) for any $2$-neighbour-transitive code $C$ with minimum distance at least $5$. If $|C|>2$ then this socle is either one of the Mathieu groups $\mg_{11}$ or $\mg_{12}$, as in Theorem~\ref{binaryx2ntchar} part 2, or contained in \cite[Table~1]{minimal2nt}, as in Theorem~\ref{binaryx2ntchar} part 3. Table~\ref{binarytable} lists the groups we consider here (see the following remark) for the socle of the action of $\Aut(C)_{\b 0}$ on $M$.

\begin{table}
 \begin{center}
 \begin{tabular}{ccc}
  $G$ & $m$ & conditions\\
  \hline
  $\PSL_3(4)$ & $21$ & - \\
  
  $\alt_7$ & $15$ & - \\
  
  $\PSL_2(r)$ & $r+1$ & $23\leq r\equiv\pm 1\pmod 8$ \\
  
  
  $\PSU_3(r)$ & $r^3+1$ & $r$ is odd \\
  
  $\Ree(r)$ & $r^3+1$ & $r\geq 3$ \\
  
  $\mg_{m}$ & $11,12,22,23,24$ & - \\
  
  $\HS$ & $176$ & - \\
  
  $\Co_3$ & $276$ & - \\
  \hline
 \end{tabular}
 \caption{Groups $G$ for which there exists a non-trivial binary $2$-neighbour-transitive code $C$ in $H(m,2)$ with $G\cong\soc(\Aut(C)_{\b 0}^M)$ and $G$ is not isomorphic to i) $\PSL_t(2^k)$ unless $(t,k)=(3,2)$, ii) $\Sp_{2t}(2)$ for $t\geq 3$, or iii) $\Z_2^t$ for $t\geq 3$. }
 \label{binarytable}
 \end{center}
\end{table}

\begin{remark}
 The groups appearing in \cite[Table~1]{minimal2nt} but not in Table~\ref{binarytable}, and hence not covered in this paper are: $\PSL_t(2^k)$ with $(t,k)\neq(3,2)$; $\Sp_{2t}(2)$ for $t\geq 3$; and $\Z_2^t$ for $t\geq 3$. This is due to the fact that the first two cases require more technical representation theory than those treated in this paper, while the third case requires the classification of finite transitive linear groups to be applied. However, Table~\ref{binarytable} does contain $\PSL_3(4)$. In this case the corresponding codes are related to the geometry of $\pg_2(4)$, which also arises in the treatment of the larger Mathieu groups; see Examples~\ref{pslcodesex} and~\ref{codesMathgps}, and \cite[Sections~6.5--6.7]{dixon1996permutation}. We believe that, given significant effort, the cases not treated in this paper can be resolved. We plan to work on this in the future, but leave this as a separate open question due to the reasons given above. Moreover, we do expect codes to arise in these cases. For a known example, take $C$ to be the extended Nordstrom--Robinson code with parameters $(16,256,6;4)$. In this case $C$ is completely transitive and the socle of $\Aut(C)_{\b 0}^M$ is $\Z_2^4$.
\end{remark}

We can now state Theorem~\ref{binaryCTclass}. Note: we say that a code $C$ is \emph{non-trivial} if $|C|>2$; if $C$ is a linear code of length $m$, dimension $k$, minimum distance $\delta$ and covering radius $\rho$ then we say it has \emph{parameters} $[m,k,\delta;\rho]$; if $C$ is a non-linear code of length $m$, minimum distance $\delta$ and covering radius $\rho$ then we say it has \emph{parameters} $(m,|C|,\delta;\rho)$.

\begin{theorem}\label{binaryCTclass}
 Let $C$ be a non-trivial binary completely transitive code in $H(m,2)$ with minimum distance $\delta\geq 5$ and suppose the socle of the action of $\Aut(C)_{\b 0}$ on $M$ is isomorphic to $G$ as in one of the lines of Table~\ref{binarytable}. Then $C$ is equivalent to one of the codes given in Table~\ref{binaryCTtable}. Conversely, each code in Table~\ref{binaryCTtable} is completely transitive.
\end{theorem}

\begin{table}
 \begin{center}
 \begin{tabular}{rccccc}
  line & $C$ & $\Aut(C)$ & parameters & dist.~reg.~graph \\
  \hline  
  1 & $\H$ & $2\mg_{12}$ & $(12,24,6;3)$ & no \\
  
  2 & $\PH$ & $2\rtimes \mg_{11}$ & $(11,24,5;3)$ & no \\
  
  3 & $\NR$ & $2^5\rtimes\alt_8$ & $(15,256,5;3)$ & see Section~\ref{distreggraphsect} \\
  
  4 & $\P^\perp$ & $T_C\rtimes \PGaL_3(4)$ & $[21,12,5;3]$ & \cite[Theorem~11.3.6]{brouwer} \\
  
  5 & $\langle\L,\Delta_1\rangle$ & $T_C\rtimes \PSiL_3(4)$ & $[21,11,5;6]$ & \cite[Section~11.3.H b)]{brouwer} \\
  
  6 & $\langle\L,\Delta_1\rangle\cup\langle\L,\Delta_2\rangle$ & $T_\L\rtimes \PGaL_3(4)$ & $(21,2^{10}\cdot 3,5;6)$ & no \\
  
  7 & $\L$ & $T_C\rtimes \PGaL_3(4)$ & $[21,10,5;6]$ & \cite[Section~11.3.H c)]{brouwer} \\
  
  8 & $\G_{24}$ & $T_C\rtimes \mg_{24}$ & $[24,12,8;4]$ & \cite[Theorem~11.3.2]{brouwer} \\
  
  9 & $\G_{23}$ & $T_C\rtimes \mg_{23}$ & $[23,12,7;3]$ & \cite[Theorem~11.3.4]{brouwer} \\
  
  10 & $\G_{23}^\perp$ & $T_C\rtimes \mg_{23}$ & $[23,11,8;7]$ & \cite[Section~11.3.E]{brouwer} \\
  
  11 & $\G_{22}$ & $T_C\rtimes (\mg_{22}:2)$ & $[22,12,6;3]$ & \cite[Theorem~11.3.5]{brouwer} \\
  
  12 & $\EG_{22}$ & $T_C\rtimes (\mg_{22}:2)$ & $[22,11,6;7]$ & \cite[Section~11.3.F]{brouwer} \\
  
  13 & $\SG_{22}$ & $T_C\rtimes \mg_{22}$ & $[22,11,7;6]$ & \cite[Section~11.3.H a)]{brouwer} \\
  
  \hline
 \end{tabular}
 \caption{Non-trivial binary completely transitive codes $C$ with minimum distance at least $5$ and where $\Aut(C)_{\b 0}^M$ has as socle one of the groups $G$ contained in Table~\ref{binarytable}. See Section~\ref{sect:codes} for the definitions of these codes. The last column gives a reference for a distance-regular graph related to $C$, or contains `no' if no such graph can be constructed in the usual manner; see Section~\ref{distreggraphsect} for more details.}
 \label{binaryCTtable}
 \end{center}
\end{table}

While this paper is concerned with classifying certain subfamilies of binary completely transitive codes with minimum distance at least $5$, many completely transitive codes having minimum distance less then $5$ are known. For instance, \cite{Borges201468} exhibits completely transitive codes with minimum distances $3$ and $4$, whilst \cite{borges2010q} constructs completely transitive codes with minimum distances $3$ and below. For more examples of such codes, see \cite{borges2019completely}. It is worth mentioning that Borges et al.~\cite{borges2001nonexistence} proved that any binary linear completely transitive code has error correction capacity $e=\lfloor (\delta-1)/2\rfloor\leq 3$, and all codes arising in Theorem~\ref{binaryCTclass} also satisfy $e\leq 3$.

As part of our classification we also construct two non-linear codes. Using the notation defined in Example~\ref{pslcodesex}, these are the $2$-neighbour-transitive code $\langle\P,\Delta_1\rangle\cup\langle\P,\Delta_2\rangle$ with parameters $(21,2^9\cdot 3,6;6)$ (see Lemma~\ref{ptimes32nt}) and the completely transitive code $\langle\L,\Delta_1\rangle\cup\langle\L,\Delta_2\rangle$ with parameters $(21,2^{10}\cdot 3,5;6)$ (see Lemma~\ref{nonlinearpsl34}). Note that although not explicitly appearing previously in the literature the code $\langle\L,\Delta_1\rangle\cup\langle\L,\Delta_2\rangle$ is implicitly known to be completely transitive by applying \cite[Example (2) on page 355]{Neumaier1992completely}, which states that if $C$ is completely regular with covering radius $\rho$ then $C_\rho$ is also completely regular with the same distance partition (in reverse order). In particular, if $C$ is in fact completely transitive then, since the distance partition of $C$ (and thus also of $C_\rho$) consists of $\Aut(C)$-orbits, $C_\rho$ is also completely transitive. In this instance, the code $C=\L+\Delta_3$, {\em i.e.} the translate of $\L$ by $\Delta_3$, is known to be completely transitive and has covering radius $6$ with $C_6=\langle\L,\Delta_1\rangle\cup\langle\L,\Delta_2\rangle$. Given this, all of the codes appearing in Table~\ref{binaryCTtable} are known to be completely regular and completely transitive (see \cite[Pages 22--24]{borges2019completely} and the references therein). We remark though that our proofs of the complete transitivity of $\langle\L,\Delta_1\rangle$, $\L$, $\G_{23}^\perp$, $\E_{22}$ and $\S_{22}$ do appear to be new (see Section~\ref{sect:codes} and Lemmas~\ref{halfg23ct}, \ref{math22CT} and \ref{psl34CT}).

The next definition associates with any code a certain linear subcode. Note that the condition that $T_{C_{\text{max}}}\leq \Aut(C)$ is included to ensure $C_{\text{max}}$ is uniquely defined. Also, $T_{C_{\text{max}}}$ is sometimes called the \emph{kernel} of the code, though we do not use that terminology here.

\begin{definition}\label{maximallineardef}
 Let $Q=\F_2$ and $C$ be a code in $H(m,2)$. The \emph{maximal linear subcode} of $C$, denoted $C_{\text{max}}$, is the largest linear subcode of $C$ with the additional property that $T_{C_{\text{max}}}\leq \Aut(C)$, where $T_{C_{\text{max}}}$ is the group of translations by elements of $C_{\text{max}}$. 
\end{definition}

Below we restate \cite[Theorem~1.2]{minimal2nt}, which is the starting point for the proof of Theorem~\ref{binaryCTclass}. Note that Theorem~\ref{binaryx2ntchar} is presented in terms of $C_{\text{max}}$, as in Definition~\ref{maximallineardef}, and so differs slightly from the original statement in \cite{minimal2nt}. Including the conditions on $C_{\text{max}}$ here is intended to increase the transparency of our proof strategy. Indeed, the proof of \cite[Theorem~1.2]{minimal2nt} is effectively divided up according to whether $C_{\text{max}}$ has dimension $0$, $1$, or at least $2$, though the notation $C_{\text{max}}$ is not used there. Thus, we hope the reader finds it fairly straightforward to verify that the formulation below is indeed equivalent. Note also that $C$ non-trivial implies that $|C|>2$, and hence the binary repetition code does not appear as a possibility for $C$. Recall that the groups from \cite[Table~1]{minimal2nt} that are relevant to Theorem~\ref{binaryCTclass} also appear here in Table~\ref{binarytable}.

\begin{theorem}\label{binaryx2ntchar}
 Let $C$ be a non-trivial binary code in $H(m,2)$ with minimum distance at least $5$ containing the codeword ${\b 0}$. Then $C$ is $2$-neighbour-transitive if and only if one of the following holds:
 \begin{enumerate}
  \item $C_{\text{max}}=\{{\b 0}\}$ and $C$ is the even weight subcode $\E$ of the punctured Hadamard code with $m=11$ and minimum distance $6$;
  \item $C_{\text{max}}$ is the binary repetition code and $C$ is one of the following codes:
  \begin{enumerate}
   \item the Hadamard code $\H$ with $m=12$ and minimum distance $6$,
   \item the punctured Hadamard code $\PH$ with $m=11$ and minimum distance $5$; or,
  \end{enumerate}
  \item $C_{\text{max}}$ has dimension at least $2$ and minimum distance $\delta_{\text{max}}\geq 5$, and there exists a subgroup $X_{\b 0}\leq \Aut(C_{\text{max}})_{\b 0}$, where $X_{\b 0}$ and $m$ are as in \cite[Table~1]{minimal2nt}, such that
  \begin{enumerate}
   \item $C_{\text{max}}$ is $(X,2)$-neighbour-transitive, where $X=T_{C_{\text{max}}}\rtimes X_{\b 0}\leq \Aut(C)$ with $T_{C_{\text{max}}}$ denoting the group of translations by elements of $C_{\text{max}}$, and,
   \item $C$ is the union of a set $\S$ of cosets of $C_{\text{max}}$, and $\Aut(C)$ acts transitively on $\S$.
  \end{enumerate}
 \end{enumerate}
\end{theorem}


Hypothesis~\ref{hyp1}, presented below, will be assumed to hold throughout the majority of the paper. Note that the condition on the dimension of $C_{\text{max}}$ simply ensures that $C_{\text{max}}$ is neither the repetition code nor its dual, which is the case we are interested in here. 

\begin{hypothesis}\label{hyp1}
 Let $C$ be a completely transitive code in $H(m,2)$ with minimum distance $\delta\geq 5$ and let $X=\Aut(C)$. Let $C_{\text{max}}$ be the maximal linear subcode of $C$, let $\delta_{\text{max}}$ be the minimum distance of $C_{\text{max}}$, and let $X_{\text{max}}$ be the setwise stabiliser in $X$ of $C_{\text{max}}$. Furthermore, assume that $2\leq \dim(C_{\text{max}})\leq m-2$.
\end{hypothesis}


Note that Hypothesis~\ref{hyp2}, an extension of Hypothesis~\ref{hyp1}, includes an additional assumption relating to the socles of the actions of $X$ and $X_{\text{max}}$ on $M$. We briefly outline the significance of this hypothesis. In Section~\ref{sect:prelim} we show that $X^M\cong X/K$ and $X_{\text{max}}^M\cong X_{\text{max}}/K$, where $K=T_{C_{\text{max}}}$, and that both of these actions, of $X$ on $M$ and of $X_{\text{max}}$ on $M$, are $2$-homogeneous. Since $2$-homogeneous groups are classified according to their socles, and we deal with the case that these two actions have different socles in Section~\ref{sect:a7}, Hypothesis~\ref{hyp2} will be assumed in Sections~\ref{sect:mathgroups} to \ref{sect:Reer}.


\begin{hypothesis}\label{hyp2}
 Assume Hypothesis~\ref{hyp1} holds, that the kernel $K$ of the action of $X$ on $M$ is $T_{C_{\text{max}}}$ and that
 \[
  \soc(X/K)=\soc(X_{\text{max}}/K).
 \]
\end{hypothesis}

In the next section, we introduce the notation used throughout the paper and various preliminary results. In particular, in Section~\ref{sect:comptrans} we prove some new results regarding completely transitive codes useful in later sections. Section~\ref{sect:codes} contains the definitions of each of the codes appearing in Table~\ref{binaryCTtable} or otherwise, and Section~\ref{distreggraphsect} answers the question of the existence of a distance-regular graph related to each code. Section~\ref{sect:a7} deals with the case where Hypothesis~\ref{hyp1} holds but Hypothesis~\ref{hyp2} does not, which turns out to be precisely the case $\alt_7$ from Table~\ref{binarytable}. Thus the stronger Hypothesis~\ref{hyp2} may be assumed in the subsequent Sections~\ref{sect:mathgroups} to \ref{sect:Reer}, each of which deals with some collection of groups from Table~\ref{binarytable}. Finally, Section~\ref{sect:mainproof} proves Theorem~\ref{binaryCTclass}.

\section{Preliminaries}\label{sect:prelim}

Let the \emph{set of coordinate entries} $M$ and the \emph{alphabet} $Q$ be sets of sizes $m$ and $q$, respectively, where both $m$ and $q$ are integers and at least $2$. The vertex set $V\varGamma$ of the Hamming graph $\varGamma=H(m,q)$ consists of all strings in the alphabet $Q$ with positions indexed by the set $M$, which are usually written as $m$-tuples. Let $Q_i\cong Q$ be the copy of the alphabet in the entry $i\in M$ so that the vertex set of $H(m,q)$ is identified with the Cartesian product 
\[
 V\varGamma=\prod_{i\in M}Q_i.
\] 
An edge in $\varGamma$ exists between two vertices if and only if they differ as $m$-tuples in exactly one entry. If $\alpha$ is a vertex of $H(m,q)$ and $i\in M$ then $\alpha_i$ refers to the value of $\alpha$ in the $i$-th entry, that is, $\alpha_i\in Q_i$, so that $\alpha=(\alpha_1,\ldots,\alpha_m)$ when $M=\{1,\ldots,m\}$.  For more in depth background material on coding theory see \cite{cameron1991designs} or \cite{macwilliams1978theory}.

When considering binary codes we will be flexible with our interpretation of the vertices of $H(m,2)$. When $q=2$ we can identify $Q$ with $\F_2$ and $V\varGamma$ with the vector space $\F_2^m$. Let $\{e_i\mid i\in M\}$ be the standard basis for $V\varGamma\cong\F_2^m$. Then a vertex $\alpha$ of $H(m,2)$ will be simultaneously considered to be a vector $\alpha=\sum_{i\in M}\alpha_ie_i$, a function $\alpha:M\ra\F_2$ such that $i\mapsto \alpha_i$, and a subset $\{i\in M\mid \alpha_i=1\}$. Note that the Hamming metric ensures that there is only one sensible way to do this. As such, if $\alpha,\beta$ are vertices of $H(m,2)$, then so are the sum $\alpha+\beta$ (where the vertices are taken to be either vectors or functions), the product $\alpha\beta$ (where the vertices are taken to be functions), and the intersection $\alpha\cap\beta$, union $\alpha\cup\beta$ and set difference $M\setminus\alpha$ (where vertices are considered to be sets).

Let $\alpha,\beta$ be vertices and $C$ be a code in a Hamming graph $H(m,q)$ with $0\in Q$ a distinguished element of the alphabet. A summary of important notation regarding codes in Hamming graphs is contained in Table~\ref{hammingnotation}.
\begin{table}
 \begin{center}
 \begin{tabular}{cp{7 cm}}
  Notation & Explanation\\
  \hline
  $\b 0$ & vertex with $0$ in each entry\\
  
  $\wt(\alpha)=|\supp(\alpha)|$ & weight of $\alpha$\\
  
  $d(\alpha,\beta)=|\diff(\alpha,\beta)|$ & Hamming distance\\
  
  $\varGamma_s(\alpha)=\{\beta\in V\varGamma \mid d(\alpha,\beta)=s\}$ & set of $s$-neighbours of $\alpha$\\
  
  $\delta=\min\{d(\alpha,\beta)\mid \alpha,\beta\in C,\alpha\neq\beta\}$ & minimum distance of $C$\\
  
  $d(\alpha,C)=\min\{d(\alpha,\beta) \mid \beta\in C\}$ & distance from $\alpha$ to $C$\\
  
  $\rho =\max\{d(\alpha,C)\mid\alpha\in V\varGamma\}$ & covering radius of $C$\\
  
  $C_s=\{\alpha\in V\varGamma \mid d(\alpha,C)=s\}$ & set of $s$-neighbours of $C$\\
  
  $\{C=C_0,C_1,\ldots, C_\rho\}$ & distance partition of $C$\\
    
  \hline
 \end{tabular}
 \caption{Hamming graph notation.}
 \label{hammingnotation}
 \end{center}
\end{table}

Note that if the minimum distance $\delta$ of a code $C$ satisfies $\delta\geq 2s$, then the set of $s$-neighbours $C_s$ satisfies $C_s=\cup_{\alpha\in C}\varGamma_s(\alpha)$ and if $\delta\geq 2s+1$ this is a disjoint union. This fact is crucial in many of the proofs below; it is often assumed that $\delta\geq 5$, in which case every element of $C_2$ is distance $2$ from a unique codeword.

A \emph{linear} code is a code $C$ in $H(m,q)$ with alphabet $Q=\F_q$ a finite field, so that the vertices of $H(m,q)$ form a vector space $V$, such that $C$ is an $\F_q$-subspace of $V$. Given $\alpha,\beta\in V$, the usual inner product is given by $\langle\alpha,\beta\rangle=\sum_{i\in M}\alpha_i\beta_i$. The \emph{dual} code of $C$ is $C^\perp=\{\beta\in V\mid \forall \alpha\in C,\langle\alpha,\beta\rangle=0\}$. Note also that we write the binary repetition code, that is, the subspace of $\F_2^m$ generated by the vector ${\b 1}=(1,\ldots,1)$, as $\langle {\b 1}\rangle$ and its dual as $\langle {\b 1}\rangle^\perp$.

Let $C$ be a code in $H(m,q)$. The \emph{punctured code}, also sometimes called the \emph{truncated code}, of $C$ at the coordinate $i\in M$ is the code in $H(m-1,q)$ obtained by deleting the entry $\alpha_i$ from each codeword $\alpha\in C$. The \emph{shortened code} of $C$ at $i\in M$ is the code in $H(m-1,q)$ obtained by choosing only those codewords $\alpha\in C$ with $\alpha_i=0$ and deleting the entry $\alpha_i$. If the set of punctured (respectively shortened) codes obtained as $i$ ranges over $M$ are pairwise equivalent then we simply say the punctured (respectively the shortened) code of $C$; in particular, this occurs if $\Aut(C)_{{\b 0}}$ acts transitively on $M$ (for notation regarding automorphisms see Section~\ref{hamminggraphautoprelim}). The \emph{even-weight subcode} of $C$ is the code consisting of precisely those codewords in $C$ having even weight.

The \emph{Singleton bound} (see \cite[4.3.2]{delsarte1973algebraic}) is a well known bound for the size of a code $C$ in $H(m,q)$ with minimum distance $\delta$, stating that $|C|\leq q^{m-\delta+1}$. For a linear code $C$ this may be stated as $\delta^\perp-1\leq k\leq m-\delta+1$, where $k$ is the dimension of $C$, $\delta$ is the minimum distance of $C$ and $\delta^\perp$ is the minimum distance of $C^\perp$.

Let $C$ be a code and $s$ be the external distance of $C$, that is, $s+1$ is the number of non-zero terms in the \emph{dual distance distribution} of $C$. It follows from \cite[Theorem~4.1]{sole1990completely} that if $C$ is completely regular then $s$ is equal to the covering radius $\rho$ of $C$. Indeed, it is shown in \cite{Bassalygo1977note} that if the parameters of a code $C$ satisfy $\rho=s=e+1$ then $C$ is completely regular.


\subsection{Automorphisms of a Hamming graph}\label{hamminggraphautoprelim}

The automorphism group $\Aut(\varGamma)$ of the Hamming graph is the semi-direct product $B\rtimes L$, where $B\cong (\Sym(Q))^m$ and $L\cong \Sym(M)$ (see \cite[Theorem 9.2.1]{brouwer}). Note that $B$ and $L$ are called the \emph{base group} and the \emph{top group}, respectively, of $\Aut(\varGamma)$. Since we identify $Q_i$ with $Q$, we also identify $\Sym(Q_i)$ with $\Sym(Q)$. If $h\in B$ and $i\in M$ then $h_i\in \Sym(Q_i)$ is the image of the action of $h$ in the entry $i\in M$. Let $h\in B$, $\sigma\in L$ and $\alpha\in V\varGamma$. Then $h$ and $\sigma$ act on $\alpha$ explicitly via: 
\begin{equation*}
\alpha^h =(\alpha_1^{h_1},\ldots,\alpha_m^{h_m})\quad\text{and}\quad
\alpha^\sigma=(\alpha_{1^{\sigma^{-1}}},\ldots,\alpha_{m^{\sigma^{-1}}}).
\end{equation*}
The automorphism group of a code $C$ in $\varGamma=H(m,q)$ is $\Aut(C)=\Aut(\varGamma)_C$, the setwise stabiliser of $C$ in $\Aut(\varGamma)$. 

Let $G$ be a group acting on a set $\Omega$, $\omega$ be an element of $\Omega$, and $S$ be a subset of $\Omega$. Then we use the following notation:
\begin{enumerate}
 \item $G_\omega$ denotes the subgroup of $G$ stabilising $\omega$,
 \item $G_S$ denotes the setwise stabiliser of $S$ in $G$,
 \item $G_{(S)}$ denotes the point-wise stabiliser of $S$ in $G$, and,
 \item if $G$ fixes $S$ setwise then $G^S$ denotes the subgroup of $\Sym(S)$ induced by $G$.
\end{enumerate}
(For more background and notation on permutation groups see, for instance, \cite{dixon1996permutation}.) In particular, let $X\leq\Aut(\varGamma)$. Then:
\begin{enumerate}
 \item For $x\in X$, recall that $x=h\sigma$ where $h\in B$ and $\sigma \in L$. Then $x^M=\sigma$ denotes the permutation of $M$ induced by $x$, and we write $X^M=\{x^M\mid x\in X\}$; we call $X^M$ the \emph{action of $X$ on entries}. Note that a pre-image $x$ of an element $x^M$ of $X^M$ need not fix any vertex of $H(m,q)$.
 \item $K=X\cap B$ is the kernel $X_{(M)}$ of the action of $X$ on entries and is precisely the subgroup of $X$ fixing $M$ point-wise.
 \item If $i\in M$, then $X_i$ denotes the subgroup of $X$ stabilising the entry $i$ and any $x\in X_i$ is of the form $h\sigma$ ($h\in B$ and $\sigma \in L$) where $\sigma$ fixes $i\in M$. So $x=h\sigma$ induces the permutation $h_i\in\Sym(Q_i)$ on the alphabet $Q_i$. This defines a homomorphism from $X_i$ to $\Sym(Q_i)$ and we denote the image of this homomorphism by $X_i^{Q_i}$. We refer to $X_i^{Q_i}$ as the \emph{action on the alphabet}.
\end{enumerate}


It is worth mentioning that coding theorists often consider more restricted groups of automorphisms, such as the group $\PermAut(C)=\{\sigma\mid h\sigma\in\Aut(C), h=1\in B, \sigma\in L\}$. The elements of this group are called \emph{pure permutations} on the entries of the code. 

Two codes $C$ and $C'$ in $H(m,q)$ are said to be \textit{equivalent} if there exists some $x\in \Aut(\varGamma)$ such that $C^x=\{\alpha^x\mid\alpha\in C\}=C'$. Equivalence preserves many of the important properties in coding theory, such as minimum distance and covering radius, since $\Aut(\varGamma)$ preserves distances in $H(m,q)$. If $q=2$ then, since any element of $B\cong \Z_2^m$ can be seen to be a translation by a vertex, we have that if $C$ and $C'$ are equivalent codes in $H(m,2)$ with ${\b 0}\in C\cap C'$, then $C_{\text{max}}$ and $C'_{\text{max}}$ are also equivalent. This fact is important in the context of this paper and will be used without reference. However, it should be noted that whilst this is true also for $q=3$, it does not hold in general.

\subsection{\texorpdfstring{$s$}{s}-Neighbour-transitive codes}\label{sntr}

This section presents preliminary results regarding $(X,s)$-neighbour-transitive codes, defined in Definition~\ref{sneighbourtransdef}. The next results give certain $2$-homogeneous and $2$-transitive actions associated with an $(X,2)$-neighbour-transitive code.

\begin{proposition}\cite[Proposition~2.5]{ef2nt}\label{ihom} 
 Let $C$ be an $(X,s)$-neighbour-transitive code in $H(m,q)$ with minimum distance $\delta$, where $\delta\geq 3$ and $s\geq 1$. Then for $\alpha\in C$ and $i\leq\min\{s,\lfloor\frac{\delta-1}{2}\rfloor\}$, the stabiliser $X_\alpha$ fixes setwise and acts transitively on $\varGamma_i(\alpha)$.  In particular, the action of $X_\alpha$ on $M$ is $i$-homogeneous.    
\end{proposition}

%

\begin{proposition}\cite[Proposition~2.7]{ef2nt}\label{x12trans} 
 Let $C$ be an $(X,1)$-neighbour-transitive code in $H(m,q)$ with minimum distance $\delta\geq 3$ and $|C|>1$.  Then $X_i^{Q_i}$ acts $2$-transitively on $Q_i$ for all $i\in M$.
\end{proposition}
%
%
%

The concept of a block design, introduced below, comes up frequently in coding theory. Recall that we identify vertices in $H(m,2)$ with subsets of $M$ so that a design as defined below can be identified with a code in $H(m,2)$.

\begin{definition}\label{desdef}
 A \emph{$t$-$(m,k,\lambda)$ design} is a collection $\mathcal{D}$ of subsets of a set $M$ of size $m$ such that every $t$-subset of $M$ is contained in exactly $\lambda$ elements of $\mathcal{D}$. The elements of $\D$ are called \emph{blocks}.
\end{definition}

The following equations can be found, for instance, in \cite{stinson2004combinatorial}. Let $\D$ be a $t$-$(m,k,\lambda)$ design with $|\D|=b$ blocks and let $r$ be the number of blocks containing any given point. Then $mr=bk$, $r(k-1)=\lambda(m-1)$ and 
\begin{equation}
 b=\frac{m(m-1)\cdots(m-s+1)}{k(k-1)\cdots(k-s+1)}\lambda.\label{designparams}
\end{equation}

The definition below is required in order to state the remaining two results of this section. Recall that any $s$-neighbour-transitive code is an $s$-regular code and any completely transitive code is completely regular.

\begin{definition}\label{regcodedefinition}
 Let $C$ be a code in $H(m,q)$ with covering radius $\rho$, and $s$ be an integer with $0\leq s\leq \rho$. Then,
 \begin{enumerate}
  \item $C$ is \emph{$s$-regular} if, for each $i\in\{0,1,\ldots,s\}$, each $k\in\{0,1,\ldots,m\}$, and every vertex $\nu\in C_i$, the number $|\varGamma_k(\nu)\cap C|$ depends only on $i$ and $k$, and,
  \item $C$ is \emph{completely regular} if $C$ is $\rho$-regular.
 \end{enumerate}
\end{definition}

\begin{lemma}\cite[Lemma~2.16]{ef2nt}\label{design}
 Let $C$ be an $(X,s)$-neighbour transitive code in $H(m,q)$. Then $C$ is $s$-regular. Moreover, if $C$ has minimum distance $\delta\geq 2s$ and contains $\b 0$, then for each $k\leq m$ the set of codewords of weight $k$ forms a $q$-ary $s$-$(m,k,\lambda)$ design, for some $\lambda$.
\end{lemma}

\subsection{Completely transitive codes}\label{sect:comptrans}

This section contains mostly new preliminary results relating to the main concern of this paper, completely transitive codes. In particular, many of the results of this section assume Hypothesis~\ref{hyp1}. We believe the next result to be known, though we were unable to find an exact reference.

\begin{lemma}\label{upboundmindist}
 Let $C$ be a completely transitive code on $H(m,q)$, let $X=\Aut(C)$ and suppose that $X_{\b 0}$ acts $t$-homogeneously, but not $(t+1)$-homogeneously, on $M$. Then $\delta\leq 2t+2$.
\end{lemma}

\begin{proof}
 Suppose, for a contradiction, that $\delta\geq 2t+3$. Then $\left\lfloor\frac{\delta-1}{2}\right\rfloor\geq t+1$, so that any two balls of radius $t+1$ centered around distinct codewords are disjoint. Note that this also implies that the covering radius $\rho$ of $C$ is at least $t+1$. By assumption, $C$ is $(X,\rho)$-neighbour-transitive, and hence Proposition~\ref{ihom} implies that $X_{\b 0}$ acts at least $(t+1)$-homogeneously on $M$, giving a contradiction.
\end{proof}
%
%

%

The following relies on \cite[Proposition~3.5]{minimal2nt} and gives more explicit information about a code satisfying Hypothesis~\ref{hyp1}.

\begin{proposition}\label{xmaxlemma}
 Assume Hypothesis~\ref{hyp1}. Then each of the following hold:
 \begin{enumerate}
  \item $X_{\text{max}}=T_{C_{\text{max}}}\rtimes X_{\b 0}$.
  \item $C_{\text{max}}$ is $(X_{\text{max}},2)$-neighbour-transitive with $\delta_{\text{max}}\geq 5$.
  \item The kernel $K$ of the action of $X$ on $M$ is equal to the kernel of the action of $X_{\text{max}}$ on $M$, that is, $K=X\cap B=X_{\text{max}}\cap B=T_{C_{\text{max}}}$.
  \item $C_{\text{max}}$ is an $\F_2 X_{\b 0}$-module
  \item $\displaystyle\frac{|C|}{|C_{\text{max}}|}= \frac{|X|}{|X_{\text{max}}|}=\frac{|X^M|}{|X_{\b 0}^M|}$.
 \end{enumerate}
\end{proposition}

\begin{proof}
 We shall apply \cite[Proposition~3.5]{minimal2nt}, and so we first we argue that $C$ satisfies the required conditions. Since $C$ has minimum distance $\delta\geq 5$, it follows that $C$ has covering radius at least $2$. Hence, $X$ acts transitively on $C$, $C_1$ and $C_2$, each of which are non-empty, so that $C$ is $(X,2)$-neighbour-transitive. By Proposition~\ref{ihom}, $X$ acts transitively on $M$, by Proposition~\ref{x12trans} we have that $X_i^{Q_i}\cong \s_2$, for each $i\in M$, and $T_{C_{\text{max}}}\leq X$ implies that $K\neq 1$. Hence $C$ is \emph{$X$-alphabet-affine} (see \cite[Definition~1.4]{minimal2nt}) and we may apply \cite[Proposition~3.5]{minimal2nt}. 
 
 We claim that $O_2(K)=K=T_{C_{\text{max}}}$, where $K=X\cap B$ and $O_2(K)$ is the largest normal $2$-subgroup of $K$. This implies that the code $W={\b 0}^{O_2(K)}$ appearing in the statement of \cite[Proposition~3.5]{minimal2nt} is in fact $C_{\text{max}}$, and that the above parts 1, 2, 3 and 4 hold. Since $T_{C_{\text{max}}}$ is a $2$-group and hence $O_2(T_{C_{\text{max}}})=T_{C_{\text{max}}}$, in order to prove the claim we need only show that $K=T_{C_{\text{max}}}$. Suppose $x\in K$. Then $x=(h_1,\ldots,h_m)$ and each $h_i\in \Z_2$, that is, $x=t_\alpha$ where $\alpha$ is the vertex of $H(m,2)$ such that $\alpha_i=0$ if $h_i=0$ and $\alpha_i=1$ if $h_i=1$, for each $i\in M$. It follows that there exists a subcode $C'$ of $C$ such that $K=T_{C'}$. Since $K$ is a group and $q=2$ (and hence $\F_2^\times$ is the trivial group) we have that $C'$ is linear. The fact $K=X\cap B$ implies that there is no larger group of translations contained in $X$ and thus, by Definition~\ref{maximallineardef}, we have that $C'=C_{\text{max}}$ and $K=T_{C_{\text{max}}}$, proving the claim.

 Now $K=T_{C_{\text{max}}}$ is normal in $X$, which means that the $T_{C_{\text{max}}}$-orbits of $C$ form a system of imprimitivity for the action of $X$ on $C$. Since $X$ acts transitively on the set of $T_{C_{\text{max}}}$-orbits of $C$, the orbit-stabiliser theorem implies that
 \[
  \frac{|C|}{|C_{\text{max}}|}=\frac{|X|}{|X_{\text{max}}|}= \frac{|X/T_{C_{\text{max}}}|}{|X_{\text{max}}/T_{C_{\text{max}}}|}=\frac{|X^M|}{|X_{\b 0}^M|},
 \]
 where the final equality follows from parts 1 and 3. Thus, part 5 holds.
\end{proof}

The following result shows, roughly speaking, that if a completely transitive code is not much bigger than its maximal linear subcode then the minimum distance of the maximal linear subcode is not much bigger than the minimum distance of the code. When combined with Lemma~\ref{upboundmindist}, this leads to non-existence results in several cases. 

\begin{lemma}\label{largedeltamax}
 Assume Hypothesis~\ref{hyp1} holds. If 
 \[
  \frac{|C|}{|C_{\text{max}}|}< \frac{m(m-1)}{\delta(\delta-1)},
 \]
 then $\delta_{\text{max}}\leq 2\delta$.
\end{lemma}

\begin{proof}
 We shall prove the contrapositive, and hence we assume that $\delta_{\text{max}}\geq 2\delta+1$ and will prove that then $\frac{|C|}{|C_{\text{max}}|}\geq \frac{m(m-1)}{\delta(\delta-1)}$. By Lemma~\ref{design}, the set $C\cap \varGamma_\delta({\b 0})$ of weight $\delta$ codewords of $C$ form a $2$-$(m,\delta,\lambda)$ design, for some integer $\lambda$. Hence
 \[
  |C\cap \varGamma_\delta({\b 0})|=\frac{m(m-1)\lambda}{\delta(\delta-1)}\geq \frac{m(m-1)}{\delta(\delta-1)}.
 \]
 For $\alpha,\beta\in C\cap \varGamma_\delta({\b 0})$ we have, by the triangle inequality, that $d(\alpha,\beta)\leq 2\delta$. Since $\delta_{\text{max}}\geq 2\delta+1$, it follows that no two elements of $C\cap \varGamma_\delta({\b 0})$ are contained in the same coset of $C_{\text{max}}$. Thus, $C$ is comprised of at least $|C\cap \varGamma_\delta({\b 0})|$ cosets of $C_{\text{max}}$ and the result holds.
\end{proof}

The following lemma is useful for proving non-existence of $X$-neighbour-transitive codes in the case where we are dealing with a family of groups $X$ having order that is polynomial in the length $m$ of the code.

\begin{lemma}\label{morbitsbound}
 Let $C$ be an $X$-completely transitive code in $H(m,q)$. Then 
 \[
  (m+1)|X|\geq q^m.
 \]
\end{lemma}

\begin{proof}
 Since $C$ is $X$-completely transitive, we have that there are $\rho+1$ orbits of $X$ on the vertex set of $H(m,q)$. Now, the covering radius $\rho$ is bounded above by the diameter $m$ of $H(m,q)$, and hence $m+1\geq \rho+1$. Thus, using the fact that the length of any orbit of $X$ on vertices of $H(m,q)$ is bounded above by $|X|$, we have that 
 \[
  (m+1)|X|\geq(\rho+1)|X|\geq |V\varGamma|=q^m,
 \]
 and the result holds.
\end{proof}

The next result shows that any completely transitive code in $H(m,2)$, where $Q$ is taken to be $\F_2$, is a subcode of some linear completely transitive code.

\begin{lemma}\label{linearcompletion}
 Assume Hypothesis~\ref{hyp1}. Furthermore, suppose $C\neq C_{\text{max}}$, let $n$ be the integer such that $|C|=n|C_{\text{max}}|$ and let $\langle C\rangle$ be the $\F_2$-span of the codewords of $C$. Then each of the following holds:
 \begin{enumerate}
  \item $\Aut(C)\leq \Aut(\langle C\rangle)$.
  \item $\langle C\rangle$ is completely transitive.
  \item The codimension of $C_{\text{max}}$ in $\langle C\rangle$ is at least $2$ and at most $n-1$.
 \end{enumerate}
\end{lemma}

\begin{proof}
 Let $\alpha\in{\langle C\rangle}$. Then there exists an integer $k$ such that $\alpha=\alpha_1+\cdots+\alpha_k$, where each $\alpha_i\in C$. Let $x=h\sigma\in \Aut(C)$, where $h\in B\cong \Z_2^m$ and $\sigma\in L$, and let $\beta={\b 0}^h$. Note that $\beta\in C$ since $\sigma$, and hence $\sigma^{-1}$, fixes the vertex ${\b 0}$, so that
 \[
  {\b 0}^{x^{-1}}={\b 0}^{\sigma^{-1}h^{-1}}={\b 0}^{h^{-1}}=\beta,
 \]
 with the last step following from the fact that $h=h^{-1}$. Hence $x=t_\beta\sigma$, where $t_\beta$ is the translation by $\beta$. Now $\PermAut(\varGamma)=\s_m$ embeds as the group of permutation matrices in $\GL_m(2)$, from which it follows that $V\varGamma$ is an $\F_2\PermAut(\varGamma)$-module. Since $\sigma\in\PermAut(\varGamma)$, we have that
 \begin{align*}
  \alpha^x=(\alpha+\beta)^{\sigma}
  &=(\alpha_1)^{\sigma}+\cdots+(\alpha_k)^{\sigma}+\beta^{\sigma}\\
  &=(\alpha_1+\beta)^x+\cdots+(\alpha_k+\beta)^x+{\b 0}^x.
 \end{align*}
 As ${\b 0}$ and each $\alpha_i+\beta$ are in $C$, and $x\in\Aut(C)$, we have that ${\b 0}^x$ and each $(\alpha_i+\beta)^x$ are also in $C$. Thus, we deduce that $\alpha^x\in {\langle C\rangle}$, which implies that $x\in \Aut({\langle C\rangle})$. Thus $\Aut(C)\leq \Aut({\langle C\rangle})$. 
 
 Since $C$ is a subcode of $\langle C\rangle$, it follows that the covering radius $\rho'$ of $\langle C\rangle$ is at most the covering radius $\rho$ of $C$. First, note that $\Aut(\langle C\rangle)$ contains the translations by $\langle C\rangle$ and thus acts transitively on $\langle C\rangle$. Let $\mu,\nu\in \langle C\rangle_i$ for some $i$, where $1\leq i\leq \rho'$. Then there exist $\alpha,\beta\in \langle C\rangle$ such that $d(\alpha,\mu)=d(\beta,\nu)=i$. If $t_\alpha$ and $t_\beta$ are the translations by $\alpha$ and $\beta$, respectively, then it follows that $\alpha^{t_\alpha}={\b 0}$ and $\beta^{t_\beta}={\b 0}$ so that $d({\b 0},\mu^{t_\alpha})=d({\b 0},\nu^{t_\alpha})=i$. Hence, since ${\b 0}\in C$, we have that $\mu^{t_\alpha}$ and $\nu^{t_\beta}$ are in $C_i$, and thus there exists an $x\in \Aut(C)$ such that $(\mu^{t_\alpha})^x=\nu^{t_\beta}$. This implies that $\mu^{t_\alpha x t_\beta}=\nu$ and, since $\Aut(C)\leq \Aut(\langle C\rangle)$, that $\Aut(\langle C\rangle)$ acts transitively on $\langle C\rangle_i$. Thus $\langle C\rangle$ is completely transitive, proving part 2. 
 
 By Theorem~\ref{binaryx2ntchar} part 3 (b), $C$ is the union of $C_{\text{max}}$ and $n-1$ non-trivial cosets of $C_{\text{max}}$. Let $v_1,\ldots,v_{n-1}$ be coset representatives for these $n-1$ non-trivial cosets. It follows that ${\langle C\rangle}=\langle C_{\text{max}},v_1,\ldots, v_{n-1}\rangle$, and hence that the codimension of $C_{\text{max}}$ in $\langle C\rangle$ is at most $n-1$. Suppose that the codimension of $C_{\text{max}}$ in $\langle C\rangle$ is $1$ and hence, since $q=2$, we have that $n=2$. Thus $|C|=2|C_{\text{max}}|$ and $\langle C\rangle = \langle C_{\text{max}},v_1\rangle$, that is $|\langle C\rangle|=2|C_{\text{max}}|$, and in fact $C=\langle C\rangle$. This implies that $C$ is linear, contradicting the assumption that $C\neq C_{\text{max}}$. Thus the codimension of $C_{\text{max}}$ in $\langle C\rangle$ is at least $2$, as in part 3.
\end{proof}

The following result effectively gives upper bounds for the number of weight $\delta$ codewords in a binary completely transitive code when $\delta$ is $5$ or $6$.

\begin{lemma}\label{lambdabound}
 Let $C$ be a completely transitive code in $H(m,2)$ with minimum distance $\delta=5$ or $6$. Then the set of all weight $\delta$ codewords form a $t$-$(m,\delta,\lambda)$ design, where:
 \begin{enumerate}
  \item $\delta=5$, $t=2$ and $\lambda\leq (m-2)/3$, or;
  \item $\delta=6$, $t=3$ and $\lambda\leq (m-3)/3$.
 \end{enumerate}
\end{lemma}

\begin{proof}
 The set of all weight $\delta$ codewords form the block set $\B$ of a $t$-$(m,\delta,\lambda)$ design by \cite[Theorem~2]{borges2019completely}, where $t=2$ when $\delta=5$ and $t=3$ when $\delta=6$. First, let $\delta=5$ and $i,j$ be distinct elements of $M$. Let $\B_{i,j}$ be the set of blocks incident with $i$ and $j$ projected down to $M\setminus\{i,j\}$. It follows that there are $\lambda$ elements of $\B_{i,j}$, each one being a subset of size $3$ in $M\setminus\{i,j\}$. Let $\alpha$ and $\beta$ be distinct subsets of $M$ of size $5$ with $i,j\in \alpha,\beta$. Then $d(\alpha,\beta)$ is either $2$, $4$ or $6$ depending on whether $\alpha\cap\beta$ has size $4$, $3$ or $2$, respectively. Thus, since $\delta= 5$, if we assume $\alpha$ and $\beta$ are in $C$, then $|\alpha\cap\beta|=2$ and the images of $\alpha$ and $\beta$ in $\B_{i,j}$ are disjoint. Since any two distinct blocks of $\B_{i,j}$ are disjoint there are at most $(m-2)/3$ such blocks, and hence at most $(m-2)/3$ weight $5$ codewords that are non-zero in a particular pair of coordinates, proving part 1.
 
 Let $\delta=6$ and $i,j,k$ be pairwise distinct elements of $M$. Let $\B_{i,j,k}$ be the set of blocks incident with $i,j$ and $k$ projected down to $M\setminus\{i,j,k\}$. Again, there are $\lambda$ elements of $\B_{i,j,k}$, and each block is incident with $3$ elements of $M\setminus\{i,j,k\}$. Let $\alpha$ and $\beta$ be distinct subsets of $M$ of size $6$ with $i,j,k\in \alpha,\beta$. Then $d(\alpha,\beta)$ is either $2$, $4$ or $6$ depending on whether $\alpha\cap\beta$ has size $5$, $4$ or $3$, respectively. Thus, since $\delta=6$, if we assume $\alpha$ and $\beta$ are in $C$, then $|\alpha\cap\beta|=3$ and the images of $\alpha$ and $\beta$ in $\B_{i,j,k}$ are disjoint. Since any two distinct blocks of $\B_{i,j,k}$ are disjoint there are at most $(m-3)/3$ such blocks, and hence at most $(m-3)/3$ weight $6$ codewords that are non-zero in a particular triple $\{i,j,k\}$ of coordinates, proving part 2.
\end{proof}

\section{The Codes}\label{sect:codes}

In this section we fix our notation for the codes that arise in later sections, giving explicit constructions where required, references and other details required in later sections. From this point forward, all codes considered are binary. Recall that for binary codes we will be flexible with our interpretation of the vertices of $H(m,2)$; see Section~\ref{sect:prelim}. More specifically, we interchangeably view vertices as: binary strings indexed by $M$, subsets of $M$, or functions from $M$ to $Q=\{0,1\}$.

\begin{example}\cite[Definition~4.1]{ef2nt}\label{hadamarddef}
 We obtain the Hadamard $12$ code $\H$ and its punctured code $\PH$ as follows (see \cite[Part 1, Section 2.3]{macwilliams1978theory}). First, we construct a normalised Hadamard matrix $H_{12}$ of order $12$ using the Paley construction. 
 \begin{enumerate}
  \item Let $M=\F_{11}\cup \{*\}$ and let $H_{12}$ be the $12\times 12$ matrix with first row $v$, where $v_a=-1$ if $a$ is a square in $\F_{11}$ (including $0$), and $v_a=1$ if $a$ is a non-square in $\F_{11}$ or $a=*\in M$, taking the orbit of $v$ under the additive group of $\F_{11}$ acting on $M$ to form $10$ more rows and adding a final row, the vector $(-1,\ldots,-1)$. 
  
  \item The Hadamard code $\H$ of length $12$ in $H(12,2)$ then consists of the vertices $\alpha$ such that there exists a row $u$ in $H_{12}$ or $-H_{12}$ satisfying $\alpha_a=0$ when $u_a=1$ and $\alpha_a=1$ when $u_a=-1$. 
  
  \item The punctured code $\PH$ of $\H$ is obtained by deleting the coordinate $*$ from $M$. The weight $6$ codewords of $\PH$ form a binary $2$-$(11,6,3)$ design $\D$. The code $\PH$ consists of the following codewords: the zero codeword, the vector $(1,\ldots,1)$, the characteristic vectors of the $2$-$(11,6,3)$ design $\D$, and the characteristic vectors of the complement of that design, which forms a $2$-$(11,5,2)$ design. (Both $\D$ and its complement are unique up to isomorphism \cite{todd1933}.)
  
  \item The even weight subcode $\mathcal E$ of $\PH$ is the code consisting of the zero codeword and the $2$-$(11,6,3)$ design. 
 \end{enumerate}
\end{example}

\begin{example}\label{codesNR}
 The Nordstrom--Robinson code, which we denote throughout by $\NR$ is a well-known non-linear code having parameters $(15,256,5;3)$. Since we will not require a precise definition of the Nordstrom--Robinson code here, we refer the reader to \cite{gillespie2012nord} or \cite{Semakov1969perfect}, suffice it to say that $\NR$ is the union of $8$ cosets of the punctured Reed--Muller code $\RM(1,4)^*$ and is contained in $\RM(2,4)^*$ \cite{Semakov1969perfect}.
\end{example}

\begin{example}\label{pslcodesex}
 In this example we consider codes in $H(21,2)$. In particular, let $M$ be the set of points of $\pg_{2}(4)$. Note that here the points of ${\rm PG}_2(4)$ are in bijection with the weight $1$ vertices of $H(21,2)$.
 \begin{enumerate}
  \item Let $\P$ be the code generated by the set of all complements of lines of $\pg_2(4)$ and $\L$ be the code generated by the set of all lines of $\pg_2(4)$.
   
  \item The group $\PSL_3(4)$ has three orbits $\H_1,\H_2,\H_3$ on hyperovals (\emph{i.e.} sets of $6$ points with no three collinear), and three orbits $\mathcal{F}_1,\mathcal{F}_2,\mathcal{F}_3$ on Fano subplanes (\emph{i.e.} sets of $7$ points forming a projective plane of order $2$, when only those lines with $3$ such collinear points are considered) labelled so that if $\Delta\in\H_i$ and $\Phi\in\mathcal{F}_j$ then $|\Delta\cap\Phi|\leq 3$ if and only if $i\neq j$ (see \cite[Theorem~6.5B]{dixon1996permutation} and also Example~\ref{codesMathgps}). Let $\ell$ be a line of $\pg_2(4)$ and, for each $i=1,2,3$, let $\Phi_i\in \mathcal{F}_i$ and let $\Delta_i\in\H_i$. Let $R=\{\ell,\Delta_1,\Delta_2,\Delta_3,\Phi_1,\Phi_2,\Phi_3\}$. For a given code $C$ in $H(21,2)$ and some $\alpha_1,\ldots,\alpha_j\in R$ we shall often consider the code given by the linear span of $C$ and $\alpha_1,\ldots,\alpha_j$ over $\F_2$, which we denote simply by
  \[
   \langle C,\alpha_1,\ldots,\alpha_j\rangle.
  \]
  For example, we show in Section~\ref{sect:psl34} that $\langle\P,\ell\rangle=\L$. The codes arising in this manner and appearing in Table~\ref{binaryCTtable} are $\langle \L,\Delta_1\rangle$ and $\langle \L,\Delta_1\rangle\cup \langle \L,\Delta_2\rangle$.
 \end{enumerate}
\end{example}

\begin{example}\label{codesMathgps}
 Let $\{a,b,c\}$ be a set of size $3$ that is disjoint from the set $P$ of points of $\pg_2(4)$ and let $M=P\cup S$, where $S$ is a subset of $\{a,b,c\}$ of size $1$, $2$ or $3$. Thus $m=22$, $23$ or $24$. Following \cite[Theorem~6.7C and Section~6.5-6.7]{dixon1996permutation}, fix a construction for the block set of the Witt design $W_{24}$ (\emph{i.e.} the famous Steiner system $S(24, 8, 5)$), which will correspond to the set of weight $8$ codewords of the extended binary Golay code $\G_{24}$ in $H(24,2)$. As in Example~\ref{pslcodesex}, the group $\PSL_3(4)$ has three orbits $\H_1,\H_2,\H_3$ on hyperovals, and three orbits $\mathcal{F}_1,\mathcal{F}_2,\mathcal{F}_3$ on Fano subplanes labelled so that for $\Delta\in\H_i$ and $\Phi\in\mathcal{F}_j$ we have that $|\Delta\cap\Phi|\leq 3$ if and only if $i\neq j$. The blocks of $W_{24}$ are then:
 \begin{enumerate}
  \item $\ell\cup \{a,b,c\}$ for each line $\ell$ of $\pg_2(4)$,
  \item $\Delta\cup \{a,b\}$ for each hyperoval $\Delta\in\H_1$,
  \item $\Delta\cup \{b,c\}$ for each hyperoval $\Delta\in\H_2$,
  \item $\Delta\cup \{c,a\}$ for each hyperoval $\Delta\in\H_3$,
  \item $\Phi\cup \{c\}$ for each Fano plane $\Phi\in\mathcal{F}_1$,
  \item $\Phi\cup \{a\}$ for each Fano plane $\Phi\in\mathcal{F}_2$,
  \item $\Phi\cup \{b\}$ for each Fano plane $\Phi\in\mathcal{F}_3$,
  \item the symmetric difference $\ell_1+\ell_2$ for each pair of distinct lines of $\pg_2(4)$.
 \end{enumerate}
 We define the following codes:
 \begin{enumerate}
  \item For $S=\{a,b,c\}$, and hence $m=24$, we denote the code spanned by the blocks of $W_{24}$ by $\G_{24}$, which is self-dual.
  \item For $S=\{a,b\}$, and hence $m=23$, we denote the punctured code of $\G_{24}$ by $\G_{23}$ and its dual, the shortened code of $\G_{24}$, by $\G_{23}^\perp$. 
  \item For $S=\{a\}$, and hence $m=22$, we denote the punctured code of $\G_{23}$ by $\G_{22}$, the shortened code of $\G_{23}$ by $\SG_{22}$ and the even weight subcode of $\G_{22}$ by $\EG_{22}$.
 \end{enumerate}
\end{example}

\begin{remark}\label{dixonandmorterrors}
 Note that there is a typographical error on page 204 of \cite{dixon1996permutation} in the list of the different types of blocks in $W_{24}$ (v) and (vii): $\mathcal{F}_1$ and $\mathcal{F}_3$ should be interchanged. It is not difficult to see that this is required by the condition that if $\Delta\in\H_i$ and $\Phi\in\mathcal{F}_j$ then $|\Delta\cap\Phi|\leq 3$ if and only if $i\neq j$. Also, \cite[Theorem~6.5B (iv)]{dixon1996permutation} claims that $\PGaL_3(4)$ induces a cyclic permutation on the set $\{\H_1,\H_2,\H_3\}$ of $\PSL_3(4)$-orbits of hyperovals. In particular, the proof of \cite[Theorem~6.5B (iv)]{dixon1996permutation} claims without justification that the quotient $\PGaL_3(4)/\PSL_3(4)$ is cyclic of order $6$. Note that this is not true; indeed we show in Lemma~\ref{symdiffhyperovals} that $\PGaL_3(4)/\PSL_3(4)$ is isomorphic to $\s_3$.
\end{remark}

\subsection{Distance-regular graphs}\label{distreggraphsect}

The standard method (see \cite[Chapter~11]{brouwer}) of constructing a distance-regular graph from completely regular code $C$ in a distance-regular graph $\varGamma$ is to use $C$ (in some way) to determine a \emph{completely regular} partition $\Pi$ of the vertex set of $\varGamma$. In particular, each element of $\Pi$ must be a completely regular code with the same \emph{intersection numbers} as $C$. Given such a distance-regular partition $\Pi$ of $\varGamma$ we have that, by \cite[Theorem~11.1.6]{brouwer}, the quotient graph $\varGamma/\Pi$ is distance-regular. 

Suppose $C$ is a linear completely regular code in a Hamming graph $H(m,q)$ with a group $T$ of translations acting transitively on the vertex set of $H(m,q)$. Then the group $T_C$ of all translations by elements of $C$ is a normal subgroup of $T$ and hence the set of orbits of $T_C$ form a system of imprimitivity $\Pi$ for the action of $T$ on the vertex set of $H(m,q)$ and, by \cite[Theorem~11.1.6]{brouwer}, $\Pi$ is completely regular. The graph $H(m,q)/\Pi$ constructed from $C$ in this manner is called the \emph{coset graph} of $C$. Hence, all of the linear codes from Table~\ref{binaryCTtable} give rise to (known) distance-regular coset graphs in this manner. A reference for each is given in the relevant line of Table~\ref{binaryCTtable}. Note that $\G_{23}^\perp$ and $\E_{22}$ are the even weight subcodes of $\G_{23}$ and $\G_{22}$, respectively, so that their coset graphs are the bipartite doubles of the coset graphs of  $\G_{23}$ and $\G_{22}$.

The question of whether a given non-linear completely regular code can be used to construct a distance-regular graph is more complicated. By \cite[Proposition~11.1.5~(i)]{brouwer}, each part of a completely regular partition $\Pi$ involving a completely regular code $C$ must have the same size as $C$. Thus, such a construction requires that the number of vertices of $H(m,q)$ is a multiple of $|C|$. Now, the codes $\H$, $\PH$ and $\langle\L,\Delta_1\rangle\cup \langle\L,\Delta_2\rangle$ each have size divisible by an odd prime, and hence these codes do not give rise to distance-regular graphs as quotients of $H(m,2)$. 

From Table~\ref{binaryCTtable}, this leaves the code $\NR$. The Preparata codes (see \cite{baker1983preparata}) of length $2^{2t}-1$ are an infinite family of non-linear completely regular codes \cite{SemZinZai71}, of which the Nordstrom--Robinson code $\NR$ is the smallest, given by setting $t=2$. It should be noted that although $\NR$ is the unique code when $t=2$, there are several non-equivalent Preparata codes for each $t\geq 3$. If there was a linear code with the same parameters as the Preparata code of length $2^{2t}-1$ then its coset graph would have intersection array $\{2^{2t}-1,2^{2t}-2,1;1,2,2^{2t}-1\}$. Note that such a graph is a $2^{2t-1}$-antipodal cover of the complete graph ${\rm K}_{2^{2t}}$. Whilst each Preparata code is non-linear, and hence does not necessarily lead to a coset graph construction, a family of distance regular graphs with the aforementioned intersection array is constructed in \cite{de1995family}. Moreover, the graphs in \cite{de1995family} can be used to reconstruct the Preparata codes.

\section{Distinct socles: \texorpdfstring{$\alt_7$}{Alt(7)} and the Nordstrom--Robinson code}\label{sect:a7}

In this section we assume that Hypothesis~\ref{hyp1} holds, and that the socles of $X/K$ and $X_{\text{max}}/K$ are distinct. As we shall see, this occurs in precisely one case: when $X_{\text{max}}/K\cong \alt_7$ and $C$ is equivalent to the Nordstrom--Robinson code.

\begin{lemma}\label{socXnotAm}
 Assume Hypothesis~\ref{hyp1}. Then $\soc(X/K)\neq\alt_m$.
\end{lemma}

\begin{proof}
 Suppose that $\soc(X/K)=\alt_m$, in which case we have that $X^M\cong X/K\cong \alt_m$ or $\s_m$. Hypothesis~\ref{hyp1} implies that $C$ falls under part 3 of Theorem~\ref{binaryx2ntchar}. Thus, it follows from \cite[Table~1]{minimal2nt} that $X_{\b 0}^M$ does not contain $\alt_m$. By Proposition~\ref{ihom}, we have that $X_{\b 0}^M$ is $2$-homogeneous, and hence also primitive. Thus, by \cite[Theorem~14.2]{wielandt2014finite}, we have that
 \[
  |\s_m:X_{\b 0}^M|\geq \left\lfloor\frac{m+1}{2}\right\rfloor!.
 \]
 Now $|\s_m:X_{\b 0}^M|\leq 2 |X^M:X_{\b 0}^M|$ and hence, by part 4 of Proposition~\ref{xmaxlemma}, we have that 
 \[
  \left\lfloor\frac{m+1}{2}\right\rfloor!\leq 2\frac{|C|}{|C_{\text{max}}|}.
 \]
 Since $C_{\text{max}}$ has dimension at least $2$, and thus size at least $4$, we deduce that
 \[
  \left\lfloor\frac{m+1}{2}\right\rfloor!\leq |C|/2.
 \]
 Since $\delta\geq 5$, the Singleton bound implies that $|C|\leq 2^{m-4}$. Combining the last two inequalities gives
 \[
  \left\lfloor\frac{m+1}{2}\right\rfloor!\leq 2^{m-5}.
 \]
 However, this does not hold for $m\geq 5$, and hence $\soc(X/K)\neq\alt_m$.
\end{proof}

\begin{lemma}\label{differentsocles}
 Assume Hypothesis~\ref{hyp1} holds, and suppose that $\soc(X_{\text{max}}/K)\neq \soc(X/K)$. Then $m=15$, $X_{\text{max}}/K\cong \alt_7$ and $X/K\cong\alt_8$.
\end{lemma}

\begin{proof}
 By Proposition~\ref{ihom}, $X_{\b 0}^M$ is $2$-homogeneous on $M$ and, by Proposition~\ref{xmaxlemma}, we have that $X_{\b 0}^M\cong X_{\text{max}}/K$. By \cite{kantor1972k}, any $2$-homogeneous but not $2$-transitive group of degree $m$ is a subgroup of $\AGaL_1(m)$, that is, has socle the additive group $\F_m^+$ of the field $\F_m$. If $X_{\b 0}^M$ is $2$-homogeneous but not $2$-transitive then, since $\soc(X/K)\neq \soc(X_{\text{max}}/K)$, it follows that the socle of $X/K\cong X^M$ is not $\F_m^+$, and hence that $X$ acts $2$-transitively on $M$. Hypothesis~\ref{hyp1} requires that $C_{\text{max}}$ has dimension at least $2$, and hence Theorem~\ref{binaryx2ntchar} implies that $\soc(X_{\b 0})$ is as in one of the lines of \cite[Table~1]{minimal2nt}. By Lemma~\ref{socXnotAm}, we have that $\soc(X/K)\neq \alt_m$. Thus, we may apply \cite[Proposition~4.4]{ef2nt}, in which case $X_{\b 0}$ is one of the groups $G$ in \cite[Table~3]{ef2nt}. Those groups $G=X_{\b 0}$ such that $\soc(X_{\b 0})$ appears in \cite[Table~1]{minimal2nt} and $G$ appears in \cite[Table~3]{ef2nt} (thus giving the corresponding $H\cong X/K$) leave us with the possibility that either:
 \begin{enumerate}
  \item $m=15$, $X_{\b 0}\cong\alt_7$ and $X/K\cong \alt_8$; 
  \item $m=23$, $X_{\b 0}\cong\Z_{23}\rtimes \Z_{11}$ and $X/K\cong \mg_{23}$; or,
  \item $m=24$, $X_{\b 0}\cong\PSL_2(23)$ and $X/K\cong \mg_{24}$.
 \end{enumerate}
 Suppose $m=23$. Then by part 4 of Proposition~\ref{xmaxlemma},
 \[
  |C|=|X^M:X_{\b 0}^M|\cdot |C_{\text{max}}|=2^7\cdot 3^2\cdot 5\cdot 7\cdot|C_{\text{max}}|.
 \]
 In particular, $|C|\geq 2^{14} |C_{\text{max}}|$. By \cite[Table~1]{minimal2nt}, $|C_{\text{max}}|\geq 2^{11}$ so that $|C|\geq 2^{25}$. However, this is greater than the total number of vertices in $H(23,2)$, giving a contradiction. Suppose $m=24$. Then by part 4 of Proposition~\ref{xmaxlemma},
 \[
  |C|=2^7\cdot 3^2\cdot 5\cdot 7\cdot|C_{\text{max}}|.
 \]
 In particular, $|C|\geq 2^{14} |C_{\text{max}}|$. By \cite[Table~1]{minimal2nt}, $|C_{\text{max}}|\geq 2^{12}$ so that $|C|\geq 2^{26}$. However, this is greater than the number of vertices in $H(24,2)$, giving a contradiction. Thus $m=15$, $X_{\text{max}}^M\cong \alt_7$ and $X^M\cong\alt_8$.
\end{proof}

This brings us to the first of our classification results.

\begin{proposition}\label{a7fifteen}
 Suppose Hypothesis~\ref{hyp1} holds, $m=15$ and $X_{\text{max}}/K\cong \alt_7$. Then $C$ is the Nordstrom--Robinson code $\NR$ in $H(15,2)$ and $\delta=5$.
\end{proposition}

\begin{proof}
 Suppose $\delta=5$. Then, by \cite[Theorem~1.1]{gillespie2012nord}, $C$ is equivalent to the Nordstrom--Robinson code. Thus $\delta\geq 6$. Now $\alt_7$ acts $2$-transitively, but not $3$-homogeneously, on $15$ points so that, by Lemma~\ref{upboundmindist}, $\delta= 6$. Thus, by \cite[Table~I]{Best78boundsfor}, $|C|\leq 128$. Also, by Theorem~\ref{binaryx2ntchar} and \cite[Table~1, line 4]{minimal2nt}, the minimal $(X,2)$-neighbour-transitive subcode of $C$ has minimum distance $8$, and hence $C$ contains codewords of weight $8$. By \cite[Theorem~2]{borges2019completely}, the set of all weight $k$ codewords form a $3$-design, so that, for some integers $\lambda_6,\lambda_8\geq 1$, we have:
 \[
  |C\cap \varGamma_6({\b 0})|=\frac{7\cdot 13}{4}\lambda_6\quad \text{and}\quad |C\cap \varGamma_8({\b 0})|=\frac{5\cdot 13}{8}\lambda_8.
 \]
 Since each of these must be integers, we have that $4$ divides $\lambda_6$ and $8$ divides $\lambda_8$. Hence $|C\cap \varGamma_6({\b 0})|\geq 91$ and $|C\cap \varGamma_8({\b 0})|\geq 65$. However, this implies that $|C|\geq 91+65$, contradicting $|C|\leq 128$.
\end{proof}

\section{The Mathieu Groups}\label{sect:mathgroups}

Throughout this section we assume that Hypothesis~\ref{hyp2} holds. In particular, now $\soc(X/K)=\soc(X_{\text{max}}/K)$. Moreover, we assume that $\soc(X/K)\cong\mg_{m}$, where $m=22$, $23$ or $24$. For the constructions of the codes of this section, and an explanation of the notation used, see Example~\ref{codesMathgps}.

We now determine the possibilities for $C_{\text{max}}$, which must be a submodule of the permutation module over $\F_2$ for $\mg_m$ and, since $2\leq\dim(C_{\text{max}})\leq m-2$, not equal to either $\langle {\b 1}\rangle$ or $\langle {\b 1}\rangle^\perp$. For $m=23$ and $24$ this is covered in the discussion at the beginning of Section~8 of \cite{Ivanov1993}, with only partial information given there for $m=22$. Code parameters can be found in \cite[Section 5.1]{borges2019completely}. If $m=24$ then $C_{\text{max}}$ is the extended binary Golay code $\G_{24}$ with parameters $[24,12,8;4]$. If $m=23$, then $C_{\text{max}}$ contains the dual code $\G_{23}^\perp$ of the binary Golay code $\G_{23}$. Hence $C_{\text{max}}$ is either $\G_{23}$, which has parameters $[23,12,7;3]$, or $\G_{23}^\perp$, which has parameters $[23,11,8;7]$. Note that if $M$ is then taken to be the points of $\pg_2(4)$ with $a$ and $b$ (see Example~\ref{codesMathgps}) adjoined then $\G_{23}$ may be obtained as the linear span of the blocks of $W_{24}$ containing the point $c$, and $\G_{23}^\perp$ may be obtained as the linear span of the blocks not containing $c$. 

Let $m=22$ and $\G_{22}$ denote the punctured binary Golay code. Then $C_{\text{max}}$ contains $\G_{22}^\perp$, which has parameters $[22,10,8;7]$, and is contained in $\G_{22}$, which has parameters $[22,12,6;3]$. Letting $M$ be the points of $\pg_2(4)$ with $a$ adjoined  (see Example~\ref{codesMathgps}), the linear span of the set of those blocks of $W_{24}$ containing both $b$ and $c$ is $\G_{22}$, and the linear span of the set of blocks containing neither $b$ nor $c$ is $\G_{22}^\perp$. The linear span of the set of blocks containing neither $b$ nor $c$ along with the vector $(1,1,\ldots,1)$ gives the even weight subcode $\E_{22}$ of $\G_{22}$, which has parameters $[22,11,6;7]$. Two further (equivalent) codes, each with parameters $[22,11,7;6]$, are obtained by taking the linear span of either the blocks containing $b$ but not $c$, or the blocks containing $c$ but not $b$, one of which is the shortened code $\S_{22}$ of $\G_{23}$. Each of $\G_{22}$, $\G_{22}^\perp$ and $\E_{22}$ are invariant under $\Aut(\mg_{22})=\Aut(W_{24})_{\{b,c\}}=\mg_{22}:2\leq X_{\b 0}^M$, whilst the outer automorphism of $\mg_{22}=\Aut(W_{24})_{b,c}$ interchanges the two $[22,11,7;6]$ codes. These are all the codes for $m=22$, since there are three non-trivial cosets of $\G_{22}^\perp$ in $\G_{22}$. Thus, we have shown the following.

\begin{lemma}\label{mathieuCmax}
 Let $m=22,23$ or $24$, let $C$ be a linear code in $H(m,2)$ with $2\leq \dim(C)\leq m-2$, let $X=\Aut(C)$ and suppose $\soc(X/K)\cong\mg_m$. Then $C=C_{\text{max}}$ is equivalent to one of following codes in $H(m,2)$.
 \begin{itemize}
  \item The $[24,12,8;4]$ extended binary Golay code $\G_{24}$.
  \item The $[23,12,7;3]$ binary Golay code $\G_{23}$.
  \item The $[23,11,8;7]$ dual $\G_{23}^\perp$ of $\G_{23}$.
  \item The $[22,12,6;3]$ punctured binary Golay code $\G_{22}$.
  \item The $[22,11,6;7]$ even weight subcode $\E_{22}$ of $\G_{22}$.
  \item The $[22,11,7;6]$ shortened binary Golay code $\S_{22}$.
  \item The $[22,10,8;7]$ dual code $\G_{22}^\perp$.
 \end{itemize}
\end{lemma}

In the remainder of this section we determine all the completely transitive codes that arise in this case, each of which is in fact a code from the previous lemma.  

\begin{lemma}\label{halfg23ct}
 Suppose $m=23$ or $24$ and Hypothesis~\ref{hyp2} holds with $\soc(X/K)\cong\mg_m$. Then $C$ is equivalent to one of following completely transitive codes in $H(m,2)$:
 \begin{itemize}
  \item The $[24,12,8;4]$ extended binary Golay code $\G_{24}$.
  \item The $[23,12,7;3]$ binary Golay code $\G_{23}$.
  \item The $[23,12,8;7]$ dual code $\G_{23}^\perp$.
 \end{itemize}
\end{lemma}

\begin{proof}
 Recall the notation of Example~\ref{codesMathgps}, in particular, $M\subseteq P\cup\{a,b,c\}$, where $P$ is the set of points of $\pg_2(4)$. Since we are assuming $\soc(X/K)=\soc(X_{\text{max}}/K)$ and we know that $\Aut(\mg_{m})\cong\mg_{m}$ for $m=23$ and $24$, we have, by Proposition~\ref{xmaxlemma} part 4, that $C=C_{\text{max}}$. Thus, by Lemma~\ref{mathieuCmax}, $C$ is either $\G_{24}$, $\G_{23}$ or $\G_{23}^\perp$. Both $\G_{24}$ and $\G_{23}$ are completely transitive, by \cite[Example on p.~199]{sole1990completely}.

 Let $C=\G_{23}^\perp$ which, as noted in the discussion preceding Lemma~\ref{mathieuCmax}, may be obtained as the linear span of the blocks of $W_{24}$ not containing $c$, where $M$ is taken to be the points of $\pg_2(4)$ with $a$ and $b$ adjoined. In particular, $\Delta\cup \{a,b\}$ is a codeword, for each hyperoval $\Delta\in\H_1$. Since $T_C$ acts transitively on $C$ and $C$ has covering radius $7$, it is enough to prove that $\mg_{23}$ acts transitively on $C_i\cap\varGamma_i({\b 0})$ for each $i=1,\ldots,\rho=7$. First, $\mg_{23}$ acts $4$-transitively on $M$, and thus transitively on $\varGamma_i({\b 0})$ for $i=1,2,3,4$.

 Consider a vertex $\alpha\in C_i\cap\varGamma_i({\b 0})$, where $5\leq i\leq 7$, identified with the subset of which it is a characteristic vector. Since $\mg_{23}$ acts $4$-transitively, and thus $2$-transitively, on $M$, we may assume that $a,b\in \alpha$. Suppose three points $p_1,p_2,p_3\in \alpha\setminus\{a,b\}$ form a triangle in $\pg_2(4)$. Then, by \cite[Theorem~6.6B]{dixon1996permutation}, there exists a hyperoval $\Delta\in\H_1$ such that $p_1,p_2,p_3\in\Delta$, and $\Delta\cup\{a,b\}\in C$. It follows that $|\alpha\cap(\Delta\cup\{a,b\})|\geq 5$, which implies that 
 \[
  d(\alpha,\Delta\cup\{a,b\})=|\alpha|+|\Delta\cup\{a,b\}|-2|\alpha\cap(\Delta\cup\{a,b\})|\leq i-2,
 \]
 contradicting $\alpha\in C_i$. Thus, all points in $\alpha\setminus\{a,b\}$ lie on a single line of $\pg_2(4)$. Now $\left(\mg_{23}\right)_{a,b}=\PSL_3(4)$ acts transitively on the set of $j$-subsets, where $j=3,4,5$, of collinear points in $\pg_2(4)$. Thus $\mg_{23}$ acts transitively on $C_i\cap\varGamma_i({\b 0})$ for $i=5,6,7$, completing the proof.
\end{proof}

 The main purpose of next result is to help in determining the complete submodule structure of the permutation module over $\F_2$ for the action of $\PSL_3(4)$ on $\pg_2(4)$. However, we include it in this section as it is required to complete the proof of Lemma~\ref{math22CT}. Note that, as mentioned in Remark~\ref{dixonandmorterrors}, the result below contradicts \cite[Theorem~6.5B (iv)]{dixon1996permutation}, which erroneously claims that $\PGaL_3(4)/\PSL_3(4)\cong \Z_6$.

\begin{lemma}\label{symdiffhyperovals}
 Let $\Delta_1$ and $\Delta_2$ be hyperovals of $\pg_2(4)$ such that $\Delta_1$ and $\Delta_2$ intersect in a triangle. Then the symmetric difference $\Delta_1+\Delta_2=\Delta_3$ is a hyperoval of $\pg_2(4)$ and each of $\Delta_1$, $\Delta_2$ and $\Delta_3$ lie in different orbits under $\PSL_3(4)$. Moreover, the action induced by $\PGaL_3(4)$ on the set $\{\H_1,\H_2,\H_3\}$ of $\PSL_3(4)$-orbits of hyperovals is that of $\s_3$ on three points.
\end{lemma}

\begin{proof}
 Let $\omega$ be a generator of the multiplicative group $\F_4^\times$, so that $\omega$ satisfies $\omega^2+\omega+1=0$, and identify $M$ with the set $\{(x,y,z)\mid x,y,z\in\F_4, (x,y,z)\neq (0,0,0)\}$ with the understanding that two triples $(x_1,x_2,x_3)$ and $(y_1,y_2,y_3)$ represent the same point of $M$ whenever there exists some $\lambda\in \F_4$ such that $x_i=\lambda y_i$ for all $i=1,2,3$. Moreover, let
 \begin{align*}
  &\Delta_1=\{(1,0,0),(0,1,0),(0,0,1),(1,1,1),(1,\omega,\omega^2),(1,\omega^2,\omega)\}\quad\text{and}\\
  &\Delta_2=\{(1,0,0),(0,1,0),(0,0,1),(1,1,\omega),(1,\omega,1),(\omega,1,1)\},
 \end{align*}
 so that
 \[
  \Delta_3=\{(1,1,1),(1,\omega,\omega^2),(1,\omega^2,\omega),(1,1,\omega),(1,\omega,1),(\omega,1,1)\}.
 \]
 Then $\Delta_1$ is indeed a hyperoval since it has the standard form 
 \[
  \{(1,t,t^2)\mid t\in\F_4\}\cup \{(0,1,0),(0,0,1)\},
 \]
 made up of a conic and its nucleus (see, for instance, \cite{hirschfeld1998projective}). Also, $\Delta_2$ is a hyperoval since it is the image of $\Delta_1$ under the diagonal matrix with non-zero entries $1,1,\omega$, corresponding to an element of $\PGL_3(4)\setminus \PSL_3(4)$. To see that $\Delta_3$ is a hyperoval, consider the quadrangle
 \[
  \Xi=\{(1,1,1),(1,1,\omega),(1,\omega,1),(\omega,1,1)\}.
 \]
 The line through $(1,1,1),(1,1,\omega)$ meets the line through $(1,\omega,1),(\omega,1,1)$ at the point $(1,1,0)$, the line through $(1,1,1),(1,\omega,1)$ meets the line through $(1,1,\omega),(\omega,1,1)$ at the point $(1,0,1)$, and the line through $(1,1,1),(\omega,1,1)$ meets the line through $(1,1,\omega),(1,\omega,1)$ at the point $(0,1,1)$. These three points, $(1,1,0),(1,0,1),(0,1,1)$ all lie on the same line, and the remaining two points on this line, namely $(1,\omega,\omega^2)$ and $(1,\omega^2,\omega)$ are the only points of $\pg_2(4)$ not on one of the six lines through a pair of points of $\Xi$, and thus are the remaining points on the (unique) hyperoval containing $\Xi$. Thus $\Delta_3$ is a hyperoval. Now, $\Delta_1$, $\Delta_2$ and $\Delta_3$ all pairwise intersect in a triangle, and hence, by \cite[Theorem~6.6B]{dixon1996permutation}, are all in different $\PSL_3(4)$-orbits.

 Let
 \[
  \Delta_4=
  \{(1,0,0),(0,1,0),(0,0,1),(1,1,\omega^2),(1,\omega^2,1),(\omega^2,1,1)\}.
 \]
 Then, again by \cite[Theorem~6.6B]{dixon1996permutation}, $\Delta_1,\Delta_2,\Delta_4$ are in different $\PSL_3(4)$-orbits. The diagonal matrix with non-zero entries $1,1,\omega$ then maps
 \[
  \Delta_1\mapsto\Delta_2\mapsto\Delta_4\mapsto\Delta_1.
 \]
 The Frobenius automorphism fixes $\Delta_1$ and interchanges $\Delta_2$ and $\Delta_4$, completing the proof.
\end{proof}

We are now in a position to prove the last result of this section.

\begin{lemma}\label{math22CT}
 Suppose $m=22$ and Hypothesis~\ref{hyp2} holds with $\soc(X/K)\cong\mg_{22}$. Then $C$ is equivalent to one of following completely transitive codes in $H(22,2)$.
 \begin{itemize}
  \item The $[22,12,6;3]$ punctured binary Golay code $\G_{22}$.
  \item The $[22,11,6;7]$ even weight subcode $\E_{22}$.
  \item The $[22,11,7;6]$ shortened binary Golay code $\S_{22}$.
 \end{itemize}
\end{lemma}

\begin{proof}
 Since $\mg_{22}$ has index $2$ in $\Aut(\mg_{22})$ we have, by part 4 of Proposition~\ref{xmaxlemma} and part 3 of Lemma~\ref{linearcompletion}, that $C=C_{\text{max}}$. The possibilities for $C_{\text{max}}$ are given in Lemma~\ref{mathieuCmax}. First, $\G_{22}^\perp$ has $11$ non-zero weights, and thus $C$ has external distance $11$, but covering radius $7$. By \cite[Theorem~4.1]{sole1990completely}, the covering radius of a completely regular code must equal its external distance, therefore $\G_{22}^\perp$ is not completely regular, and hence not completely transitive. Thus $C$ is $\G_{22}$, $\E_{22}$ or $\S_{22}$. 
 
 As each of the codes $\G_{22}$, $\E_{22}$ or $\S_{22}$ are linear, Proposition~\ref{xmaxlemma} implies that $X=T_C\rtimes X_{\b 0}$, and in each case $T_C\leq X$ acts transitively on $C$. Thus, to show complete transitivity, it suffices to show that $X_{\b 0}$ acts transitively on $C_i\cap \varGamma_i({\b 0})$ for $i=1,\ldots,\rho$. Moreover, $\mg_{22}\leq X_{\b 0}^M$ acts $3$-transitively on $M$, and thus $X_{\b 0}$ acts transitively on $\varGamma_i({\b 0})$, for each $i=1,2,3$. Since $\G_{22}$ has covering radius $\rho=3$, $\G_{22}$ is completely transitive. Recall that $M$ here is the set of all points of $\pg_2(4)$ with the extra point $a$ adjoined, as in Example~\ref{codesMathgps}.

 Let $C=\S_{22}$. Then $C$ contains the codeword $\Delta\cup \{a\}$ for each hyperoval $\Delta\in\H_1\cup\H_3$. Consider a vertex $\alpha\in C_i\cap\varGamma_i({\b 0})$, where $i=4,5$ or $6$, identified with the corresponding subset of $M$. Since $\mg_{22}$ acts transitively on $M$, we may assume that $a\in\alpha$. Suppose three points $p_1,p_2,p_3\in \alpha\setminus\{a\}$ form a triangle in $\pg_2(4)$. Then, by \cite[Theorem~6.6B]{dixon1996permutation}, there exists a hyperoval $\Delta\in\H_1$ such that $p_1,p_2,p_3\in\Delta$. It follows that $|\alpha\cap(\Delta\cup\{a\})|\geq 4$, which implies that 
 \[
  d(\alpha,\Delta\cup\{a\})=|\alpha|+|\Delta\cup\{a\}|-2|\alpha\cap(\Delta\cup\{a\})|\leq i-1,
 \]
 contradicting $\alpha\in C_i$. Thus, all points in $\alpha\setminus\{a\}$ lie on a single line of $\pg_2(4)$. Now $\left(\mg_{22}\right)_{a}=\PSL_3(4)$ acts transitively on the set of $j$-subsets, where $j=3,4,5$, of collinear points in $\pg_2(4)$. Thus $\mg_{22}$ acts transitively on $C_i\cap\varGamma_i({\b 0})$ for $i=4,5,6$, and $C$ is completely transitive.

 Let $C=\E_{22}$. Then, $C$ contains the codewords $\ell\cup \{a\}$ for each line $\ell$ of $\pg_2(4)$ as well as the codewords $\Delta$ for each hyperoval $\Delta\in\H_2$. Note that here $X_{\b 0}$ is isomorphic to $\mg_{22}:2=\Aut(\mg_{22})$ and $\PSiL_3(4)$ is isomorphic to the stabiliser $X_{{\b 0},a}$ of $a$ inside $X_{\b 0}^M$. In particular, $X_{{\b 0},a}$ is an index $3$ subgroup of the group $\PGaL_3(4)$ appearing in Lemma~\ref{symdiffhyperovals}. Consider a vertex $\alpha\in C_i\cap\varGamma_i({\b 0})$, where $i=4,5,6$ or $7$, identified with the corresponding subset of $M$. Since $X_{\b 0}$ acts transitively on $M$, we may assume that $a\in\alpha$. Suppose three distinct points $p_1,p_2,p_3\in \alpha\setminus\{a\}$ lie on a line $\ell$ in $\pg_2(4)$. Since $p_1,p_2,p_3,a\in \alpha\cap(\ell\cup\{a\})$, it follows that $|\alpha\cap(\ell\cup\{a\})|\geq 4$, which implies that
 \[
  d(\alpha,\ell\cup\{a\})=|\alpha|+|\ell\cup\{a\}|-2|\alpha\cap(\ell\cup\{a\})|\leq i-2,
 \]
 contradicting $\alpha\in C_i$. Thus, no subset of $3$ distinct points in $\alpha\setminus\{a\}$ are collinear in $\pg_2(4)$. It follows that $\alpha\setminus\{a\}$ is contained in some hyperoval $\Delta$ in $\H_1,\H_2$ or $\H_3$. Suppose $i=4$. The group $X_{{\b 0},a}\cong\PSiL_3(4)$ acts transitively on the set of triangles of $\pg_2(4)$, and hence $X_{\b 0}$ acts transitively on $C_4\cap\varGamma_4({\b 0})$. Suppose $i=5,6$ or $7$. If $\alpha\setminus\{a\}\subseteq\Delta\in\H_2$ then, since $\H_2\subseteq C$, we have that $d(\alpha,\Delta)\leq 3$, contradicting $\alpha\in C_i$. Hence, $\Delta\in\H_1$ or $\H_3$. By Lemma~\ref{symdiffhyperovals}, there exists an element (corresponding to the field automorphism in the proof of Lemma~\ref{symdiffhyperovals}) of $X_{{\b 0},a}$ that interchanges $\H_1$ and $\H_3$ (and also interchanges $b$ and $c$). Moreover, by \cite[Exercise~6.5.13]{dixon1996permutation}, the stabiliser of $\Delta$ inside this copy of $\PSiL_3(4)$ acts as the symmetric group $\s_6$ on the points of $\Delta$. Thus, we have that $X_{\b 0}$ acts transitively on $C_i\cap \varGamma_i({\b 0})$ for each $i=5,6,7$. Hence $C$ is completely transitive.
\end{proof}

\section{\texorpdfstring{$\PSL_3(4)$}{PSL(3,4)}}\label{sect:psl34}

In this section we shall consider codes satisfying Hypothesis~\ref{hyp2}, where $M$ is the point-set of the projective plane $\pg_2(4)$ and where $\soc(X/K)\cong\PSL_3(4)$. See Example~\ref{pslcodesex} for the constructions of the codes arising in this section.

\begin{lemma}\label{pslmaxminsubmodules}
 Assume Hypothesis~\ref{hyp2} holds, $m=21$ and $\soc(X/K)\cong\PSL_3(4)$. Then $C_{\text{max}}$ contains the code $\P$, which has parameters $[21,9,8;7]$, and is contained in the code $\P^\perp$, which has parameters $[21,12,5;3]$.
\end{lemma}

\begin{proof}
 By \cite[Section~8]{Ivanov1993}, the unique minimal submodule of the permutation module over $\F_2$ for the action of $\PSL_3(4)$ on $\pg_2(4)$, which we identify here with the vertex set of $H(21,2)$, is $\P$, that is, the code generated by the set of all complements of lines in $\pg_2(4)$. Now, for any two distinct lines $\ell_1,\ell_2$ the sum of their complements is $(\ell_1+M)+(\ell_2+M)=\ell_1+\ell_2$. Thus, it follows that $\P$ contains the code generated by the symmetric difference $\ell_1+\ell_2$ of all pairs of lines $\ell_1,\ell_2$ of $\pg_2(4)$. Moreover, if $\ell_1,\ell_2,\ell_3,\ell_4$ are pairwise distinct lines in $\pg_2(4)$ all through a common point $p$, then $(\ell_1+\ell_2)+(\ell_3+\ell_4)=\ell_5+M$, where $\ell_5$ is the remaining line through $p$. As the complement of any line may be formed in this way, $\P$ is in fact generated by the set of all symmetric differences $\ell_1+\ell_2$ of pairs $\ell_1,\ell_2$ of lines of $\pg_2(4)$. It follows that $\P$ is generated by only those blocks appearing under case 8 of Example~\ref{codesMathgps} and therefore $\P$ is the shortened code of the $[22,10,8;7]$ code $\G_{22}^\perp$ (see Lemma~\ref{mathieuCmax}), and thus has parameters $[21,9,8;7]$. Since Hypothesis~\ref{hyp2} assumes that $C_{\text{max}}$ has codimension at least $2$ in $V\varGamma\cong\F_2^m$, it follows that $C_{\text{max}}$ is contained in $\P^\perp$, which is the punctured code of the $[22,12,6;3]$ code $\G_{22}$ (again, see Lemma~\ref{mathieuCmax}), and so $\P^\perp$ has parameters $[21,12,5;3]$.
\end{proof}

The following lemma demonstrates an important relationship between the cosets of $\P$ in $\P^\perp$. Note that the sum in the conclusion should be interpreted in terms of characteristic vectors and is equivalent to the symmetric difference of the subsets of points, and by a \emph{quadrangle} we mean a set of four points with the property that no $3$ are collinear.

\begin{lemma}\label{sumoflineandFanoplane}
 Let $\ell$ be a line of $\pg_2(4)$ and, for each $i=1,2,3$, let $\Delta_i$ be a hyperoval in $\H_i$ and $\Phi_i$ be a Fano subplane in $\mathcal{F}_i$, as in Example~\ref{pslcodesex}. Then $\ell$, $\Delta_i$ and $\Phi_i$ may be chosen so that $\Delta_i=\ell+\Phi_i$.
\end{lemma}

\begin{proof}
 Let $\omega$ be a generator of the multiplicative group $\F_4^\times$, so that $\omega$ satisfies $\omega^2+\omega+1=0$, and identify $M$ with the set $\{(x,y,z)\mid x,y,z\in\F_4, (x,y,z)\neq (0,0,0)\}$ with the understanding that two triples $(x_1,x_2,x_3)$ and $(y_1,y_2,y_3)$ represent the same point of $M$ whenever there exists some $c\in \F_4$ such that $x_i=cy_i$ for all $i=1,2,3$. Let 
 \[
  \ell=\{ (1,0,0),(0,1,0),(1,1,0),(1,\omega,0),(1,\omega^2,0)\},
 \]
 and let 
 \[
  \Phi_1=\{(1,0,0),(0,1,0),(0,0,1),(0,1,1),(1,0,1),(1,1,0),(1,1,1)\}.
 \]
 Consider the subset 
 \[
  \ell+\Phi_1=\{(1,\omega,0),(1,\omega^2,0),(0,0,1),(0,1,1),(1,0,1),(1,1,1)\},
 \]
 of $M$. Now the four points $(0,0,1),(0,1,1),(1,0,1),(1,1,1)$ form a quadrangle $\Xi$, and considering each pair of points in the quadrangle gives six lines. These six lines contain all the points of $M$ except for $(1,\omega,0)$ and $(1,\omega^2,0)$, and hence $\ell+\Phi_1$ is the unique hyperoval containing the quadrangle $\Xi$. Moreover, since $\Phi_1$ and $\ell+\Phi_1$ intersect precisely in $\Xi$, it follows that $\ell+\Phi_1$ is a hyperoval in $\H_1$ (see Example~\ref{pslcodesex}). Thus we may let $\Delta_1=\ell+\Phi_1$. Let $\tau$ be the permutation on $M$ defined by $\tau:(x,y,z)\mapsto (x,y,z\omega)$. Since $\tau$ corresponds to a linear transformation with determinant $\omega$, $\tau$ corresponds to an element of $\PGL_3(4)\setminus\PSL_3(4)$. Now, by Lemma~\ref{symdiffhyperovals}, the quotient group $\PGaL_3(4)/\PSL_3(4)$ acts as $\s_3$ on the set $\{\H_1,\H_2,\H_3\}$, from which it follows that $\tau$ induces a $3$-cycle on each of the sets $\{\H_1,\H_2,\H_3\}$ and $\{\mathcal{F}_1,\mathcal{F}_2,\mathcal{F}_3\}$. Note that $\tau$ also fixes the line $\ell$. Hence, we may choose $\Phi_2=\Phi_1^\tau$, $\Phi_3=\Phi_1^{\tau^2}$, $\Delta_2=\Delta_1^\tau$ and $\Delta_3=\Delta_1^{\tau^2}$ so that $\Delta_i=\ell+\Phi_i$ for each $i=1,2,3$. This completes the proof.
\end{proof}

\begin{figure}
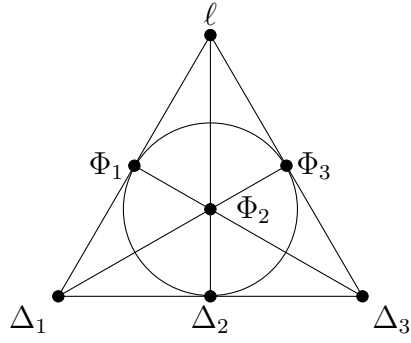

 \[
  \FanoPlane[4cm]
 \]
 \caption{Arrangement of the non-trivial cosets of $\P$ in $\P^\perp$, labelled by their representatives as in the proof of Lemma~\ref{psl34cmax}.}
 \label{fanoconfig}
\end{figure}

We now determine the structure of the submodule lattice, which in turn allows us to determine all the possibilities for $C_{\text{max}}$.

\begin{lemma}\label{submodulelattice}
 The submodule lattice of the permutation module over $\F_2$ for the action of $\PSL_3(4)$ on $\pg_2(4)$, identified with $V\varGamma$, is as in Figure~\ref{pslsubmodlattdiag}.
\end{lemma}

\begin{proof}
 By Lemma~\ref{pslmaxminsubmodules}, $\P$ and $\P^\perp$ are, respectively, the unique minimal and unique maximal non-trivial submodules. For $i=1,2,3$, let $\ell,\Delta_i,\Phi_i$ be as in Example~\ref{pslcodesex}, chosen as in Lemma~\ref{sumoflineandFanoplane}, so that $\Delta_i=\ell+\Phi_i$. We claim that the Fano plane in Figure~\ref{fanoconfig} represents the configuration of the non-trivial cosets of $\P$ in $\P^\perp$, where a vertex with label $\alpha$ represents the coset $\P+\alpha$. First, $\P$ has codimension $3$ in $\P^\perp$, and hence the non-trivial cosets of $\P$ may be represented by the points of a Fano plane. Also, $\P$ has minimum distance $8$ and so, for $\alpha=\ell,\Delta_i,\Phi_i$ with $i=1,2,3$, we see that $\P$ does not contain $\alpha$, and hence $\P+\alpha\neq \P$. Note that since $\P^\perp$ has covering radius $3$,  $\P^\perp$ does indeed contain each of $\ell,\Delta_i,\Phi_i$ for $i=1,2,3$. Now, $\Delta_i=\ell+\Phi_i$ for each $i=1,2,3$. Moreover, since each coset of $\P$ also has minimum distance $8$ we deduce that each of $\P+\ell$, $\P+\Delta_i$ and $\P+\Phi_i$ are pairwise distinct, for $i=1,2,3$, giving the three lines through $\ell$. By Lemma~\ref{symdiffhyperovals}, $\P+\Delta_1$, $\P+\Delta_2$ and $\P+\Delta_3$ are distinct cosets and lie on a common line, which then forces the remainder of the configuration. Thus our claim holds. 
 
 Each module in the submodule lattice, as in Figure~\ref{pslsubmodlattdiag}, may then be deduced from Figure~\ref{fanoconfig}: the points of the Fano plane correspond to the codes $\L$, $\langle\P,\Delta_i\rangle$, $\langle\P,\Phi_i\rangle$, for $i=1,2,3$ and the lines of the Fano plane correspond to the codes $\langle\P,\Delta_1,\Delta_2,\Delta_3\rangle$, $\langle\P,\ell,\Delta_i,\Phi_i\rangle$ and $\langle\P,\Delta_i,\Phi_{i+1},\Phi_{i+2}\rangle$, for $i=1,2,3$ and with subscripts modulo $3$. Now, all codewords of $\P$ have even weight, and hence $\P$ does not contain the repetition code $\langle{\b 1}\rangle$, whilst set $M$ is the sum of the five lines through a given point, and hence $\L$ does contain $\langle{\b 1}\rangle$. The remaining inclusions in Figure~\ref{pslsubmodlattdiag} follow from the incidences of Figure~\ref{fanoconfig}, completing the proof.
\end{proof}

\begin{figure}
 \centering
\begin{tikzpicture}[x=0.75pt,y=0.75pt,yscale=-0.75,xscale=0.85]

\draw    (525,85) -- (525,155) ;
\draw [shift={(525,155)}, rotate = 90] [color={rgb, 255:red, 0; green, 0; blue, 0 }  ][fill={rgb, 255:red, 0; green, 0; blue, 0 }  ][line width=0.75]      (0, 0) circle [x radius= 3.35, y radius= 3.35]   ;
\draw [shift={(525,85)}, rotate = 90] [color={rgb, 255:red, 0; green, 0; blue, 0 }  ][fill={rgb, 255:red, 0; green, 0; blue, 0 }  ][line width=0.75]      (0, 0) circle [x radius= 3.35, y radius= 3.35]   ;
\draw    (585,115) -- (585,185) ;
\draw [shift={(585,185)}, rotate = 90] [color={rgb, 255:red, 0; green, 0; blue, 0 }  ][fill={rgb, 255:red, 0; green, 0; blue, 0 }  ][line width=0.75]      (0, 0) circle [x radius= 3.35, y radius= 3.35]   ;
\draw [shift={(585,115)}, rotate = 90] [color={rgb, 255:red, 0; green, 0; blue, 0 }  ][fill={rgb, 255:red, 0; green, 0; blue, 0 }  ][line width=0.75]      (0, 0) circle [x radius= 3.35, y radius= 3.35]   ;
\draw    (585,460) -- (585,530) ;
\draw [shift={(585,530)}, rotate = 90] [color={rgb, 255:red, 0; green, 0; blue, 0 }  ][fill={rgb, 255:red, 0; green, 0; blue, 0 }  ][line width=0.75]      (0, 0) circle [x radius= 3.35, y radius= 3.35]   ;
\draw [shift={(585,460)}, rotate = 90] [color={rgb, 255:red, 0; green, 0; blue, 0 }  ][fill={rgb, 255:red, 0; green, 0; blue, 0 }  ][line width=0.75]      (0, 0) circle [x radius= 3.35, y radius= 3.35]   ;
\draw    (525,85) -- (585,115) ;
\draw [shift={(585,115)}, rotate = 26.57] [color={rgb, 255:red, 0; green, 0; blue, 0 }  ][fill={rgb, 255:red, 0; green, 0; blue, 0 }  ][line width=0.75]      (0, 0) circle [x radius= 3.35, y radius= 3.35]   ;
\draw [shift={(525,85)}, rotate = 26.57] [color={rgb, 255:red, 0; green, 0; blue, 0 }  ][fill={rgb, 255:red, 0; green, 0; blue, 0 }  ][line width=0.75]      (0, 0) circle [x radius= 3.35, y radius= 3.35]   ;
\draw    (525,155) -- (585,185) ;
\draw [shift={(585,185)}, rotate = 26.57] [color={rgb, 255:red, 0; green, 0; blue, 0 }  ][fill={rgb, 255:red, 0; green, 0; blue, 0 }  ][line width=0.75]      (0, 0) circle [x radius= 3.35, y radius= 3.35]   ;
\draw [shift={(525,155)}, rotate = 26.57] [color={rgb, 255:red, 0; green, 0; blue, 0 }  ][fill={rgb, 255:red, 0; green, 0; blue, 0 }  ][line width=0.75]      (0, 0) circle [x radius= 3.35, y radius= 3.35]   ;
\draw    (450,180) -- (525,155) ;
\draw [shift={(525,155)}, rotate = 341.57] [color={rgb, 255:red, 0; green, 0; blue, 0 }  ][fill={rgb, 255:red, 0; green, 0; blue, 0 }  ][line width=0.75]      (0, 0) circle [x radius= 3.35, y radius= 3.35]   ;
\draw [shift={(450,180)}, rotate = 341.57] [color={rgb, 255:red, 0; green, 0; blue, 0 }  ][fill={rgb, 255:red, 0; green, 0; blue, 0 }  ][line width=0.75]      (0, 0) circle [x radius= 3.35, y radius= 3.35]   ;
\draw    (660,435) -- (585,460) ;
\draw [shift={(585,460)}, rotate = 161.57] [color={rgb, 255:red, 0; green, 0; blue, 0 }  ][fill={rgb, 255:red, 0; green, 0; blue, 0 }  ][line width=0.75]      (0, 0) circle [x radius= 3.35, y radius= 3.35]   ;
\draw [shift={(660,435)}, rotate = 161.57] [color={rgb, 255:red, 0; green, 0; blue, 0 }  ][fill={rgb, 255:red, 0; green, 0; blue, 0 }  ][line width=0.75]      (0, 0) circle [x radius= 3.35, y radius= 3.35]   ;
\draw    (450,240) -- (525,155) ;
\draw [shift={(525,155)}, rotate = 311.42] [color={rgb, 255:red, 0; green, 0; blue, 0 }  ][fill={rgb, 255:red, 0; green, 0; blue, 0 }  ][line width=0.75]      (0, 0) circle [x radius= 3.35, y radius= 3.35]   ;
\draw [shift={(450,240)}, rotate = 311.42] [color={rgb, 255:red, 0; green, 0; blue, 0 }  ][fill={rgb, 255:red, 0; green, 0; blue, 0 }  ][line width=0.75]      (0, 0) circle [x radius= 3.35, y radius= 3.35]   ;
\draw    (450,210) -- (525,155) ;
\draw [shift={(525,155)}, rotate = 323.75] [color={rgb, 255:red, 0; green, 0; blue, 0 }  ][fill={rgb, 255:red, 0; green, 0; blue, 0 }  ][line width=0.75]      (0, 0) circle [x radius= 3.35, y radius= 3.35]   ;
\draw [shift={(450,210)}, rotate = 323.75] [color={rgb, 255:red, 0; green, 0; blue, 0 }  ][fill={rgb, 255:red, 0; green, 0; blue, 0 }  ][line width=0.75]      (0, 0) circle [x radius= 3.35, y radius= 3.35]   ;
\draw    (450,300) -- (525,155) ;
\draw [shift={(525,155)}, rotate = 297.35] [color={rgb, 255:red, 0; green, 0; blue, 0 }  ][fill={rgb, 255:red, 0; green, 0; blue, 0 }  ][line width=0.75]      (0, 0) circle [x radius= 3.35, y radius= 3.35]   ;
\draw [shift={(450,300)}, rotate = 297.35] [color={rgb, 255:red, 0; green, 0; blue, 0 }  ][fill={rgb, 255:red, 0; green, 0; blue, 0 }  ][line width=0.75]      (0, 0) circle [x radius= 3.35, y radius= 3.35]   ;
\draw    (450,270) -- (525,155) ;
\draw [shift={(525,155)}, rotate = 303.11] [color={rgb, 255:red, 0; green, 0; blue, 0 }  ][fill={rgb, 255:red, 0; green, 0; blue, 0 }  ][line width=0.75]      (0, 0) circle [x radius= 3.35, y radius= 3.35]   ;
\draw [shift={(450,270)}, rotate = 303.11] [color={rgb, 255:red, 0; green, 0; blue, 0 }  ][fill={rgb, 255:red, 0; green, 0; blue, 0 }  ][line width=0.75]      (0, 0) circle [x radius= 3.35, y radius= 3.35]   ;
\draw    (450,330) -- (525,155) ;
\draw [shift={(525,155)}, rotate = 293.2] [color={rgb, 255:red, 0; green, 0; blue, 0 }  ][fill={rgb, 255:red, 0; green, 0; blue, 0 }  ][line width=0.75]      (0, 0) circle [x radius= 3.35, y radius= 3.35]   ;
\draw [shift={(450,330)}, rotate = 293.2] [color={rgb, 255:red, 0; green, 0; blue, 0 }  ][fill={rgb, 255:red, 0; green, 0; blue, 0 }  ][line width=0.75]      (0, 0) circle [x radius= 3.35, y radius= 3.35]   ;
\draw    (660,375) -- (585,460) ;
\draw [shift={(585,460)}, rotate = 131.42] [color={rgb, 255:red, 0; green, 0; blue, 0 }  ][fill={rgb, 255:red, 0; green, 0; blue, 0 }  ][line width=0.75]      (0, 0) circle [x radius= 3.35, y radius= 3.35]   ;
\draw [shift={(660,375)}, rotate = 131.42] [color={rgb, 255:red, 0; green, 0; blue, 0 }  ][fill={rgb, 255:red, 0; green, 0; blue, 0 }  ][line width=0.75]      (0, 0) circle [x radius= 3.35, y radius= 3.35]   ;
\draw    (660,405) -- (585,460) ;
\draw [shift={(585,460)}, rotate = 143.75] [color={rgb, 255:red, 0; green, 0; blue, 0 }  ][fill={rgb, 255:red, 0; green, 0; blue, 0 }  ][line width=0.75]      (0, 0) circle [x radius= 3.35, y radius= 3.35]   ;
\draw [shift={(660,405)}, rotate = 143.75] [color={rgb, 255:red, 0; green, 0; blue, 0 }  ][fill={rgb, 255:red, 0; green, 0; blue, 0 }  ][line width=0.75]      (0, 0) circle [x radius= 3.35, y radius= 3.35]   ;
\draw    (660,315) -- (585,460) ;
\draw [shift={(585,460)}, rotate = 117.35] [color={rgb, 255:red, 0; green, 0; blue, 0 }  ][fill={rgb, 255:red, 0; green, 0; blue, 0 }  ][line width=0.75]      (0, 0) circle [x radius= 3.35, y radius= 3.35]   ;
\draw [shift={(660,315)}, rotate = 117.35] [color={rgb, 255:red, 0; green, 0; blue, 0 }  ][fill={rgb, 255:red, 0; green, 0; blue, 0 }  ][line width=0.75]      (0, 0) circle [x radius= 3.35, y radius= 3.35]   ;
\draw    (660,345) -- (585,460) ;
\draw [shift={(585,460)}, rotate = 123.11] [color={rgb, 255:red, 0; green, 0; blue, 0 }  ][fill={rgb, 255:red, 0; green, 0; blue, 0 }  ][line width=0.75]      (0, 0) circle [x radius= 3.35, y radius= 3.35]   ;
\draw [shift={(660,345)}, rotate = 123.11] [color={rgb, 255:red, 0; green, 0; blue, 0 }  ][fill={rgb, 255:red, 0; green, 0; blue, 0 }  ][line width=0.75]      (0, 0) circle [x radius= 3.35, y radius= 3.35]   ;
\draw    (660,285) -- (585,460) ;
\draw [shift={(585,460)}, rotate = 113.2] [color={rgb, 255:red, 0; green, 0; blue, 0 }  ][fill={rgb, 255:red, 0; green, 0; blue, 0 }  ][line width=0.75]      (0, 0) circle [x radius= 3.35, y radius= 3.35]   ;
\draw [shift={(660,285)}, rotate = 113.2] [color={rgb, 255:red, 0; green, 0; blue, 0 }  ][fill={rgb, 255:red, 0; green, 0; blue, 0 }  ][line width=0.75]      (0, 0) circle [x radius= 3.35, y radius= 3.35]   ;
\draw    (525,430) -- (525,500) ;
\draw [shift={(525,500)}, rotate = 90] [color={rgb, 255:red, 0; green, 0; blue, 0 }  ][fill={rgb, 255:red, 0; green, 0; blue, 0 }  ][line width=0.75]      (0, 0) circle [x radius= 3.35, y radius= 3.35]   ;
\draw [shift={(525,430)}, rotate = 90] [color={rgb, 255:red, 0; green, 0; blue, 0 }  ][fill={rgb, 255:red, 0; green, 0; blue, 0 }  ][line width=0.75]      (0, 0) circle [x radius= 3.35, y radius= 3.35]   ;
\draw    (525,430) -- (585,460) ;
\draw [shift={(585,460)}, rotate = 26.57] [color={rgb, 255:red, 0; green, 0; blue, 0 }  ][fill={rgb, 255:red, 0; green, 0; blue, 0 }  ][line width=0.75]      (0, 0) circle [x radius= 3.35, y radius= 3.35]   ;
\draw [shift={(525,430)}, rotate = 26.57] [color={rgb, 255:red, 0; green, 0; blue, 0 }  ][fill={rgb, 255:red, 0; green, 0; blue, 0 }  ][line width=0.75]      (0, 0) circle [x radius= 3.35, y radius= 3.35]   ;
\draw    (525,500) -- (585,530) ;
\draw [shift={(585,530)}, rotate = 26.57] [color={rgb, 255:red, 0; green, 0; blue, 0 }  ][fill={rgb, 255:red, 0; green, 0; blue, 0 }  ][line width=0.75]      (0, 0) circle [x radius= 3.35, y radius= 3.35]   ;
\draw [shift={(525,500)}, rotate = 26.57] [color={rgb, 255:red, 0; green, 0; blue, 0 }  ][fill={rgb, 255:red, 0; green, 0; blue, 0 }  ][line width=0.75]      (0, 0) circle [x radius= 3.35, y radius= 3.35]   ;
\draw    (660,285) -- (450,180) ;
\draw [shift={(450,180)}, rotate = 206.57] [color={rgb, 255:red, 0; green, 0; blue, 0 }  ][fill={rgb, 255:red, 0; green, 0; blue, 0 }  ][line width=0.75]      (0, 0) circle [x radius= 3.35, y radius= 3.35]   ;
\draw [shift={(660,285)}, rotate = 206.57] [color={rgb, 255:red, 0; green, 0; blue, 0 }  ][fill={rgb, 255:red, 0; green, 0; blue, 0 }  ][line width=0.75]      (0, 0) circle [x radius= 3.35, y radius= 3.35]   ;
\draw    (660,315) -- (450,210) ;
\draw [shift={(450,210)}, rotate = 206.57] [color={rgb, 255:red, 0; green, 0; blue, 0 }  ][fill={rgb, 255:red, 0; green, 0; blue, 0 }  ][line width=0.75]      (0, 0) circle [x radius= 3.35, y radius= 3.35]   ;
\draw [shift={(660,315)}, rotate = 206.57] [color={rgb, 255:red, 0; green, 0; blue, 0 }  ][fill={rgb, 255:red, 0; green, 0; blue, 0 }  ][line width=0.75]      (0, 0) circle [x radius= 3.35, y radius= 3.35]   ;
\draw    (660,345) -- (450,240) ;
\draw [shift={(450,240)}, rotate = 206.57] [color={rgb, 255:red, 0; green, 0; blue, 0 }  ][fill={rgb, 255:red, 0; green, 0; blue, 0 }  ][line width=0.75]      (0, 0) circle [x radius= 3.35, y radius= 3.35]   ;
\draw [shift={(660,345)}, rotate = 206.57] [color={rgb, 255:red, 0; green, 0; blue, 0 }  ][fill={rgb, 255:red, 0; green, 0; blue, 0 }  ][line width=0.75]      (0, 0) circle [x radius= 3.35, y radius= 3.35]   ;
\draw    (660,405) -- (450,270) ;
\draw [shift={(450,270)}, rotate = 212.74] [color={rgb, 255:red, 0; green, 0; blue, 0 }  ][fill={rgb, 255:red, 0; green, 0; blue, 0 }  ][line width=0.75]      (0, 0) circle [x radius= 3.35, y radius= 3.35]   ;
\draw [shift={(660,405)}, rotate = 212.74] [color={rgb, 255:red, 0; green, 0; blue, 0 }  ][fill={rgb, 255:red, 0; green, 0; blue, 0 }  ][line width=0.75]      (0, 0) circle [x radius= 3.35, y radius= 3.35]   ;
\draw    (660,435) -- (450,300) ;
\draw [shift={(450,300)}, rotate = 212.74] [color={rgb, 255:red, 0; green, 0; blue, 0 }  ][fill={rgb, 255:red, 0; green, 0; blue, 0 }  ][line width=0.75]      (0, 0) circle [x radius= 3.35, y radius= 3.35]   ;
\draw [shift={(660,435)}, rotate = 212.74] [color={rgb, 255:red, 0; green, 0; blue, 0 }  ][fill={rgb, 255:red, 0; green, 0; blue, 0 }  ][line width=0.75]      (0, 0) circle [x radius= 3.35, y radius= 3.35]   ;
\draw    (660,375) -- (450,300) ;
\draw [shift={(450,300)}, rotate = 199.65] [color={rgb, 255:red, 0; green, 0; blue, 0 }  ][fill={rgb, 255:red, 0; green, 0; blue, 0 }  ][line width=0.75]      (0, 0) circle [x radius= 3.35, y radius= 3.35]   ;
\draw [shift={(660,375)}, rotate = 199.65] [color={rgb, 255:red, 0; green, 0; blue, 0 }  ][fill={rgb, 255:red, 0; green, 0; blue, 0 }  ][line width=0.75]      (0, 0) circle [x radius= 3.35, y radius= 3.35]   ;
\draw    (660,435) -- (450,270) ;
\draw [shift={(450,270)}, rotate = 218.16] [color={rgb, 255:red, 0; green, 0; blue, 0 }  ][fill={rgb, 255:red, 0; green, 0; blue, 0 }  ][line width=0.75]      (0, 0) circle [x radius= 3.35, y radius= 3.35]   ;
\draw [shift={(660,435)}, rotate = 218.16] [color={rgb, 255:red, 0; green, 0; blue, 0 }  ][fill={rgb, 255:red, 0; green, 0; blue, 0 }  ][line width=0.75]      (0, 0) circle [x radius= 3.35, y radius= 3.35]   ;
\draw    (660,405) -- (450,330) ;
\draw [shift={(450,330)}, rotate = 199.65] [color={rgb, 255:red, 0; green, 0; blue, 0 }  ][fill={rgb, 255:red, 0; green, 0; blue, 0 }  ][line width=0.75]      (0, 0) circle [x radius= 3.35, y radius= 3.35]   ;
\draw [shift={(660,405)}, rotate = 199.65] [color={rgb, 255:red, 0; green, 0; blue, 0 }  ][fill={rgb, 255:red, 0; green, 0; blue, 0 }  ][line width=0.75]      (0, 0) circle [x radius= 3.35, y radius= 3.35]   ;
\draw    (660,375) -- (450,330) ;
\draw [shift={(450,330)}, rotate = 192.09] [color={rgb, 255:red, 0; green, 0; blue, 0 }  ][fill={rgb, 255:red, 0; green, 0; blue, 0 }  ][line width=0.75]      (0, 0) circle [x radius= 3.35, y radius= 3.35]   ;
\draw [shift={(660,375)}, rotate = 192.09] [color={rgb, 255:red, 0; green, 0; blue, 0 }  ][fill={rgb, 255:red, 0; green, 0; blue, 0 }  ][line width=0.75]      (0, 0) circle [x radius= 3.35, y radius= 3.35]   ;
\draw    (660,315) -- (585,185) ;
\draw [shift={(585,185)}, rotate = 240.02] [color={rgb, 255:red, 0; green, 0; blue, 0 }  ][fill={rgb, 255:red, 0; green, 0; blue, 0 }  ][line width=0.75]      (0, 0) circle [x radius= 3.35, y radius= 3.35]   ;
\draw [shift={(660,315)}, rotate = 240.02] [color={rgb, 255:red, 0; green, 0; blue, 0 }  ][fill={rgb, 255:red, 0; green, 0; blue, 0 }  ][line width=0.75]      (0, 0) circle [x radius= 3.35, y radius= 3.35]   ;
\draw    (660,345) -- (585,185) ;
\draw [shift={(585,185)}, rotate = 244.89] [color={rgb, 255:red, 0; green, 0; blue, 0 }  ][fill={rgb, 255:red, 0; green, 0; blue, 0 }  ][line width=0.75]      (0, 0) circle [x radius= 3.35, y radius= 3.35]   ;
\draw [shift={(660,345)}, rotate = 244.89] [color={rgb, 255:red, 0; green, 0; blue, 0 }  ][fill={rgb, 255:red, 0; green, 0; blue, 0 }  ][line width=0.75]      (0, 0) circle [x radius= 3.35, y radius= 3.35]   ;
\draw    (660,285) -- (585,185) ;
\draw [shift={(585,185)}, rotate = 233.13] [color={rgb, 255:red, 0; green, 0; blue, 0 }  ][fill={rgb, 255:red, 0; green, 0; blue, 0 }  ][line width=0.75]      (0, 0) circle [x radius= 3.35, y radius= 3.35]   ;
\draw [shift={(660,285)}, rotate = 233.13] [color={rgb, 255:red, 0; green, 0; blue, 0 }  ][fill={rgb, 255:red, 0; green, 0; blue, 0 }  ][line width=0.75]      (0, 0) circle [x radius= 3.35, y radius= 3.35]   ;
\draw    (450,240) -- (525,430) ;
\draw [shift={(525,430)}, rotate = 68.46] [color={rgb, 255:red, 0; green, 0; blue, 0 }  ][fill={rgb, 255:red, 0; green, 0; blue, 0 }  ][line width=0.75]      (0, 0) circle [x radius= 3.35, y radius= 3.35]   ;
\draw [shift={(450,240)}, rotate = 68.46] [color={rgb, 255:red, 0; green, 0; blue, 0 }  ][fill={rgb, 255:red, 0; green, 0; blue, 0 }  ][line width=0.75]      (0, 0) circle [x radius= 3.35, y radius= 3.35]   ;
\draw    (450,180) -- (525,430) ;
\draw [shift={(525,430)}, rotate = 73.3] [color={rgb, 255:red, 0; green, 0; blue, 0 }  ][fill={rgb, 255:red, 0; green, 0; blue, 0 }  ][line width=0.75]      (0, 0) circle [x radius= 3.35, y radius= 3.35]   ;
\draw [shift={(450,180)}, rotate = 73.3] [color={rgb, 255:red, 0; green, 0; blue, 0 }  ][fill={rgb, 255:red, 0; green, 0; blue, 0 }  ][line width=0.75]      (0, 0) circle [x radius= 3.35, y radius= 3.35]   ;
\draw    (450,210) -- (525,430) ;
\draw [shift={(525,430)}, rotate = 71.18] [color={rgb, 255:red, 0; green, 0; blue, 0 }  ][fill={rgb, 255:red, 0; green, 0; blue, 0 }  ][line width=0.75]      (0, 0) circle [x radius= 3.35, y radius= 3.35]   ;
\draw [shift={(450,210)}, rotate = 71.18] [color={rgb, 255:red, 0; green, 0; blue, 0 }  ][fill={rgb, 255:red, 0; green, 0; blue, 0 }  ][line width=0.75]      (0, 0) circle [x radius= 3.35, y radius= 3.35]   ;
\draw    (660,285) -- (450,270) ;
\draw [shift={(450,270)}, rotate = 184.09] [color={rgb, 255:red, 0; green, 0; blue, 0 }  ][fill={rgb, 255:red, 0; green, 0; blue, 0 }  ][line width=0.75]      (0, 0) circle [x radius= 3.35, y radius= 3.35]   ;
\draw [shift={(660,285)}, rotate = 184.09] [color={rgb, 255:red, 0; green, 0; blue, 0 }  ][fill={rgb, 255:red, 0; green, 0; blue, 0 }  ][line width=0.75]      (0, 0) circle [x radius= 3.35, y radius= 3.35]   ;
\draw    (660,315) -- (450,300) ;
\draw [shift={(450,300)}, rotate = 184.09] [color={rgb, 255:red, 0; green, 0; blue, 0 }  ][fill={rgb, 255:red, 0; green, 0; blue, 0 }  ][line width=0.75]      (0, 0) circle [x radius= 3.35, y radius= 3.35]   ;
\draw [shift={(660,315)}, rotate = 184.09] [color={rgb, 255:red, 0; green, 0; blue, 0 }  ][fill={rgb, 255:red, 0; green, 0; blue, 0 }  ][line width=0.75]      (0, 0) circle [x radius= 3.35, y radius= 3.35]   ;
\draw    (660,345) -- (450,330) ;
\draw [shift={(450,330)}, rotate = 184.09] [color={rgb, 255:red, 0; green, 0; blue, 0 }  ][fill={rgb, 255:red, 0; green, 0; blue, 0 }  ][line width=0.75]      (0, 0) circle [x radius= 3.35, y radius= 3.35]   ;
\draw [shift={(660,345)}, rotate = 184.09] [color={rgb, 255:red, 0; green, 0; blue, 0 }  ][fill={rgb, 255:red, 0; green, 0; blue, 0 }  ][line width=0.75]      (0, 0) circle [x radius= 3.35, y radius= 3.35]   ;
\draw    (660,375) -- (450,180) ;
\draw [shift={(450,180)}, rotate = 222.88] [color={rgb, 255:red, 0; green, 0; blue, 0 }  ][fill={rgb, 255:red, 0; green, 0; blue, 0 }  ][line width=0.75]      (0, 0) circle [x radius= 3.35, y radius= 3.35]   ;
\draw [shift={(660,375)}, rotate = 222.88] [color={rgb, 255:red, 0; green, 0; blue, 0 }  ][fill={rgb, 255:red, 0; green, 0; blue, 0 }  ][line width=0.75]      (0, 0) circle [x radius= 3.35, y radius= 3.35]   ;
\draw    (660,405) -- (450,210) ;
\draw [shift={(450,210)}, rotate = 222.88] [color={rgb, 255:red, 0; green, 0; blue, 0 }  ][fill={rgb, 255:red, 0; green, 0; blue, 0 }  ][line width=0.75]      (0, 0) circle [x radius= 3.35, y radius= 3.35]   ;
\draw [shift={(660,405)}, rotate = 222.88] [color={rgb, 255:red, 0; green, 0; blue, 0 }  ][fill={rgb, 255:red, 0; green, 0; blue, 0 }  ][line width=0.75]      (0, 0) circle [x radius= 3.35, y radius= 3.35]   ;
\draw    (660,435) -- (450,240) ;
\draw [shift={(450,240)}, rotate = 222.88] [color={rgb, 255:red, 0; green, 0; blue, 0 }  ][fill={rgb, 255:red, 0; green, 0; blue, 0 }  ][line width=0.75]      (0, 0) circle [x radius= 3.35, y radius= 3.35]   ;
\draw [shift={(660,435)}, rotate = 222.88] [color={rgb, 255:red, 0; green, 0; blue, 0 }  ][fill={rgb, 255:red, 0; green, 0; blue, 0 }  ][line width=0.75]      (0, 0) circle [x radius= 3.35, y radius= 3.35]   ;

\begin{footnotesize}
 
\draw (512.5,144.5) node    {$\P^{\perp }$};
\draw (512,430) node    {$\L$};
\draw (500,80) node    {$V\varGamma$};
\draw (603,470) node    {$\P$};
\draw (612,109.5) node    {$\langle {\b 1}\rangle ^{\perp }$};
\draw (512,115) node    {$9$};
\draw (603,150) node    {$9$};
\draw (512,460) node    {$9$};
\draw (603,500) node    {$9$};
\draw (507.5,495) node    {$\langle {\b 1}\rangle $};
\draw (602.5,535) node    {$\{{\b 0}\}$};
\draw (643,179.5) node    {$\langle\P,\Delta_1,\Delta_2,\Delta_3\rangle$};
\draw (390,181) node    {$\langle\P,\ell,\Delta_1,\Phi_1\rangle$};
\draw (390,211) node    {$\langle\P,\ell,\Delta_2,\Phi_2\rangle$};
\draw (390,241) node    {$\langle\P,\ell,\Delta_3,\Phi_3\rangle$};
\draw (697.5,286) node    {$\langle\P, \Delta _{1}\rangle$};
\draw (697.5,316) node    {$\langle\P, \Delta _{2}\rangle$};
\draw (697.5,346) node    {$\langle\P, \Delta _{3}\rangle$};
\draw (697,376) node    {$\langle\P, \Phi _{1}\rangle$};
\draw (698,406) node    {$\langle\P, \Phi _{2}\rangle$};
\draw (698,436) node    {$\langle\P, \Phi _{3}\rangle$};
\draw (390,269) node    {$\langle\P,\Delta_1,\Phi _2,\Phi_3\rangle$};
\draw (390,299) node    {$\langle\P,\Delta_2,\Phi _1,\Phi_3\rangle$};
\draw (390,329) node    {$\langle\P,\Delta_3,\Phi _2,\Phi_1\rangle$};

\end{footnotesize}

\end{tikzpicture}
 \caption{The submodule lattice of the permutation module over $\F_2$ for the action of $\PSL_3(4)$ on the points of $\pg_{2}(4)$. Each vertex is labelled with the corresponding code; each edge indicates the codimension of the corresponding inclusion, with an empty label denoting codimension $1$. For definitions of the codes see Example~\ref{pslcodesex} and for further information see Lemma~\ref{submodulelattice}.}
 \label{pslsubmodlattdiag}
\end{figure}
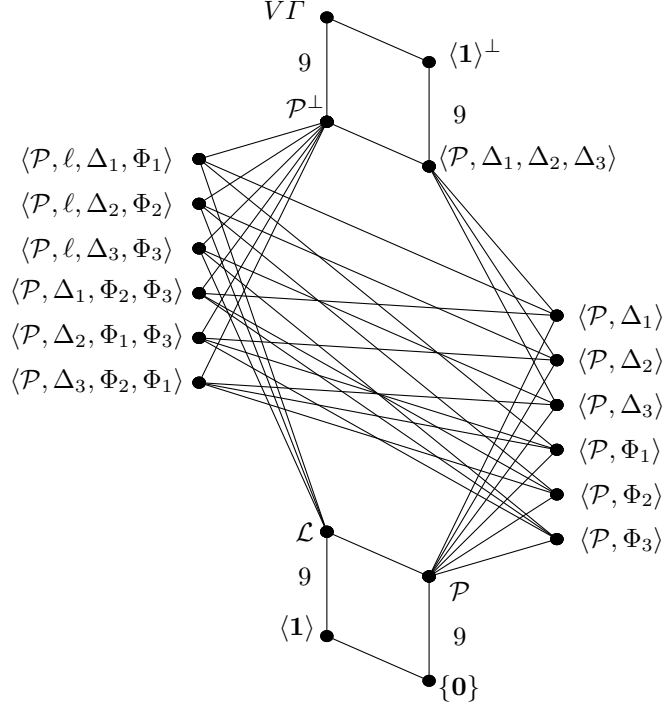

\begin{lemma}\label{psl34cmax}
 Suppose $m=21$ and Hypothesis~\ref{hyp2} holds with $\soc(X/K)\cong\PSL_3(4)$. Then $C_{\text{max}}$ is equivalent to one of following codes:
 \begin{itemize}
  \item The $[21,12,5;3]$ code $\P^\perp$ with $X_{\b 0}\cong \PGaL_3(4)$.
  \item The $[21,11,5;6]$ code $\langle\L,\Delta_1\rangle=\langle\P,\ell,\Delta_1,\Phi_1\rangle$ with $X_{\b 0}\cong \PSiL_3(4)$.
  \item The $[21,11,6;5]$ code $\langle\P,\Phi_1\rangle^\perp=\langle\P,\Delta_1,\Phi_2,\Phi_3\rangle$ with $X_{\b 0}\cong \PSiL_3(4)$.
  \item The $[21,11,6;5]$ code $\L^\perp=\langle\P,\Delta_1,\Delta_2,\Delta_3\rangle$ with $X_{\b 0}\cong \PGaL_3(4)$.
  \item The $[21,10,5;6]$ code $\L=\langle\P,\ell\rangle$ with $X_{\b 0}\cong \PGaL_3(4)$.
  \item The $[21,10,6;7]$ code $\langle\P,\Delta_1\rangle=\langle\L,\Delta_1\rangle^\perp$ with $X_{\b 0}\cong \PSiL_3(4)$.
  \item The $[21,10,7;6]$ code $\langle\P,\Phi_1\rangle$ with $X_{\b 0}\cong \PSiL_3(4)$.
  \item The $[21,9,8;7]$ code $\P$ with $X_{\b 0}\cong \PGaL_3(4)$.
 \end{itemize}
\end{lemma}

\begin{proof}
 Let $C=C_{\text{max}}$. Since $2\leq \dim(C)\leq m-2$, the possibilities for $C$ are determined by the non-trivial submodules as in Lemma~\ref{submodulelattice} and Figure~\ref{pslsubmodlattdiag}, from which the dimension of each code and equivalences between codes can be deduced. The equality $\langle\L,\Delta_1\rangle=\langle\P,\ell,\Delta_1,\Phi_1\rangle$ follows from the fact that $\L=\langle\P,\ell\rangle$ and from Lemma~\ref{sumoflineandFanoplane}, and Figure~\ref{pslsubmodlattdiag} implies the equality $\L^\perp=\langle\P,\Delta_1,\Delta_2,\Delta_3\rangle$. In each case $X_{\b 0}$ can be deduced by considering the Fano plane in Figure~\ref{fanoconfig} and the fact that $\PSL_3(4)$ fixes it point-wise while, by Lemma~\ref{symdiffhyperovals}, $\PGaL_3(4)$ fixes $\P+\ell$ and induces $\s_3$ on $\{\P+\Delta_i\mid i=1,2,3\}$. Indeed, the codes containing precisely one orbit of hyperovals, or precisely one or two orbits of Fano subplanes, satisfy $X_{\b 0}\cong\PSiL_3(4)$ whilst the remaining codes satisfy $X_{\b 0}\cong\PGaL_3(4)$. Note that there are three copies of $\PSiL_3(4)$ inside $\PGaL_3(4)$ and these correspond to the stabilisers of each of the three lines through $\ell$ in Figure~\ref{fanoconfig}. The equalities 
 \[
  \langle\P,\Phi_1\rangle^\perp=\langle\P,\Delta_1,\Phi_2,\Phi_3\rangle \quad \text{and} \quad \langle\P,\Delta_1\rangle^\perp=\langle\L,\Delta_1\rangle,
 \]
 among the four codes that are invariant under the same copy of $\PSiL_3(4)$, can be seen from the fact that $\langle\P,\Delta_1\rangle$ is contained in three pairwise inequivalent codes of dimension $11$ but $\langle\P,\Phi_1\rangle$ is contained in a pair of equivalent codes of dimension $11$, whilst $\langle\L,\Delta_1\rangle$ contains three inequivalent codes of dimension $10$ but $\langle\P,\Delta_1,\Phi_2,\Phi_3\rangle$ contains a pair of equivalent codes of dimension $10$.
 
 Now, the set of all lines of $\pg_2(4)$ form the set of weight $5$ codewords of $\P^\perp$; the set of all hyperovals of $\pg_2(4)$ form the set of weight $6$ codewords of $\P^\perp$; and the set of all Fano subplanes of $\pg_2(4)$ form the set of weight $7$ codewords of $\P^\perp$. Thus, the codes $\langle\L,\Delta_1\rangle$ and $\L$ both have minimum distance $5$; the codes $\langle\P,\Phi_1\rangle^\perp$, $\L^\perp$ and $\langle\P,\Delta_1\rangle$ have minimum distance $6$; and $\langle\P,\Phi_1\rangle$ has minimum distance $7$. The covering radii were determined via computation in the GAP package GUAVA \cite{cramwinckelgap}.
\end{proof}

\begin{lemma}\label{psl34CT}
 Suppose $m=21$ and Hypothesis~\ref{hyp2} holds with $\soc(X/K)\cong\PSL_3(4)$. If $C=C_{\text{max}}$, then $C$ is equivalent to one of the following codes in $H(21,2)$, each of which is completely transitive:
 \begin{itemize}
  \item The $[21,12,5;3]$ code $\P^\perp$,
  \item The $[21,11,5;6]$ code $\langle\L,\Delta_1\rangle$, or,
  \item The $[21,10,5;6]$ code $\L$.
 \end{itemize}
\end{lemma}

\begin{proof}
 The possibilities for $C_{\text{max}}$ are given in Lemma~\ref{psl34cmax}. Using the GAP package GUAVA \cite{cramwinckelgap} to compare the covering radius of a code with the number of non-zero weights of its dual, which for a completely regular code must be equal by \cite[Theorem~4 (iii)]{borges2019completely}, allows many of the possibilities to be eliminated. The covering radius of $\langle\P,\Phi_1\rangle^\perp$ is $5$, but $\langle\P,\Phi_1\rangle$ has $6$ non-zero weights, so $\langle\P,\Phi_1\rangle^\perp$ is not completely regular, and hence not completely transitive. The covering radius of $\L^\perp$ is $5$, but $\L$ has $7$ non-zero weights, so $\L^\perp$ is not completely transitive. The covering radius of $\langle\P,\Delta_1\rangle$ is $7$, but $\langle\P,\Delta_1\rangle^\perp=\langle\L,\Delta_1\rangle$ has $13$ non-zero weights, so $\langle\P,\Delta_1\rangle$ is not completely transitive. The covering radius of $\langle\P,\Phi_1\rangle$ is $6$, but $\langle\P,\Phi_1\rangle^\perp$ has $9$ non-zero weights, so $\langle\P,\Phi_1\rangle$ is not completely transitive. The covering radius of $\P$ is $7$, but $\P^\perp$ has $13$ non-zero weights, so $\P$ is not completely transitive.

 Now, $\P^\perp$ has covering radius $3$ and minimum distance $5$, so $\varGamma_3(\b 0)\cap \left(\P^\perp\right)_3$ consists of the set of all triangles of $\pg_2(4)$. As $\PSL_3(4)$ acts transitively on triangles we have that $\P^\perp$ is completely transitive.

 The code $\langle\L,\Delta_1\rangle$ has covering radius $6$ and contains all lines of $\pg_2(4)$ as well as one $\PSL_3(4)$-orbit of hyperovals and one $\PSL_3(4)$-orbit of Fano subplanes. Thus, $\varGamma_3(\b 0)\cap \langle\L,\Delta_1\rangle_3$ again consists of the set of all triangles of $\pg_2(4)$, on which $\PSL_3(4)$ acts transitively. Since each quadrangle (set of four points with the property that no $3$ are collinear) of $\pg_2(4)$ is contained in a unique hyperoval and a unique Fano subplane, $\varGamma_4(\b 0)\cap \langle\L,\Delta_1\rangle_4$ consists of two $\PSL_3(4)$-orbits of quadrangles. However, $X_{\b 0}\cong\PSiL_3(4)$ acts transitively on $\{\H_2,\H_3\}$. Hence $\PSiL_3(4)$ acts transitively on $\varGamma_4(\b 0)\cap \langle\L,\Delta_1\rangle_4$. Similarly, two $\PSL_3(4)$-orbits of ovals of $\pg_2(4)$ make up $\varGamma_5(\b 0)\cap \langle\L,\Delta_1\rangle_5$. Again, these orbits are interchanged under the field automorphism of $\F_4$, and hence $\PSiL_3(4)$ acts transitively on $\varGamma_5(\b 0)\cap \langle\L,\Delta_1\rangle_5$. Moreover, the remaining two $\PSL_3(4)$-orbits of hyperovals of $\pg_2(4)$ make up $\varGamma_6(\b 0)\cap \langle\L,\Delta_1\rangle_6$, these orbits are again interchanged under the field automorphism of $\F_4$, so that $\PSiL_3(4)$ acts transitively on $\varGamma_6(\b 0)\cap \langle\L,\Delta_1\rangle_6$. It follows that $\P(\Delta)^\perp$ is completely transitive.

 Finally, $\L$ has covering radius $6$ and contains all lines of $\pg_2(4)$. Now $\PGaL_3(4)$ acts transitively on the quadrangles, ovals and hyperovals of $\pg_2(4)$, respectively, which make up $\varGamma_4(\b 0)\cap \L_4$, $\varGamma_5(\b 0)\cap \L_5$ and $\varGamma_6(\b 0)\cap \L_6$, respectively. Since $\L$ is also invariant under $\PGaL_3(4)$ acting on $M$ it follows that $\L$ is completely transitive.
\end{proof}

\begin{figure}
 \[
\begin{tikzpicture}[x=0.75pt,y=0.75pt,yscale=-0.5,xscale=0.7]

\draw    (26,40) -- (26,88) ;
\draw [shift={(26,90)}, rotate = 270] [color={rgb, 255:red, 0; green, 0; blue, 0 }  ][line width=0.75]    (10.93,-3.29) .. controls (6.95,-1.4) and (3.31,-0.3) .. (0,0) .. controls (3.31,0.3) and (6.95,1.4) .. (10.93,3.29)   ;
\draw [shift={(26,40)}, rotate = 270] [color={rgb, 255:red, 0; green, 0; blue, 0 }  ][line width=0.75]    (0,5.59) -- (0,-5.59)   ;
\draw    (61,105) -- (99,105) ;
\draw [shift={(101,105)}, rotate = 180] [color={rgb, 255:red, 0; green, 0; blue, 0 }  ][line width=0.75]    (10.93,-3.29) .. controls (6.95,-1.4) and (3.31,-0.3) .. (0,0) .. controls (3.31,0.3) and (6.95,1.4) .. (10.93,3.29)   ;
\draw [shift={(61,105)}, rotate = 180] [color={rgb, 255:red, 0; green, 0; blue, 0 }  ][line width=0.75]    (0,5.59) -- (0,-5.59)   ;
\draw    (131,125) -- (131,173) ;
\draw [shift={(131,175)}, rotate = 270] [color={rgb, 255:red, 0; green, 0; blue, 0 }  ][line width=0.75]    (10.93,-3.29) .. controls (6.95,-1.4) and (3.31,-0.3) .. (0,0) .. controls (3.31,0.3) and (6.95,1.4) .. (10.93,3.29)   ;
\draw [shift={(131,125)}, rotate = 270] [color={rgb, 255:red, 0; green, 0; blue, 0 }  ][line width=0.75]    (0,5.59) -- (0,-5.59)   ;
\draw    (166,190) -- (204,190) ;
\draw [shift={(206,190)}, rotate = 180] [color={rgb, 255:red, 0; green, 0; blue, 0 }  ][line width=0.75]    (10.93,-3.29) .. controls (6.95,-1.4) and (3.31,-0.3) .. (0,0) .. controls (3.31,0.3) and (6.95,1.4) .. (10.93,3.29)   ;
\draw [shift={(166,190)}, rotate = 180] [color={rgb, 255:red, 0; green, 0; blue, 0 }  ][line width=0.75]    (0,5.59) -- (0,-5.59)   ;
\draw    (236,210) -- (236,258) ;
\draw [shift={(236,260)}, rotate = 270] [color={rgb, 255:red, 0; green, 0; blue, 0 }  ][line width=0.75]    (10.93,-3.29) .. controls (6.95,-1.4) and (3.31,-0.3) .. (0,0) .. controls (3.31,0.3) and (6.95,1.4) .. (10.93,3.29)   ;
\draw [shift={(236,210)}, rotate = 270] [color={rgb, 255:red, 0; green, 0; blue, 0 }  ][line width=0.75]    (0,5.59) -- (0,-5.59)   ;
\draw    (265.5,275) -- (308.5,275) ;
\draw [shift={(310.5,275)}, rotate = 180] [color={rgb, 255:red, 0; green, 0; blue, 0 }  ][line width=0.75]    (10.93,-3.29) .. controls (6.95,-1.4) and (3.31,-0.3) .. (0,0) .. controls (3.31,0.3) and (6.95,1.4) .. (10.93,3.29)   ;
\draw [shift={(265.5,275)}, rotate = 180] [color={rgb, 255:red, 0; green, 0; blue, 0 }  ][line width=0.75]    (0,5.59) -- (0,-5.59)   ;
\draw    (336,295) -- (336,343) ;
\draw [shift={(336,345)}, rotate = 270] [color={rgb, 255:red, 0; green, 0; blue, 0 }  ][line width=0.75]    (10.93,-3.29) .. controls (6.95,-1.4) and (3.31,-0.3) .. (0,0) .. controls (3.31,0.3) and (6.95,1.4) .. (10.93,3.29)   ;
\draw [shift={(336,295)}, rotate = 270] [color={rgb, 255:red, 0; green, 0; blue, 0 }  ][line width=0.75]    (0,5.59) -- (0,-5.59)   ;
\draw    (371,360) -- (409,360) ;
\draw [shift={(411,360)}, rotate = 180] [color={rgb, 255:red, 0; green, 0; blue, 0 }  ][line width=0.75]    (10.93,-3.29) .. controls (6.95,-1.4) and (3.31,-0.3) .. (0,0) .. controls (3.31,0.3) and (6.95,1.4) .. (10.93,3.29)   ;
\draw [shift={(371,360)}, rotate = 180] [color={rgb, 255:red, 0; green, 0; blue, 0 }  ][line width=0.75]    (0,5.59) -- (0,-5.59)   ;
\draw    (441,380) -- (441,428) ;
\draw [shift={(441,430)}, rotate = 270] [color={rgb, 255:red, 0; green, 0; blue, 0 }  ][line width=0.75]    (10.93,-3.29) .. controls (6.95,-1.4) and (3.31,-0.3) .. (0,0) .. controls (3.31,0.3) and (6.95,1.4) .. (10.93,3.29)   ;
\draw [shift={(441,380)}, rotate = 270] [color={rgb, 255:red, 0; green, 0; blue, 0 }  ][line width=0.75]    (0,5.59) -- (0,-5.59)   ;
\draw    (476,445) -- (514,445) ;
\draw [shift={(516,445)}, rotate = 180] [color={rgb, 255:red, 0; green, 0; blue, 0 }  ][line width=0.75]    (10.93,-3.29) .. controls (6.95,-1.4) and (3.31,-0.3) .. (0,0) .. controls (3.31,0.3) and (6.95,1.4) .. (10.93,3.29)   ;
\draw [shift={(476,445)}, rotate = 180] [color={rgb, 255:red, 0; green, 0; blue, 0 }  ][line width=0.75]    (0,5.59) -- (0,-5.59)   ;
\draw    (546,465) -- (546,513) ;
\draw [shift={(546,515)}, rotate = 270] [color={rgb, 255:red, 0; green, 0; blue, 0 }  ][line width=0.75]    (10.93,-3.29) .. controls (6.95,-1.4) and (3.31,-0.3) .. (0,0) .. controls (3.31,0.3) and (6.95,1.4) .. (10.93,3.29)   ;
\draw [shift={(546,465)}, rotate = 270] [color={rgb, 255:red, 0; green, 0; blue, 0 }  ][line width=0.75]    (0,5.59) -- (0,-5.59)   ;
\draw    (567,530) -- (620,530) ;
\draw [shift={(622,530)}, rotate = 180] [color={rgb, 255:red, 0; green, 0; blue, 0 }  ][line width=0.75]    (10.93,-3.29) .. controls (6.95,-1.4) and (3.31,-0.3) .. (0,0) .. controls (3.31,0.3) and (6.95,1.4) .. (10.93,3.29)   ;
\draw [shift={(567,530)}, rotate = 180] [color={rgb, 255:red, 0; green, 0; blue, 0 }  ][line width=0.75]    (0,5.59) -- (0,-5.59)   ;
\draw    (334.5,40) -- (334.5,88) ;
\draw [shift={(334.5,90)}, rotate = 270] [color={rgb, 255:red, 0; green, 0; blue, 0 }  ][line width=0.75]    (10.93,-3.29) .. controls (6.95,-1.4) and (3.31,-0.3) .. (0,0) .. controls (3.31,0.3) and (6.95,1.4) .. (10.93,3.29)   ;
\draw [shift={(334.5,40)}, rotate = 270] [color={rgb, 255:red, 0; green, 0; blue, 0 }  ][line width=0.75]    (0,5.59) -- (0,-5.59)   ;
\draw    (369.5,105) -- (407.5,105) ;
\draw [shift={(409.5,105)}, rotate = 180] [color={rgb, 255:red, 0; green, 0; blue, 0 }  ][line width=0.75]    (10.93,-3.29) .. controls (6.95,-1.4) and (3.31,-0.3) .. (0,0) .. controls (3.31,0.3) and (6.95,1.4) .. (10.93,3.29)   ;
\draw [shift={(369.5,105)}, rotate = 180] [color={rgb, 255:red, 0; green, 0; blue, 0 }  ][line width=0.75]    (0,5.59) -- (0,-5.59)   ;
\draw    (439.5,125) -- (439.5,173) ;
\draw [shift={(439.5,175)}, rotate = 270] [color={rgb, 255:red, 0; green, 0; blue, 0 }  ][line width=0.75]    (10.93,-3.29) .. controls (6.95,-1.4) and (3.31,-0.3) .. (0,0) .. controls (3.31,0.3) and (6.95,1.4) .. (10.93,3.29)   ;
\draw [shift={(439.5,125)}, rotate = 270] [color={rgb, 255:red, 0; green, 0; blue, 0 }  ][line width=0.75]    (0,5.59) -- (0,-5.59)   ;
\draw    (474.5,190) -- (512.5,190) ;
\draw [shift={(514.5,190)}, rotate = 180] [color={rgb, 255:red, 0; green, 0; blue, 0 }  ][line width=0.75]    (10.93,-3.29) .. controls (6.95,-1.4) and (3.31,-0.3) .. (0,0) .. controls (3.31,0.3) and (6.95,1.4) .. (10.93,3.29)   ;
\draw [shift={(474.5,190)}, rotate = 180] [color={rgb, 255:red, 0; green, 0; blue, 0 }  ][line width=0.75]    (0,5.59) -- (0,-5.59)   ;
\draw    (544.5,210) -- (544.5,258) ;
\draw [shift={(544.5,260)}, rotate = 270] [color={rgb, 255:red, 0; green, 0; blue, 0 }  ][line width=0.75]    (10.93,-3.29) .. controls (6.95,-1.4) and (3.31,-0.3) .. (0,0) .. controls (3.31,0.3) and (6.95,1.4) .. (10.93,3.29)   ;
\draw [shift={(544.5,210)}, rotate = 270] [color={rgb, 255:red, 0; green, 0; blue, 0 }  ][line width=0.75]    (0,5.59) -- (0,-5.59)   ;
\draw    (565.5,275) -- (623.5,275) ;
\draw [shift={(625.5,275)}, rotate = 180] [color={rgb, 255:red, 0; green, 0; blue, 0 }  ][line width=0.75]    (10.93,-3.29) .. controls (6.95,-1.4) and (3.31,-0.3) .. (0,0) .. controls (3.31,0.3) and (6.95,1.4) .. (10.93,3.29)   ;
\draw [shift={(565.5,275)}, rotate = 180] [color={rgb, 255:red, 0; green, 0; blue, 0 }  ][line width=0.75]    (0,5.59) -- (0,-5.59)   ;
\draw    (574,210) -- (632.46,258.72) ;
\draw [shift={(634,260)}, rotate = 219.81] [color={rgb, 255:red, 0; green, 0; blue, 0 }  ][line width=0.75]    (10.93,-3.29) .. controls (6.95,-1.4) and (3.31,-0.3) .. (0,0) .. controls (3.31,0.3) and (6.95,1.4) .. (10.93,3.29)   ;
\draw [shift={(574,210)}, rotate = 219.81] [color={rgb, 255:red, 0; green, 0; blue, 0 }  ][line width=0.75]    (0,5.59) -- (0,-5.59)   ;
\draw    (165.5,125) -- (223.96,173.72) ;
\draw [shift={(225.5,175)}, rotate = 219.81] [color={rgb, 255:red, 0; green, 0; blue, 0 }  ][line width=0.75]    (10.93,-3.29) .. controls (6.95,-1.4) and (3.31,-0.3) .. (0,0) .. controls (3.31,0.3) and (6.95,1.4) .. (10.93,3.29)   ;
\draw [shift={(165.5,125)}, rotate = 219.81] [color={rgb, 255:red, 0; green, 0; blue, 0 }  ][line width=0.75]    (0,5.59) -- (0,-5.59)   ;
\draw    (270.5,210) -- (328.96,258.72) ;
\draw [shift={(330.5,260)}, rotate = 219.81] [color={rgb, 255:red, 0; green, 0; blue, 0 }  ][line width=0.75]    (10.93,-3.29) .. controls (6.95,-1.4) and (3.31,-0.3) .. (0,0) .. controls (3.31,0.3) and (6.95,1.4) .. (10.93,3.29)   ;
\draw [shift={(270.5,210)}, rotate = 219.81] [color={rgb, 255:red, 0; green, 0; blue, 0 }  ][line width=0.75]    (0,5.59) -- (0,-5.59)   ;
\draw    (365.5,290) -- (423.96,338.72) ;
\draw [shift={(425.5,340)}, rotate = 219.81] [color={rgb, 255:red, 0; green, 0; blue, 0 }  ][line width=0.75]    (10.93,-3.29) .. controls (6.95,-1.4) and (3.31,-0.3) .. (0,0) .. controls (3.31,0.3) and (6.95,1.4) .. (10.93,3.29)   ;
\draw [shift={(365.5,290)}, rotate = 219.81] [color={rgb, 255:red, 0; green, 0; blue, 0 }  ][line width=0.75]    (0,5.59) -- (0,-5.59)   ;
\draw    (475.5,380) -- (533.96,428.72) ;
\draw [shift={(535.5,430)}, rotate = 219.81] [color={rgb, 255:red, 0; green, 0; blue, 0 }  ][line width=0.75]    (10.93,-3.29) .. controls (6.95,-1.4) and (3.31,-0.3) .. (0,0) .. controls (3.31,0.3) and (6.95,1.4) .. (10.93,3.29)   ;
\draw [shift={(475.5,380)}, rotate = 219.81] [color={rgb, 255:red, 0; green, 0; blue, 0 }  ][line width=0.75]    (0,5.59) -- (0,-5.59)   ;
\draw    (575.5,465) -- (633.96,513.72) ;
\draw [shift={(635.5,515)}, rotate = 219.81] [color={rgb, 255:red, 0; green, 0; blue, 0 }  ][line width=0.75]    (10.93,-3.29) .. controls (6.95,-1.4) and (3.31,-0.3) .. (0,0) .. controls (3.31,0.3) and (6.95,1.4) .. (10.93,3.29)   ;
\draw [shift={(575.5,465)}, rotate = 219.81] [color={rgb, 255:red, 0; green, 0; blue, 0 }  ][line width=0.75]    (0,5.59) -- (0,-5.59)   ;
\draw    (359,35) -- (417.46,83.72) ;
\draw [shift={(419,85)}, rotate = 219.81] [color={rgb, 255:red, 0; green, 0; blue, 0 }  ][line width=0.75]    (10.93,-3.29) .. controls (6.95,-1.4) and (3.31,-0.3) .. (0,0) .. controls (3.31,0.3) and (6.95,1.4) .. (10.93,3.29)   ;
\draw [shift={(359,35)}, rotate = 219.81] [color={rgb, 255:red, 0; green, 0; blue, 0 }  ][line width=0.75]    (0,5.59) -- (0,-5.59)   ;
\draw    (474,125) -- (532.46,173.72) ;
\draw [shift={(534,175)}, rotate = 219.81] [color={rgb, 255:red, 0; green, 0; blue, 0 }  ][line width=0.75]    (10.93,-3.29) .. controls (6.95,-1.4) and (3.31,-0.3) .. (0,0) .. controls (3.31,0.3) and (6.95,1.4) .. (10.93,3.29)   ;
\draw [shift={(474,125)}, rotate = 219.81] [color={rgb, 255:red, 0; green, 0; blue, 0 }  ][line width=0.75]    (0,5.59) -- (0,-5.59)   ;
\draw    (50.5,35) -- (108.96,83.72) ;
\draw [shift={(110.5,85)}, rotate = 219.81] [color={rgb, 255:red, 0; green, 0; blue, 0 }  ][line width=0.75]    (10.93,-3.29) .. controls (6.95,-1.4) and (3.31,-0.3) .. (0,0) .. controls (3.31,0.3) and (6.95,1.4) .. (10.93,3.29)   ;
\draw [shift={(50.5,35)}, rotate = 219.81] [color={rgb, 255:red, 0; green, 0; blue, 0 }  ][line width=0.75]    (0,5.59) -- (0,-5.59)   ;
\draw    (25.5,295) -- (25.5,343) ;
\draw [shift={(25.5,345)}, rotate = 270] [color={rgb, 255:red, 0; green, 0; blue, 0 }  ][line width=0.75]    (10.93,-3.29) .. controls (6.95,-1.4) and (3.31,-0.3) .. (0,0) .. controls (3.31,0.3) and (6.95,1.4) .. (10.93,3.29)   ;
\draw [shift={(25.5,295)}, rotate = 270] [color={rgb, 255:red, 0; green, 0; blue, 0 }  ][line width=0.75]    (0,5.59) -- (0,-5.59)   ;
\draw    (60.5,360) -- (98.5,360) ;
\draw [shift={(100.5,360)}, rotate = 180] [color={rgb, 255:red, 0; green, 0; blue, 0 }  ][line width=0.75]    (10.93,-3.29) .. controls (6.95,-1.4) and (3.31,-0.3) .. (0,0) .. controls (3.31,0.3) and (6.95,1.4) .. (10.93,3.29)   ;
\draw [shift={(60.5,360)}, rotate = 180] [color={rgb, 255:red, 0; green, 0; blue, 0 }  ][line width=0.75]    (0,5.59) -- (0,-5.59)   ;
\draw    (50,290) -- (108.46,338.72) ;
\draw [shift={(110,340)}, rotate = 219.81] [color={rgb, 255:red, 0; green, 0; blue, 0 }  ][line width=0.75]    (10.93,-3.29) .. controls (6.95,-1.4) and (3.31,-0.3) .. (0,0) .. controls (3.31,0.3) and (6.95,1.4) .. (10.93,3.29)   ;
\draw [shift={(50,290)}, rotate = 219.81] [color={rgb, 255:red, 0; green, 0; blue, 0 }  ][line width=0.75]    (0,5.59) -- (0,-5.59)   ;
\draw    (440.5,635) -- (440.5,683) ;
\draw [shift={(440.5,685)}, rotate = 270] [color={rgb, 255:red, 0; green, 0; blue, 0 }  ][line width=0.75]    (10.93,-3.29) .. controls (6.95,-1.4) and (3.31,-0.3) .. (0,0) .. controls (3.31,0.3) and (6.95,1.4) .. (10.93,3.29)   ;
\draw [shift={(440.5,635)}, rotate = 270] [color={rgb, 255:red, 0; green, 0; blue, 0 }  ][line width=0.75]    (0,5.59) -- (0,-5.59)   ;
\draw    (475.5,700) -- (513.5,700) ;
\draw [shift={(515.5,700)}, rotate = 180] [color={rgb, 255:red, 0; green, 0; blue, 0 }  ][line width=0.75]    (10.93,-3.29) .. controls (6.95,-1.4) and (3.31,-0.3) .. (0,0) .. controls (3.31,0.3) and (6.95,1.4) .. (10.93,3.29)   ;
\draw [shift={(475.5,700)}, rotate = 180] [color={rgb, 255:red, 0; green, 0; blue, 0 }  ][line width=0.75]    (0,5.59) -- (0,-5.59)   ;
\draw    (545.5,720) -- (545.5,768) ;
\draw [shift={(545.5,770)}, rotate = 270] [color={rgb, 255:red, 0; green, 0; blue, 0 }  ][line width=0.75]    (10.93,-3.29) .. controls (6.95,-1.4) and (3.31,-0.3) .. (0,0) .. controls (3.31,0.3) and (6.95,1.4) .. (10.93,3.29)   ;
\draw [shift={(545.5,720)}, rotate = 270] [color={rgb, 255:red, 0; green, 0; blue, 0 }  ][line width=0.75]    (0,5.59) -- (0,-5.59)   ;
\draw    (580.5,785) -- (613,785) ;
\draw [shift={(615,785)}, rotate = 180] [color={rgb, 255:red, 0; green, 0; blue, 0 }  ][line width=0.75]    (10.93,-3.29) .. controls (6.95,-1.4) and (3.31,-0.3) .. (0,0) .. controls (3.31,0.3) and (6.95,1.4) .. (10.93,3.29)   ;
\draw [shift={(580.5,785)}, rotate = 180] [color={rgb, 255:red, 0; green, 0; blue, 0 }  ][line width=0.75]    (0,5.59) -- (0,-5.59)   ;
\draw    (26,550) -- (26,598) ;
\draw [shift={(26,600)}, rotate = 270] [color={rgb, 255:red, 0; green, 0; blue, 0 }  ][line width=0.75]    (10.93,-3.29) .. controls (6.95,-1.4) and (3.31,-0.3) .. (0,0) .. controls (3.31,0.3) and (6.95,1.4) .. (10.93,3.29)   ;
\draw [shift={(26,550)}, rotate = 270] [color={rgb, 255:red, 0; green, 0; blue, 0 }  ][line width=0.75]    (0,5.59) -- (0,-5.59)   ;
\draw    (61,615) -- (99,615) ;
\draw [shift={(101,615)}, rotate = 180] [color={rgb, 255:red, 0; green, 0; blue, 0 }  ][line width=0.75]    (10.93,-3.29) .. controls (6.95,-1.4) and (3.31,-0.3) .. (0,0) .. controls (3.31,0.3) and (6.95,1.4) .. (10.93,3.29)   ;
\draw [shift={(61,615)}, rotate = 180] [color={rgb, 255:red, 0; green, 0; blue, 0 }  ][line width=0.75]    (0,5.59) -- (0,-5.59)   ;
\draw    (131,635) -- (131,683) ;
\draw [shift={(131,685)}, rotate = 270] [color={rgb, 255:red, 0; green, 0; blue, 0 }  ][line width=0.75]    (10.93,-3.29) .. controls (6.95,-1.4) and (3.31,-0.3) .. (0,0) .. controls (3.31,0.3) and (6.95,1.4) .. (10.93,3.29)   ;
\draw [shift={(131,635)}, rotate = 270] [color={rgb, 255:red, 0; green, 0; blue, 0 }  ][line width=0.75]    (0,5.59) -- (0,-5.59)   ;
\draw    (166,700) -- (204,700) ;
\draw [shift={(206,700)}, rotate = 180] [color={rgb, 255:red, 0; green, 0; blue, 0 }  ][line width=0.75]    (10.93,-3.29) .. controls (6.95,-1.4) and (3.31,-0.3) .. (0,0) .. controls (3.31,0.3) and (6.95,1.4) .. (10.93,3.29)   ;
\draw [shift={(166,700)}, rotate = 180] [color={rgb, 255:red, 0; green, 0; blue, 0 }  ][line width=0.75]    (0,5.59) -- (0,-5.59)   ;
\draw    (580,720) -- (638.46,768.72) ;
\draw [shift={(640,770)}, rotate = 219.81] [color={rgb, 255:red, 0; green, 0; blue, 0 }  ][line width=0.75]    (10.93,-3.29) .. controls (6.95,-1.4) and (3.31,-0.3) .. (0,0) .. controls (3.31,0.3) and (6.95,1.4) .. (10.93,3.29)   ;
\draw [shift={(580,720)}, rotate = 219.81] [color={rgb, 255:red, 0; green, 0; blue, 0 }  ][line width=0.75]    (0,5.59) -- (0,-5.59)   ;
\draw    (55.5,545) -- (113.96,593.72) ;
\draw [shift={(115.5,595)}, rotate = 219.81] [color={rgb, 255:red, 0; green, 0; blue, 0 }  ][line width=0.75]    (10.93,-3.29) .. controls (6.95,-1.4) and (3.31,-0.3) .. (0,0) .. controls (3.31,0.3) and (6.95,1.4) .. (10.93,3.29)   ;
\draw [shift={(55.5,545)}, rotate = 219.81] [color={rgb, 255:red, 0; green, 0; blue, 0 }  ][line width=0.75]    (0,5.59) -- (0,-5.59)   ;
\draw    (165.5,635) -- (223.96,683.72) ;
\draw [shift={(225.5,685)}, rotate = 219.81] [color={rgb, 255:red, 0; green, 0; blue, 0 }  ][line width=0.75]    (10.93,-3.29) .. controls (6.95,-1.4) and (3.31,-0.3) .. (0,0) .. controls (3.31,0.3) and (6.95,1.4) .. (10.93,3.29)   ;
\draw [shift={(165.5,635)}, rotate = 219.81] [color={rgb, 255:red, 0; green, 0; blue, 0 }  ][line width=0.75]    (0,5.59) -- (0,-5.59)   ;
\draw    (465,630) -- (523.46,678.72) ;
\draw [shift={(525,680)}, rotate = 219.81] [color={rgb, 255:red, 0; green, 0; blue, 0 }  ][line width=0.75]    (10.93,-3.29) .. controls (6.95,-1.4) and (3.31,-0.3) .. (0,0) .. controls (3.31,0.3) and (6.95,1.4) .. (10.93,3.29)   ;
\draw [shift={(465,630)}, rotate = 219.81] [color={rgb, 255:red, 0; green, 0; blue, 0 }  ][line width=0.75]    (0,5.59) -- (0,-5.59)   ;
\draw    (236.5,465) -- (236.5,513) ;
\draw [shift={(236.5,515)}, rotate = 270] [color={rgb, 255:red, 0; green, 0; blue, 0 }  ][line width=0.75]    (10.93,-3.29) .. controls (6.95,-1.4) and (3.31,-0.3) .. (0,0) .. controls (3.31,0.3) and (6.95,1.4) .. (10.93,3.29)   ;
\draw [shift={(236.5,465)}, rotate = 270] [color={rgb, 255:red, 0; green, 0; blue, 0 }  ][line width=0.75]    (0,5.59) -- (0,-5.59)   ;
\draw    (271.5,530) -- (309.5,530) ;
\draw [shift={(311.5,530)}, rotate = 180] [color={rgb, 255:red, 0; green, 0; blue, 0 }  ][line width=0.75]    (10.93,-3.29) .. controls (6.95,-1.4) and (3.31,-0.3) .. (0,0) .. controls (3.31,0.3) and (6.95,1.4) .. (10.93,3.29)   ;
\draw [shift={(271.5,530)}, rotate = 180] [color={rgb, 255:red, 0; green, 0; blue, 0 }  ][line width=0.75]    (0,5.59) -- (0,-5.59)   ;
\draw    (261,460) -- (319.46,508.72) ;
\draw [shift={(321,510)}, rotate = 219.81] [color={rgb, 255:red, 0; green, 0; blue, 0 }  ][line width=0.75]    (10.93,-3.29) .. controls (6.95,-1.4) and (3.31,-0.3) .. (0,0) .. controls (3.31,0.3) and (6.95,1.4) .. (10.93,3.29)   ;
\draw [shift={(261,460)}, rotate = 219.81] [color={rgb, 255:red, 0; green, 0; blue, 0 }  ][line width=0.75]    (0,5.59) -- (0,-5.59)   ;
\draw    (270.5,785) -- (308.5,785) ;
\draw [shift={(310.5,785)}, rotate = 180] [color={rgb, 255:red, 0; green, 0; blue, 0 }  ][line width=0.75]    (10.93,-3.29) .. controls (6.95,-1.4) and (3.31,-0.3) .. (0,0) .. controls (3.31,0.3) and (6.95,1.4) .. (10.93,3.29)   ;
\draw [shift={(270.5,785)}, rotate = 180] [color={rgb, 255:red, 0; green, 0; blue, 0 }  ][line width=0.75]    (0,5.59) -- (0,-5.59)   ;
\draw    (235.5,720) -- (235.5,768) ;
\draw [shift={(235.5,770)}, rotate = 270] [color={rgb, 255:red, 0; green, 0; blue, 0 }  ][line width=0.75]    (10.93,-3.29) .. controls (6.95,-1.4) and (3.31,-0.3) .. (0,0) .. controls (3.31,0.3) and (6.95,1.4) .. (10.93,3.29)   ;
\draw [shift={(235.5,720)}, rotate = 270] [color={rgb, 255:red, 0; green, 0; blue, 0 }  ][line width=0.75]    (0,5.59) -- (0,-5.59)   ;
\draw    (265,720) -- (323.46,768.72) ;
\draw [shift={(325,770)}, rotate = 219.81] [color={rgb, 255:red, 0; green, 0; blue, 0 }  ][line width=0.75]    (10.93,-3.29) .. controls (6.95,-1.4) and (3.31,-0.3) .. (0,0) .. controls (3.31,0.3) and (6.95,1.4) .. (10.93,3.29)   ;
\draw [shift={(265,720)}, rotate = 219.81] [color={rgb, 255:red, 0; green, 0; blue, 0 }  ][line width=0.75]    (0,5.59) -- (0,-5.59)   ;

\begin{tiny}
\draw (24,20) node   {$\mathcal{P}$};
\draw (30,107.5) node   {$\mathcal{P} +\Phi _{2}$};
\draw (135,107.5) node   {$\mathcal{P} +\Phi _{3}$};
\draw (78,90) node   {$\sigma $};
\draw (135,192.5) node   {$\mathcal{P} +\Delta _{1}$};
\draw (240,192.5) node   {$\mathcal{P} +\Delta _{2}$};
\draw (183,175) node   {$\sigma $};
\draw (12,60) node   {$t_{\Phi _{2}}$};
\draw (117,149) node   {$t_{\Phi _{2}}$};
\draw (222,234) node   {$t_{\Phi _{2}}$};
\draw (237,275) node   {$\mathcal{P} +\ell $};
\draw (337,275) node   {$\mathcal{P} +\ell $};
\draw (288,260) node   {$\sigma $};
\draw (343.5,362.5) node   {$\mathcal{P} +\Delta _{2}$};
\draw (445,362.5) node   {$\mathcal{P} +\Delta _{3}$};
\draw (388,345) node   {$\sigma $};
\draw (446,447.5) node   {$\mathcal{P} +\Phi _{1}$};
\draw (551,447.5) node   {$\mathcal{P} +\Phi _{2}$};
\draw (493,430) node   {$\sigma $};
\draw (325,315) node   {$t_{\Phi _{2}}$};
\draw (427,404) node   {$t_{\Phi _{2}}$};
\draw (532,489) node   {$t_{\Phi _{2}}$};
\draw (544,530) node   {$\mathcal{P}$};
\draw (598,515) node   {$\sigma $};
\draw (639,530) node   {$\mathcal{P}$};
\draw (332.5,20) node   {$\mathcal{P}$};
\draw (338,107.5) node   {$\mathcal{P} +\Delta _{2}$};
\draw (443.5,107.5) node   {$\mathcal{P} +\Delta _{3}$};
\draw (386.5,90) node   {$\sigma $};
\draw (443.5,192.5) node   {$\mathcal{P} +\Delta _{1}$};
\draw (548.5,192.5) node   {$\mathcal{P} +\Delta _{2}$};
\draw (491.5,175) node   {$\sigma $};
\draw (320.5,60) node   {$t_{\Delta _{2}}$};
\draw (425.5,149) node   {$t_{\Delta _{2}}$};
\draw (530.5,234) node   {$t_{\Delta _{2}}$};
\draw (542.5,275) node   {$\mathcal{P}$};
\draw (642.5,275) node   {$\mathcal{P}$};
\draw (596.5,260) node   {$\sigma $};
\draw (89.5,45) node   {$x_{\Phi }$};
\draw (201,137) node   {$x_{\Phi }$};
\draw (310,222) node   {$x_{\Phi }$};
\draw (405,303) node   {$x_{\Phi }$};
\draw (515,387) node   {$x_{\Phi }$};
\draw (615,472) node   {$x_{\Phi }$};
\draw (399.5,47) node   {$x_{\Delta }$};
\draw (509.5,132) node   {$x_{\Delta }$};
\draw (609.5,218) node   {$x_{\Delta }$};
\draw (26,277.5) node   {$\mathcal{P} +\Delta _{1}$};
\draw (29,362.5) node   {$\mathcal{P} +\Delta _{3}$};
\draw (134.5,362.5) node   {$\mathcal{P} +\Delta _{1}$};
\draw (77.5,345) node   {$\sigma $};
\draw (11.5,315) node   {$t_{\Delta _{2}}$};
\draw (90.5,302) node   {$x_{\Delta }$};
\draw (439,617.5) node   {$\mathcal{P} +\Delta _{1}$};
\draw (444.5,702.5) node   {$\mathcal{P} +\Phi _{3}$};
\draw (549.5,702.5) node   {$\mathcal{P} +\Phi _{1}$};
\draw (492.5,685) node   {$\sigma $};
\draw (549.5,787.5) node   {$\mathcal{P} +\Delta _{3}$};
\draw (644,787.5) node   {$\mathcal{P} +\Delta _{1}$};
\draw (597.5,770) node   {$\sigma $};
\draw (426.5,655) node   {$t_{\Phi _{2}}$};
\draw (531.5,744) node   {$t_{\Phi _{2}}$};
\draw (27,532.5) node   {$\mathcal{P} +\Phi _{2}$};
\draw (27,615) node   {$\mathcal{P} +\ell $};
\draw (132,615) node   {$\mathcal{P} +\ell $};
\draw (78,600) node   {$\sigma $};
\draw (136,702.5) node   {$\mathcal{P} +\Phi _{2}$};
\draw (241,702.5) node   {$\mathcal{P} +\Phi _{3}$};
\draw (183,685) node   {$\sigma $};
\draw (15,570) node   {$t_{\Delta _{2}}$};
\draw (117,659) node   {$t_{\Delta _{2}}$};
\draw (504,640) node   {$x_{\Phi }$};
\draw (615.5,732) node   {$x_{\Phi }$};
\draw (95,558) node   {$x_{\Delta }$};
\draw (205,642) node   {$x_{\Delta }$};
\draw (237,447.5) node   {$\mathcal{P} +\Phi _{1}$};
\draw (240,532.5) node   {$\mathcal{P} +\Phi _{3}$};
\draw (345.5,532.5) node   {$\mathcal{P} +\Phi _{1}$};
\draw (288.5,515) node   {$\sigma $};
\draw (222.5,485) node   {$t_{\Delta _{2}}$};
\draw (301.5,472) node   {$x_{\Delta }$};
\draw (239,787.5) node   {$\mathcal{P} +\Phi _{1}$};
\draw (344,787.5) node   {$\mathcal{P} +\Phi _{2}$};
\draw (287,770) node   {$\sigma $};
\draw (221.5,744) node   {$t_{\Delta _{2}}$};
\draw (300.5,728) node   {$x_{\Delta }$};
\end{tiny}

\end{tikzpicture}
 \]
 \caption{Images of the cosets of $\P$ in $\P^\perp$ under the induced action of $x_\Delta=t_{\Delta_2}\sigma$ and $x_\Phi=t_{\Phi_2}\sigma$ as in the proofs of Lemmas~\ref{ptimes32nt} and~\ref{nonlinearpsl34}.}
 \label{order3pic}
\end{figure}
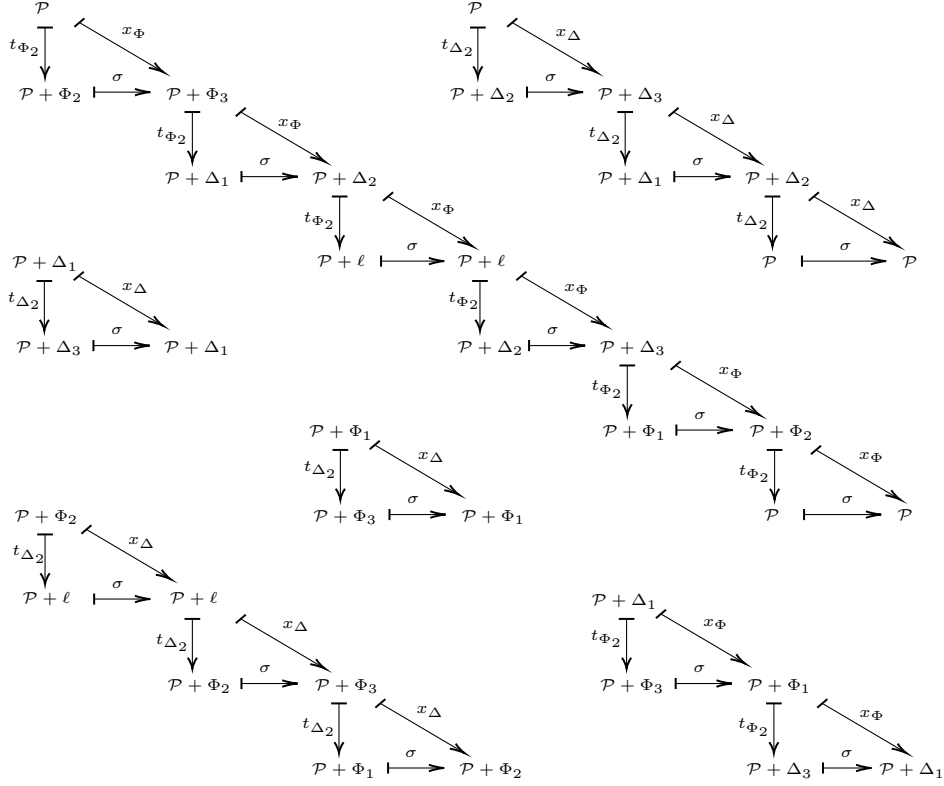

In the remainder of this section we consider $C\neq C_{\text{max}}$, introducing two new non-linear codes. One of these codes is $2$-neighbour-transitive, but not completely transitive, while the other is completely transitive.

\begin{lemma}\label{ptimes32nt}
 Let $C=\langle\P,\Delta_2\rangle\cup \langle\P,\Delta_3\rangle$. Then $C$ has parameters $(21,2^9\cdot 3,6;7)$ and is $2$-neighbour-transitive, but not $3$-neighbour-transitive, having automorphism group $\Aut(C)\cong T_{\P}\rtimes\PGaL_3(4)$ with $\Aut(C)_{\b 0}\cong \PSiL_3(4)$.
\end{lemma}

\begin{proof}
 Note that, by Lemma~\ref{pslsubmodlattdiag} and Figure~\ref{submodulelattice}, the code $\L^\perp=\langle \P,\Delta_1,\Delta_2,\Delta_3\rangle$ is the union of $\P$ and its three cosets $\P+\Delta_1$, $\P+\Delta_2$ and $\P+\Delta_3$. By Lemma~\ref{symdiffhyperovals}, we know that $G=\PGaL_3(4)$, as a subgroup of the top group $L$, induces an action of $\s_3$ on $\{\P+\Delta_1,\P+\Delta_2,\P+\Delta_3\}$ with kernel isomorphic to $\PSL_3(4)$. Hence the stabiliser of $\P+\Delta_1$ inside $G$ is isomorphic to $\PSiL_3(4)$, with the quotient $\PSiL_3(4)/\PSL_3(4)$ acting as $\s_2$ on $\{\P+\Delta_2,\P+\Delta_3\}$. Since the automorphism group of $\L^\perp$ is $T_{\L^\perp}\rtimes \PGaL_3(4)$, we have that $\Aut(C)_{\b 0}\leq\PGaL_3(4)$. Moreover, as there is no group $H$ such that the inclusions $\PSiL_3(4)<H<\PGaL_3(4)$ are proper, we have that $\Aut(C)_{\b 0}\cong \PSiL_3(4)$. Let $x_\Delta=t_{\Delta_2}\sigma\in\Aut(\varGamma)$, where $\sigma$ is an element of $\PGL_3(4)\leq L$ inducing the permutation:
 \[
  \P+\Delta_1\mapsto\P+\Delta_2\mapsto\P+\Delta_3\mapsto\P+\Delta_1.
 \]
 Then, applying Lemma~\ref{symdiffhyperovals} to find the images under $t_{\Delta_2}$, Figure~\ref{order3pic} gives the orbit of $\P$ under $\langle x_\Delta\rangle$, namely $\{\P,\P+\Delta_2,\P+\Delta_3\}$. Since $T_{\P}$ has three orbits on $C$, namely the sets $\P$, $\P+\Delta_2$ and $\P+\Delta_3$, we have that $\Aut(C)$ acts transitively on $C$. Observing that $\Aut(C)_{\b 0}\cong \PSiL_3(4)$ (as above), that the group $T_\P \rtimes \PSiL_3(4)$ acts transitively on $\P$, and that $|C|/|\P|=3$, it follows, from two applications of the orbit-stabiliser theorem, that $\Aut(C) = T_\P \rtimes \langle\PSiL_3(4), x_\Delta \rangle$. Since $\Aut(C)_{\b 0}$ is $2$-transitive on $M$ it follows that $C$ is $2$-neighbour-transitive. Since $\Phi_1\in\mathcal{F}_1$ satisfies $|\Delta\cap\Phi_1|\geq 4$ for all $\Delta\in\H_2\cup\H_3\subseteq C$, and hence $d(\Delta,\Phi_1)\geq 7$, we have that $\Phi_1\in C_7$ and $C$ has covering radius $7$.
 
 Finally, we prove that $\Aut(C)$ is not transitive on $C_3$. Since $\delta=6$ we have that $C_3$ contains all triangles and all sets of three collinear points of $\pg_2(4)$. Suppose there exists an element $x\in \Aut(C)$ where $x$ maps a triangle to a set of three collinear points. Since $x$ maps a weight $3$  vertex to another weight $3$ vertex, the image of ${\b 0}$ under $x$ has weight at most $6$. As $C$ has minimum distance $\delta=6$ we must have that ${\b 0}^x$ has weight $6$. Since $\PSiL_3(4)$ acting on $\pg_2(4)$ preserves triangles and sets of three collinear points, respectively, we then have that ${\b 0}^x$ is the sum of a triangle and a set of three collinear points disjoint from the triangle. There is no such codeword in $C$, because each weight $6$ codeword of $C$ that is a subset of the set of points of $\pg_2(4)$ is a hyperoval, and hence contains no subset of $3$ collinear points. Thus we have a contradiction and $C$ is not $3$-neighbour-transitive.
\end{proof}

\begin{lemma}\label{nonlinearpsl34}
 Suppose $m=21$ and Hypothesis~\ref{hyp2} holds with $\soc(X/K)\cong\PSL_3(4)$. If $C\neq C_{\text{max}}$ then, up to equivalence, $C=\langle\L,\Delta_2\rangle\cup \langle\L,\Delta_3\rangle$ with parameters $(21,2^{10}\cdot 3,5;6)$, which is indeed completely transitive with $\Aut(C)\cong T_{\L}\rtimes\PGaL_3(4)$ and $\Aut(C)_{\b 0}\cong \PSiL_3(4)$. 
\end{lemma}

\begin{proof}
 Suppose $C\neq C_{\text{max}}$ and $C$ is completely transitive. Note that $\PSL_3(4)$ has index $6$ in $\PGaL_3(4)$, so that, by Proposition~\ref{xmaxlemma}, $|C|/|C_{\text{max}}|$ divides $6$. Furthermore, by Lemma~\ref{linearcompletion}, we have that either $|C|=3|C_{\text{max}}|$ or $|C|=6|C_{\text{max}}|$. It follows that $X$ induces an orbit of either size $3$ or $6$ on some set of cosets of $\P$. By Lemma~\ref{symdiffhyperovals}, we know that $\PGaL_3(4)$, as a subgroup of the top group $L$, induces an action of $\s_3$ on the sets $\{\P+\Delta_i\mid i=1,2,3\}$ and $\{\P+\Phi_i\mid i=1,2,3\}$ with kernel $\PSL_3(4)$. Indeed, for each $i$, the stabiliser of $\P+\Delta_i$, which is equal to the stabiliser of $\P+\Phi_i$, is some copy of $\PSiL_3(4)$ that acts transitively on the sets $\{\P+\Delta_j\mid j\in\{1,2,3\};j\neq i\}$ and $\{\P+\Phi_j\mid j\in\{1,2,3\};j\neq i\}$. 
 
 Let $\sigma$ be an element of $\PGL_3(4)\leq L$ inducing the permutation:
 \[
  \P+\Delta_1\mapsto\P+\Delta_2\mapsto\P+\Delta_3\mapsto\P+\Delta_1.
 \] 
 Moreover, let $t_{\Delta_2}$ be the translation by $\Delta_2$, $t_{\Phi_2}$ be the translation by $\Phi_2$, $x_\Delta=t_{\Delta_2}\sigma$ and $x_\Phi= t_{\Phi_2}\sigma$. Figure~\ref{order3pic} gives the images of each of the cosets of $\P$ under $x_\Delta$ and $x_\Phi$. By the above considerations on $|C|/|C_{\text{max}}|$ we see that the relevant orbits are the two orbits of length $3$ under $\langle x_\Delta\rangle$ and the orbit of length $6$ under $\langle x_\Phi\rangle$. Individually, the length $3$ orbits correspond to the code $\langle\P,\Delta_2\rangle\cup \langle\P,\Delta_3\rangle$ and its translate by $\ell$. These codes are not completely transitive by Lemma~\ref{ptimes32nt}. However, their union is the same code given by the orbit of length $6$, namely $\langle\L,\Delta_2\rangle\cup \langle\L,\Delta_3\rangle$. Note that the actions of $\langle x_\Delta\rangle$ and $\langle x_\Phi\rangle$ cover, up to equivalence, all possibilities. Hence it remains to show that $\langle\L,\Delta_2\rangle\cup \langle\L,\Delta_3\rangle$ is completely transitive.

 Let $C=\langle\L,\Delta_2\rangle\cup \langle\L,\Delta_3\rangle$. Then $\Aut(C)=T_\L\rtimes \langle \PSiL_3(4), x_\Delta\rangle$ and, since $T_\L$ acts transitively on $\L$ and $\langle x_\Delta\rangle$ acts transitively on $\{\L,\L+\Delta_2,\L+\Delta_3\}$, it follows that $\Aut(C)$ acts transitively on $C$. Now $\langle\L,\Delta_2\rangle$ has covering radius $6$, and hence $C$ has covering radius $\rho\leq 6$. For each hyperoval $\Delta\in\H_2\cup\H_3$ we have that $d(\Delta,\Delta_1)\geq 6$, and hence $\rho=6$. Since $C$ contains all lines of $\pg_2(4)$, we have that $\varGamma_3({\b 0})\cap C_3$ is the set of triangles, $\varGamma_4({\b 0})\cap C_4$ is one $\PSL_3(4)$-orbit of quadrangles, $\varGamma_5({\b 0})\cap C_5$ is one $\PSL_3(4)$-orbit of ovals, and $\varGamma_6({\b 0})\cap C_6$ is one $\PSL_3(4)$-orbit of hyperovals of $\pg_2(4)$. It follows that $\Aut(C)_{{\b 0}}$ acts transitively on $\varGamma_i({\b 0})\cap C_i$ for each $i=1,\ldots,6$. Thus $C$ is completely transitive. 
\end{proof}

\section{Higman--Sims and Conway Groups}\label{sect:HSCo3}

Again, assume that Hypothesis~\ref{hyp2} holds. We note that both of the groups considered here are equal to their automorphism group. Proposition~\ref{xmaxlemma} then implies that $C=C_{\text{max}}$.

First, consider the case that $\soc(X/K)$ is the Higman--Sims group $\HS$. From the proof of \cite[Theorem~8.1]{Ivanov1993} we have that the $\F_2$-permutation module for the $2$-transitive action of $\HS$ is uniserial with structure $1\backslash 20\backslash 1\backslash 132\backslash 1\backslash 20\backslash 1$. Codes invariant under this action are constructed in \cite[Section~5]{moori2015higmansims}, and verified by our own GAP \cite{GAP4} computations. The codes constructed are a $[176,21,56]$ code, a $[176,22,50]$ code, a $[176,154,6]$ code and a $[176,155,6]$ code, each invariant under $\HS$. Since the permutation module is uniserial, these codes, along with the repetition code and its dual, give all proper submodules invariant under $\HS$. By Lemma~\ref{upboundmindist}, it follows that $C$ is either the $[176,154,6]$ code or the $[176,155,6]$ code. By \cite[Theorem~2 (ii)]{borges2019completely}, the weight $6$ codewords of $C$ form a $3$-$(176,6,\lambda)$ design, for some integer $\lambda$, and hence $|\varGamma_6({\b 0})\cap C|$ is divisible by
\[
 \frac{{176\choose 3}}{{6\choose 3}}=44660.
\]
However, the number of weight $6$ codewords of the $[176,154,6]$ code and the $[176,155,6]$ code are $92400$ and $129360$ respectively. Thus, these codes are not completely regular and hence also not completely transitive.

Turning to consider $\soc(X/K)$ isomorphic to the Conway group $\Co_3$, the proof of \cite[Theorem~8.1]{Ivanov1993} says that the $\F_2$-permutation module for the $2$-transitive action of $\Co_3$ is uniserial with structure $1\backslash 22\backslash 230\backslash 22\backslash 1$. A $[276,23,100]$ code invariant under $\Co_3$ is given in \cite{Haemers1993} with its dual, by computations in GAP \cite{GAP4}, being a $[276,253,6]$ code. Since the permutation module is uniserial, these codes, along with the repetition code and its dual, give all proper submodules invariant under $\Co_3$. By Lemma~\ref{upboundmindist}, it follows that $C$ is the $[276,253,6]$ code. By \cite[Theorem~2 (ii)]{borges2019completely}, the weight $6$ codewords of $C$ form a $3$-$(276,6,\lambda)$ design, for some integer $\lambda$, and hence $|\varGamma_6({\b 0})\cap C|$ is divisible by
\[
 \frac{{276\choose 3}}{{6\choose 3}}=173305.
\]
However, the number of weight $6$ codewords of the $[276,253,6]$ code is $708400$. Thus, this code is not completely regular and hence also not completely transitive. 

In summary:

\begin{lemma}\label{noHSCo}
 There are no codes satisfying Hypothesis~\ref{hyp2} with $\soc(X/K)=\HS$ or $\Co_3$.
\end{lemma}

\section{\texorpdfstring{$\PSL_2(r)$}{PSL(2,r)}}\label{sect:psl2r}

Suppose that Hypothesis~\ref{hyp2} holds, and that $\soc(X/K)\cong \PSL_2(r)$ where $m=r+1\geq 24$ with $r=p^t$ for some prime $p$ and integer $t$. It follows from \cite[Theorem~1]{kantor1972k} and Lemma~\ref{upboundmindist} that either:
\begin{enumerate}
  \item $r\equiv 3 \pmod{4}$, $\delta\leq 8$ and $X_{\b 0}^M$ is $3$-homogeneous, but not $4$-homogeneous; or,
  \item $r\equiv 1 \pmod{4}$, $\delta\leq 6$ and $X_{\b 0}^M$ is $2$-homogeneous, but not $3$-homogeneous.
\end{enumerate}

\begin{lemma}\label{psl2rtlessthan5}
 Suppose $m=r+1$, where $r=p^t$ with $p$ prime, and Hypothesis~\ref{hyp2} holds with $\soc(X/K)\cong\PSL_2(r)$. Then $C_{\text{max}}$ has dimension $\frac{r+1}{2}$ and $t\leq 4$.
\end{lemma}

\begin{proof}
 By \cite[Lemma~5.4]{Ivanov1993}, we have that $C_{\text{max}}$ is one of two equivalent codes $U_1$ and $U_2$, where $U_1$ and $U_2$ are interchanged under the action $\PGL_2(r)$. Furthermore, $U_1$ and $U_2$ are self-dual when $r\equiv 3 \pmod{4}$, and $U_1^\perp=U_2$ when $r\equiv 1 \pmod{4}$. Either way, we have that $U_1$ and $U_2$ have the same dimensions and minimum distances. Hence $C_{\text{max}}$ has dimension $\frac{r+1}{2}$.
 
 Suppose $t\geq 5$. Now \cite[Lemma~4.4]{minimal2nt} gives the bound $\delta_{\text{max}}\geq \sqrt{r}+1$. Since $r$ is odd, and hence $p\geq 3$, we have that $r\geq 243$ and $\delta_{\text{max}}\geq 17$.  We claim that
 \[
  |C|/|C_{\text{max}}|< \frac{m(m-1)}{\delta(\delta-1)},
 \]
 in which case we may apply Lemma~\ref{largedeltamax}, which implies that $\delta_{\text{max}}\leq 2\delta$. Since $\delta\leq 8$, but $\delta_{\text{max}}\geq 17$, we then have a contradiction, and hence $t\leq 4$. It remains to prove the claim.  
 
 First, $\PSL_2(r)$ has index $2t$ in $\PGaL_2(r)$. Thus, by Proposition~\ref{xmaxlemma}, $|C|/|C_{\text{max}}|=|X^M:X_{\b 0}^M|$ is at most $2t$. Now, for all $t\geq 5$ the inequality $2t< 3^t(3^t+1)/56$ holds. Using this, and since $\delta\leq 8$ and $m-1=r=p^t\geq 3^t$, we obtain:
 \[
  \frac{|C|}{|C_{\text{max}}|}\leq 2t<\frac{3^t(3^t+1)}{7\cdot 8}\leq \frac{m(m-1)}{\delta(\delta-1)},
 \] 
 thus proving the claim.
\end{proof}

\begin{lemma}\label{psl2rrlessthan43}
 Suppose $m=r+1$ and Hypothesis~\ref{hyp2} holds with $\soc(X/K)\cong\PSL_2(r)$. Then $r\in\{23,25,31,41\}$.
\end{lemma}

\begin{proof}
 By Lemma~\ref{psl2rtlessthan5}, we have that the dimension of $C_{\text{max}}$ is $\frac{r+1}{2}$. Proposition~\ref{xmaxlemma} part 3 implies that $X/T_{C_{\text{max}}}\cong X^M$ and hence $|X|=|T_{C_{\text{max}}}||X^M|$. Since $X^M\leq \PGaL_2(r)$, we have that
 \[
  |X|=|T_{C_{\text{max}}}||X^M|\leq 2^{(r+1)/2}r(r+1)(r-1)t.
 \]
 It then follows from Lemma~\ref{morbitsbound} that
 \begin{equation}\label{pslinequality}
  2^{r+1}\leq (m+1)|X|\leq (r+2)2^{(r+1)/2}r(r+1)(r-1)t.
 \end{equation}
 Moreover, Lemma~\ref{psl2rtlessthan5} implies that $t\leq 4$, so that the above becomes
 \[
  2^{(r+1)/2}\leq 4r(r+2)(r+1)(r-1).
 \]
 This holds for $r=23$, $25$, $31$, $41$ and $47$. Note that $t=1$ when $r=47$, in which case the inequality~(\ref{pslinequality}) becomes $2^{(r+1)/2}\leq (r+2)r(r+1)(r-1)$, which no longer holds.
\end{proof}

\begin{lemma}\label{psl2rcequalscmax}
 Suppose $m=r+1$ and Hypothesis~\ref{hyp2} holds with $\soc(X/K)\cong\PSL_2(r)$. Then $C=C_{\text{max}}$.
\end{lemma}

\begin{proof}
 Suppose $C\neq C_{\text{max}}$. By Lemma~\ref{psl2rrlessthan43}, we have that $t\leq 2$. Since $\soc(X_{\b 0}/K)=\soc(X/K)\cong\PSL_2(r)$, we have that $\PSL_2(r)\leq X_{\b 0}^M$ and $X^M\leq \PGaL_2(r)$. Therefore $|X^M:X_{\b 0}^M|$ divides $4$ and hence, by Proposition~\ref{xmaxlemma} part 5, we have that $|C|/|C_{\text{max}}|$ divides $4$. Lemma~\ref{linearcompletion} part 3 then implies that $|C|/|C_{\text{max}}|\neq 2$ and hence that $|C|/|C_{\text{max}}|=4$. Applying Lemma~\ref{linearcompletion} part 3 once more, we have that $C_{\text{max}}$ has codimension $2$ or $3$ in the code $\langle C\rangle$ spanned by the codewords of $C$, which must be also be completely transitive. Hence $\langle C\rangle$ has dimension $(r+5)/2$ or $(r+7)/2$. By \cite[Lemma~5.4]{Ivanov1993}, the only submodules of dimension greater than $(r+1)/2$ that are invariant under $\PSL_2(r)$ are $\langle {\b 1}\rangle^\perp$ and $V\varGamma$. This implies that $(r+5)/2$ or $(r+7)/2$ is equal to $r$ or $r+1$, in which case $r=3$, $5$ or $7$. However this contradicts Lemma~\ref{psl2rrlessthan43}.
\end{proof}

\begin{lemma}\label{nopsl2r}
 There are no completely transitive codes satisfying Hypothesis~\ref{hyp2} with $\soc(X/K)\cong\PSL_2(r)$.
\end{lemma}

\begin{proof}
 By Lemma~\ref{psl2rrlessthan43}, we have that $r\in\{23,25,31,41\}$, by Lemma~\ref{psl2rcequalscmax}, we have that $C=C_{\text{max}}$ and, by \cite[Lemma~5.4]{Ivanov1993}, there are only two possibilities for $C_{\text{max}}$, which are in fact equivalent codes. First, by \cite[p.~482]{macwilliams1978theory}, if $r=23$ then $C$ is the extended Golay code $\G_{24}$, in which case $X/K\cong\mg_{24}\neq\PSL_2(23)$.
 
 Suppose $r=25$. Then a calculation in GAP involving the package GUAVA shows that both of the two possibilities for $C_{\text{max}}$ here have external distance $9$ but covering radius $5$, and thus, by \cite[Theorem~4.1]{sole1990completely}, are not completely regular.
 
 Let $r=31$. The two codes that arise here have minimum distance $8$, and hence, by \cite[Theorem~2 (ii)]{borges2019completely}, the weight $8$ codewords form a $4$-design. This implies that the number of weight $8$ codewords is a multiple of $3596$. However, there are only $620$ such codewords in each code.
 
 Finally, let $r=41$. Then a calculation in GAP using the GUAVA package shows that $C_{\text{max}}$ has minimum distance $10$, and is thus not completely transitive.
\end{proof}

\section{\texorpdfstring{$\PSU_3(r)$}{PSU(3,r)}}\label{sect:PSU3r}

Suppose Hypothesis~\ref{hyp2} holds and $\soc(X/K)\cong\PSU_3(r)$ with $r$ an odd prime power.

\begin{lemma}\label{psuris3}
 Suppose $r=p^t$ is an odd prime power, $m=r^3+1$ and Hypothesis~\ref{hyp2} holds with $\soc(X/K)\cong\PSU_3(r)$. Then $r=3$ or $5$.
\end{lemma}

\begin{proof}
 By the remark following Theorem~7.3 in \cite{Ivanov1993}, any code invariant under $X_{\b 0}$ that has dimension at least $2$ must have dimension at least $r^2-r+1$. By duality, it follows that the dimension of $C_{\text{max}}$ is at most
 \[
  m-(r^2-r+1)=r^3-r^2+r.
 \]
 Now $X/K\leq \PGaU_3(r)$. Hence, applying Lemma~\ref{morbitsbound},
 \[
  2^{r^3+1}\leq (m+1)|X|\leq (r^3+2)2^{r^3-r^2+r}r^3(r^2-1)(r^3+1)t.
 \]
 Thus, since $t\leq r$ holds, $r$ satisfies
 \[
  2^{r^2-r+1}\leq (r^3+2)r^4(r^2-1)(r^3+1).
 \]
 This holds for an odd prime power $r$ only when $r=3$ or $5$.
\end{proof}

\begin{lemma}\label{noPSU}
 There are no codes satisfying Hypothesis~\ref{hyp2} with $\soc(X/K)\cong \PSU_3(r)$.
\end{lemma}

\begin{proof}
 By \cite{kantor1972k}, $X_{{\b 0}}^M$ is $2$-transitive but not $3$-homogeneous, and therefore Lemma~\ref{upboundmindist} implies that $\delta=5$ or $6$. By Lemma~\ref{psuris3}, we have that $r=3$ or $5$. We first consider the case $r=3$. Suppose $\delta=5$. Then, by Lemma~\ref{lambdabound}, the weight $5$ codewords form a $2$-$(28,5,\lambda)$ design, for some integer $\lambda\leq (m-2)/3$. That is, $\lambda\leq 8$. Now, the number of blocks of such a design is:
 \[
  \frac{28\cdot 27}{5\cdot 4}\lambda.
 \]
 Since this number must be an integer, we have that $5$ divides $\lambda$, and hence $\lambda=5$. Thus, there must be $189$ weight $5$ codewords. However, an analysis in GAP \cite{GAP4} shows that the smallest orbit on $5$-sets for $\PSU_3(3)$ in its $2$-transitive action has size $1512$.
 
 Suppose $\delta=6$. Then, by Lemma~\ref{lambdabound}, the weight $6$ codewords form a $3$-$(28,6,\lambda)$ design, for some integer $\lambda\leq (m-3)/3$. That is, $\lambda\leq 8$. Now, the number of blocks of such a design is:
 \[
  \frac{28\cdot 27\cdot 26}{6\cdot 5\cdot 4}\lambda.
 \]
 Since this number must be an integer, we have that $5$ divides $\lambda$, and hence $\lambda=5$. Thus, there must be $819$ weight $6$ codewords. However, an analysis in GAP \cite{GAP4} shows that, for $\PSU_3(3)$ in its $2$-transitive action, the orbits on $6$-subsets having length at most $819$ either have size $756$ or $504$. 
 
 Let $r=5$ and $k$ be the dimension of $C_{\text{max}}$. It follows from \cite[Theorem~7.2]{Ivanov1993} $C_{\text{max}}$ is either contained in the code $U$ generated by the unital (a $2$-$(126,6,1)$ design; see \cite[Section~7.7]{dixon1996permutation} for more information), which has dimension $105$, or its dual $U^\perp$, which has dimension $21$. By Lemma~\ref{morbitsbound}, we have that
 \[
  (m+1)2^k |\PGU_3(5)|\geq 2^m.
 \]
 This implies that $k\geq 104$, so that $k=104$ or $105$, and hence that either $C_{\text{max}}=U$ or $C_{\text{max}}$ has codimension $1$ in $U$. By \cite[Theorem~7.2]{Ivanov1993} there is no such code of codimension $1$ in $U$, so $C_{\text{max}}=U$. Since $\PSU_3(5)$ has index $3$ in $\PGU_3(5)$, part 5 of Proposition~\ref{xmaxlemma} and Lemma~\ref{linearcompletion} imply that $C=C_{\text{max}}$ or $C_{\text{max}}$ has codimension $2$ in $\langle C\rangle$. By \cite[Theorem~7.2]{Ivanov1993}, there is no linear code invariant under $X_{\b 0}$ having dimension $107$, and hence $C=C_{\text{max}}=U$. Constructing $U$ in the GAP package GUAVA, we see that it has minimum distance $6$ and $21525$ codewords of weight $6$. By Lemma~\ref{lambdabound}, the weight $6$ codewords form a $3$-$(126,6,\lambda)$ design for some integer $\lambda$. Now, the number of blocks of such a design is:
 \[
  \frac{126\cdot 125\cdot 124}{6\cdot 5\cdot 4}\lambda=21\cdot 25\cdot 31\lambda.
 \]
 This gives a contradiction, since $21525$ is not divisible by $31$.
\end{proof}

\section{\texorpdfstring{$\Ree(r)$}{Ree(r)}}\label{sect:Reer}

Suppose Hypothesis~\ref{hyp2} holds and that $\soc(X/K)\cong\Ree(r)$ with $r=3^{2t+1}$ and $m=r^3+1$.

\begin{lemma}\label{reeris3}
 Suppose $r$ is an odd power of $3$, that $m=r^3+1$ and Hypothesis~\ref{hyp2} holds with $\soc(X/K)\cong\Ree(r)$. Then $r=3$.
\end{lemma}

\begin{proof}
 By \cite[Lemma~7.4]{Ivanov1993}, any code invariant under $X_{\b 0}$ that has dimension at least $2$ must have dimension at least $r^2-r+1$. By duality, it follows that the dimension of $C_{\text{max}}$ is at most
 \[
  m-(r^2-r+1)=r^3-r^2+r.
 \]
 Since $X/K\leq \Ree(r)$, we have that $|X/K|$ divides $r^3(r^3+1)(r-1)(2t+1)$. Hence, applying Lemma~\ref{morbitsbound}, we have,
 \[
  2^{r^3+1}\leq (m+1)|X|\leq (r^3+2)2^{r^3-r^2+r}r^3(r^3+1)(r-1)(2t+1).
 \]
 Thus, since $(2t+1)\leq r$ holds, $r$ satisfies
 \[
  2^{r^2-r+1}\leq (r^3+2)r^4(r^3+1)(r-1).
 \]
 This holds for $r=3^{2t+1}$ only when $r= 3$.
\end{proof}

\begin{lemma}\label{noRee}
 There are no codes satisfying Hypothesis~\ref{hyp2} with $\soc(X/K)\cong \Ree(r)$.
\end{lemma}

\begin{proof}
 By \cite{kantor1972k}, $X_{{\b 0}}^M$ is $2$-transitive but not $3$-homogeneous. Thus, Lemma~\ref{upboundmindist} implies that $\delta=5$ or $6$. By Lemma~\ref{reeris3}, we have that the only possibility for is $r=3$. Suppose $\delta=5$. Then, by Lemma~\ref{lambdabound}, the weight $5$ codewords form a $2$-$(28,5,\lambda)$ design, for some integer $\lambda\leq (m-2)/3$. That is, $\lambda\leq 8$. Now, the number of blocks of such a design is:
 \[
  \frac{28\cdot 27}{5\cdot 4}\lambda.
 \]
 Since this number must be an integer, we have that $5$ divides $\lambda$, and hence $\lambda=5$. Thus, there must be $189$ weight $5$ codewords. However, an analysis in GAP \cite{GAP4} shows that the smallest orbit on $5$-sets for $\Ree(3)$ in its $2$-transitive action has size $756$.
 
 Suppose $\delta=6$. Then, by Lemma~\ref{lambdabound}, the weight $6$ codewords form a $3$-$(28,6,\lambda)$ design, for some integer $\lambda\leq (m-3)/3$. That is, $\lambda\leq 8$. Now, the number of blocks of such a design is:
 \[
  \frac{28\cdot 27\cdot 26}{6\cdot 5\cdot 4}\lambda.
 \]
 Since this number must be an integer, we have that $5$ divides $\lambda$, and hence $\lambda=5$. Thus, there must be $819$ weight $6$ codewords. However, an analysis in GAP \cite{GAP4} shows that, for $\Ree(3)$ in its $2$-transitive action, each orbit on $6$-subsets has even length. This completes the proof.
\end{proof}

\section{Proof of main result}\label{sect:mainproof}

We now prove Theorem~\ref{binaryCTclass}. Suppose that $C$ is a non-trivial completely transitive code in $H(m,2)$ with minimum distance $\delta\geq 5$. Let $C_{\text{max}}$ be the maximal linear subcode of $C$, let $X=\Aut(C)$ and let $X_{\text{max}}=X_{C_{\text{max}}}$. It then follows from Proposition~\ref{xmaxlemma} part 3 that $K=X\cap B$ is equal to $X_{\text{max}}\cap B=T_{C_{\text{max}}}$. Furthermore, suppose that $\soc(X_{\text{max}}/K)\cong G$, where $G$ is one of the groups appearing in Table~\ref{binarytable}. Note that $\delta\geq 5$ implies $C_2\neq\emptyset$ from which it follows that $C$ is $2$-neighbour-transitive and that Theorem~\ref{binaryx2ntchar} applies.

Suppose that $C_{\text{max}}$ is contained in the binary repetition code $\langle {\b 1}\rangle$. It follows that $C$ comes under part 1 or 2 of Theorem~\ref{binaryx2ntchar}. Now, by \cite[Lemma~4.2]{ef2nt}, the even weight subcode of the punctured Hadamard code is not completely transitive. Moreover, by \cite[Corollary~5]{ondimblock}, both the Hadamard code of length $12$ and its punctured code are completely transitive. This gives lines 1 and 2 of Table~\ref{binaryCTtable}. Hence, we may assume that $C_{\text{max}}$ has dimension at least $2$ for the remainder of the proof, which implies that $C$ is as in part 3 of Theorem~\ref{binaryx2ntchar}. In particular, this means that $\mg_{11}$ and $\mg_{12}$ do not need to be considered in the remainder of the proof as they do not appear in \cite[Table~1]{minimal2nt}.

Now, Theorem~\ref{binaryx2ntchar} implies that the repetition code in $H(m,2)$ is the unique linear code of dimension $1$ stabilised by $X_{{\b 0}}$. It follows that the dual $\langle{\b 1}\rangle^\perp$ of the binary repetition code is the only linear code of codimension $1$ stabilised by $X_{{\b 0}}$. Since $\langle{\b 1}\rangle^\perp$ has minimum distance $2$ it follows that $C\neq \langle{\b 1}\rangle^\perp$. Thus $2\leq\dim(C_{\text{max}})\leq m-2$ and Hypothesis~\ref{hyp1} holds. Suppose that $\soc(X/K)\neq\soc(X_{\text{max}}/K)$. Then, by Propositions~\ref{differentsocles} and \ref{a7fifteen}, we have that $m=15$, $\soc(X_{\text{max}}/K)\cong\alt_7$ and $C=\NR$, as in line 3 of Table~\ref{binaryCTtable}. Moreover, by \cite[Theorem~1.1]{gillespie2012nord}, $\NR$ is completely transitive. Note that Proposition~\ref{a7fifteen} means that we do not need to consider $\soc(X_{\text{max}}/K)=\alt_7$ in the remainder of the proof.

We now have that $\soc(X/K)=\soc(X_{\text{max}}/K)$ and so Hypothesis~\ref{hyp2} holds. Suppose that $\soc(X/K)=\PSL_3(4)$. Then, by Lemmas~\ref{psl34CT} and \ref{nonlinearpsl34}, we have that $C$ is one of the codes appearing on lines 4--7 of Table~\ref{binaryCTtable}, each of which is completely transitive. Let $\soc(X/K)=\mg_{m}$ with $m=22,23$ or $24$. Then, by Lemmas~\ref{halfg23ct} and \ref{math22CT}, it follows that one of the lines 8--13 holds, and that the codes appearing there are indeed completely transitive. By Lemmas~\ref{noHSCo}, \ref{nopsl2r}, \ref{noPSU} and \ref{noRee}, no codes satisfying Hypothesis~\ref{hyp2} exist with $\soc(X/K)=\HS$, $\Co_3$, $\PSL_2(r)$, $\PSU_3(r)$ or $\Ree(r)$. This completes the proof.

\end{document}